\documentclass[11pt]{article}

\usepackage{url}
\usepackage{hyperref}
\usepackage{amssymb}
\usepackage{amsmath}
\usepackage{theorem}
\usepackage{epsfig}
\usepackage{verbatim}
\usepackage{graphicx}
\usepackage{subfigure}
\usepackage{cancel}
\usepackage{epstopdf}
\usepackage{a4wide}
\usepackage{pdflscape}

\usepackage{color}

\DeclareMathOperator\arcsinh{arcsinh}

\newtheorem{theorem}{Theorem}[section]
\newtheorem{lemma}[theorem]{Lemma}

\newtheorem{prop}[theorem]{Proposition}
\newtheorem{corollary}[theorem]{Corollary}

\newtheorem{rem}[theorem]{Remark}

\newcommand{\RR}{\mathbb{R}}

\newcommand{\N}{\mathbb{N}}
\newcommand{\T}{\mathbb{T}}

\newcommand{\supp}{{\rm supp}\thinspace}

\newcommand{\al}{\alpha}

\newcommand{\ep}{\varepsilon}

\newcommand{\trho}{\tilde{\rho}}

\newcommand{\pa}{\partial}

\newcommand{\inti}{\int_{-\infty}^\infty}
\newcommand{\intpi}{\int_{-\pi}^\pi}

\newenvironment{proof}{\begin{trivlist} \item[] {\em Proof:}}{\hfill $\Box$
                       \end{trivlist}}

\newenvironment{proofsth}[1]{\begin{trivlist} \item[] {\textbf{Proof of #1:}}}{\hfill $\Box$
                       \end{trivlist}}

\makeatletter
\renewcommand*\l@section{\@dottedtocline{1}{0em}{1.5em}}
\renewcommand*\l@subsection{\@dottedtocline{2}{1.5em}{2.3em}}
\renewcommand*\l@subsubsection{\@dottedtocline{3}{3.8em}{3.7em}}
\makeatother

\numberwithin{equation}{section}

\allowdisplaybreaks[4]

\begin{document}

\title{Global smooth solutions for the inviscid SQG equation}

\author{Angel Castro, Diego C\'ordoba and Javier G\'omez-Serrano}

\maketitle

\begin{abstract}

In this paper, we show the existence of the first non trivial family of classical global solutions of the inviscid surface quasi-geostrophic equation.

\vskip 0.3cm


\end{abstract}

\tableofcontents

\section{Introduction}


We consider the initial value problem for the inviscid surface quasi-geostrophic equation (SQG):

\begin{align}
\partial_t \theta (x,t)+ u(x,t) \cdot \nabla \theta(x,t)& = 0, \ \ (x,t) \in \mathbb{R}^2 \times \mathbb{R}_+ \label{sqg} \\
 u(x,t) & = (-R_2(\theta), R_1(\theta))(x,t)\nonumber\\
\theta(x,0)&= \theta_0(x)\nonumber,
\end{align}

where $R_j$ is the $j$-th Riesz transform:

\begin{align*}
 R_j(\theta)(x) = \frac{1}{2\pi}P.V.\int_{\mathbb{R}^{2}} \frac{(x_j - y_j)}{|x-y|^{3}}\theta(y) dy.
\end{align*}

This equation is derived considering small Rossby and Ekman numbers and constant potential vorticity. It models the evolution of the temperature from a general quasi-geostrophic system for atmospheric and oceanic flows (see  \cite{Constantin-Majda-Tabak:formation-fronts-qg,Held-Pierrehumbert-Garner-Swanson:sqg-dynamics,Pedlosky:geophysical,Majda-Bertozzi:vorticity-incompressible-flow} for more details). The numerical and analytical study of the equation was started by Constantin, Majda and Tabak in \cite{Constantin-Majda-Tabak:formation-fronts-qg}, since the SQG system presents an analogy with the 3D Euler equations.

The aim of this paper is to address the main problem of whether its classical solution corresponding to given initial data $\theta(x,0) = \theta_0(x)$ with finite energy exists for all time or not. We remark that both the $L^{p}$ norms of theta $(1 \leq p \leq \infty)$ and the $L^{2}$ norm of $u$ (the energy of the system) are conserved quantities. Moreover, the $L^{p}$ norms of $u$ obey the following bounds:
$$||u(\cdot,t)||_{L^p}\leq C_p ||\theta_0||_{L^p}\quad 1<p<\infty.$$

Local existence of solutions for \eqref{sqg} was first shown in \cite{Constantin-Majda-Tabak:formation-fronts-qg} in Sobolev spaces. By using different functional frameworks local existence has been also addressed in several papers, see for example \cite{Chae:qg-equation-triebel-lizorkin,Li:existence-theorems-2d-sqg-plane-waves,Wu:qg-equations-morrey-spaces,Wu:solutions-2d-qg-holder}.

Resnick, in his thesis \cite{Resnick:phd-thesis-sqg-chicago}, showed global existence of weak solutions in $L^{2}$ using an extra cancellation due to the oddness of the Riesz transform. Marchand \cite{Marchand:existence-regularity-weak-solutions-sqg} extended Resnick's result to the class of initial data belonging to 
$L^{p}$ with $p > 4/3$. The question of non-uniqueness for weak solutions in $L^2$ is still a challenging open problem (see \cite{Isett-Vicol:holder-continuous-active-scalar,Azzam-Bedrossian:bmo-uniqueness-active-scalar-equations,Rusin:logarithmic-spikes-uniqueness-weak-solution-active-scalars} and references therein).

The problem of whether the SQG system presents finite time singularities or there is global existence is open for the smooth case. Kiselev and Nazarov \cite{Kiselev-Nazarov:simple-energy-pump-sqg} constructed solutions that started arbitrarily small but grew arbitrarily big in finite time, and Friedlander and Shvydkoy \cite{Friedlander-Shvydkoy:unstable-spectrum-sqg} showed the existence of unstable eigenvalues of the spectrum. Castro and C\'ordoba constructed singular solutions with infinite energy in \cite{Castro-Cordoba:infinite-energy-sqg} and Dritschel \cite{Dritschel:exact-rotating-solution-sqg} constructed global solutions that have $C^{1/2}$ regularity.


The numerical simulations in \cite{Constantin-Majda-Tabak:formation-fronts-qg} indicated a possible singularity in the form of a hyperbolic saddle closing in finite time. Ohkitani and Yamada \cite{Ohkitani-Yamada:inviscid-limit-sqg} and Constantin et al \cite{Constantin-Nie-Schorghofer:nonsingular-sqg-flow} suggested that the growth was double exponential. The question was settled by Cordoba \cite{Cordoba:nonexistence-hyperbolic-blowup-qg} who bounded the growth by a quadruple exponential and further improved by Cordoba and Fefferman \cite{Cordoba-Fefferman:growth-solutions-qg-2d-euler} to a double exponential (see also \cite{Deng-Hou-Li-Yu:non-blowup-2d-sqg}). The same scenario was recomputed almost 20 years with bigger computational power and improved algorithms by Constantin et al. \cite{Constantin-Lai-Sharma-Tseng-Wu:new-numerics-sqg}, yielding no evidence of blowup and the depletion of the hyperbolic saddle past the previously computed times. In \cite{Majda-Tabak:2d-model-sqg}, Majda and Tabak compared simulations for the SQG scenario with the Euler case. Scott, in \cite{Scott:scenario-singularity-quasigeostrophic}, starting from elliptical configurations, proposed a candidate that develops filamentation and after a few cascades, blowup of $\nabla \theta$.

Several criteria of blowup have been found and blowup can only occur through the blowup of either some geometric quantities or certain space-time norms. For more details see \cite{Chae-Constantin-Wu:deformation-symmetry-inviscid-sqg-3d-euler, Chae:continuation-principles-euler-sqg, Chae:geometric-approaches-singularities-inviscid-fluid-flows, Chae:behavior-solution-euler-related, Cordoba-Fefferman:scalars-convected-2d-incompressible-flow, Hou-Shi:dynamic-growth-vorticity-3d-euler-sqg, Ju:geometric-constraints-global-regularity-sqg,Cannone-Xue:self-similar-solutions-sqg}.

A different approach for the study of the formation of singularities for SQG comes from the patch-problem, i.e., ``sharp fronts'' . In this problem one considers that the scalar $\theta(x,t)$ is the characteristic function of some compact and simple connected domain which depends on time and with smooth boundary. Local existence for the patch was proven by Rodrigo \cite{Rodrigo:evolution-sharp-fronts-qg} for a $C^\infty$ boundary and by Gancedo \cite{Gancedo:existence-alpha-patch-sobolev} in Sobolev spaces. C\'ordoba et al. found,  in  \cite{Cordoba-Fontelos-Mancho-Rodrigo:evidence-singularities-contour-dynamics}, strong numerical evidences of the formation of a singularity in the boundary of the patch. For further numerical simulations addressing formation of singularities see \cite{Scott-Dritschel:self-similar-sqg} and \cite{Scott:scenario-singularity-quasigeostrophic}. The possibility of a splash singularity scenario (i.e. when the interface touches itself on a point but the curve does not lose regularity) was ruled out by Gancedo and Strain \cite{Gancedo-Strain:absence-splash-muskat-SQG}.

Through a different motivation, Cordoba et al. \cite{Cordoba-Fefferman-Rodrigo:almost-sharp-fronts-sqg}, Fefferman and Rodrigo \cite{Fefferman-Rodrigo:almost-sharp-fronts-sqg} and Fefferman et al. \cite{Fefferman-Luli-Rodrigo:spine-sqg-almost-sharp-front} studied the existence of a special type of solutions that are known as ``almost sharp fronts'' for SQG. These solutions can be thought of as a regularization of a front, with a small strip around the front in which the solution changes (reasonably) from one value of the front to the other. These are strong solutions of the equation with large gradient ($\sim \text{(Width of the strip)}^{-1}$).

The main purpose of the paper is to show the following theorem:

\begin{theorem}
\label{globalsqg}
 There is a nontrivial global smooth solution for the SQG equations that has finite energy.
\end{theorem}

It is well known that radial functions are stationary solutions for \eqref{sqg} due to the structure of the nonlinear term. The solutions that will be constructed in this paper are a smooth perturbation in a suitable direction of a specific radial function. The smooth profile we will perturb satisfies (in polar coordinates)
\begin{align*}
\theta(r)\equiv \left\{ \begin{array}{ccc} 1 & \text{for $0\leq r \leq 1-a$}\\ \text{smooth and decreasing} & \quad \text{for $1-a < r< 1$}  \\ 0 & \text{for $1\leq r <\infty$} \end{array}\right.,
\end{align*}
where $a$ is a small number (below we will impose some more constraints in this profile). In addition the dynamics of these solutions consist of global rotating level sets with constant angular velocity. These level sets are a perturbation of the circle. The limit case $a=0$ gives rise to the well known V-state solution for SQG, i.e., a global rotating patch which solves weakly \eqref{sqg}. The existence of V-states with $C^\infty$ boundary for SQG was proven in \cite{Castro-Cordoba-GomezSerrano:existence-regularity-vstates-gsqg}. It was shown in \cite{Castro-Cordoba-GomezSerrano:analytic-vstates-ellipses} that the boundary of these solutions is actually  analytic.

The proofs of these results are motivated by the ones for 2D incompressible Euler in the simply connected case; Burbea in \cite{Burbea:motions-vortex-patches} proved the existence of V-states for Euler and $C^\infty$-regularity for its boundary was proved by Hmidi at al. in \cite{Hmidi-Mateu-Verdera:rotating-vortex-patch}  (see also \cite{Hassainia-Hmidi:v-states-generalized-sqg}).

The paper is organized as follows: section \ref{sectionequations} is devoted to the reformulation of the equations \eqref{sqg}  in new variables.  In section \ref{CR} we state the main theorem and present the Crandall-Rabinowitz (C-R) theorem which will be the main tool in our proof. In section \ref{checking} we check that our equation satisfies the hypotheses of the C-R theorem. This will be the main part of our work.

In particular section \ref{checking3} is different from previous analysis. We stress the following main differences:

\begin{itemize}
\item The study of the linear problem is now reduced to a functional equation, as opposed to a scalar equation (which was in the patch case). Even the existence of nontrivial elements in the kernel of the linear part is not evident a priori.

 \item There is no algebraic formula for neither the eigenvalue nor the eigenvector, not even in an implicit way (such as in \cite{Castro-Cordoba-GomezSerrano:analytic-vstates-ellipses}). This makes the proof of the dimensionality of the kernel much harder since one needs to show that the eigenvalue is simple and have some control of the rest of the eigenvalues.
\end{itemize}

\begin{rem}In a forthcoming paper  \cite{Castro-Cordoba-GomezSerrano:uniformly-rotating-smooth-euler}, by using the same techniques, we are able to extend our construction to the 2D Incompressible Euler equations.
\end{rem}

Finally, in Appendix \ref{sectionasymptotics}, we compute the asymptotics and bounds on the error terms of some of the elliptic integrals that appear and Appendix \ref{sectioncomputerassisted} is devoted to discuss the details of the computer-assisted code and its implementation, as well as to show the rigorous numerical bounds used in the theorems. Appendix \ref{appendixprojections} contains an explicit description of two big matrices used in the proofs.

A major theme of our work is the interplay between rigorous computer calculations and traditional mathematics. We use interval arithmetics as part of a proof whenever they are needed.

Advances in computing power have made rigorous computer-assisted proofs realizable.  Naturally, floating-point operations can result in numerical errors. In order to overcome these, we will employ interval arithmetics to deal with this issue. The main paradigm is the following: instead of working with arbitrary real numbers, we perform computations over intervals which have representable numbers by the computer as endpoints in order to guarantee that the true result at any point belongs to the interval by which is represented. On these objects, an arithmetic is defined in such a way that we are guaranteed that for every $x \in X, y \in Y$
\begin{align*}
x \star y \in X \star Y,
\end{align*}
for any operation $\star$. For example,
\begin{align*}
[\underline{x},\overline{x}] + [\underline{y},\overline{y}] & = [\underline{x} + \underline{y}, \overline{x} + \overline{y}] \\
[\underline{x},\overline{x}] \times [\underline{y},\overline{y}] & = [\min\{\underline{x}\underline{y},\underline{x}\overline{y},\overline{x}\underline{y},\overline{x}\overline{y}\},\max\{\underline{x}\underline{y},\underline{x}\overline{y},\overline{x}\underline{y},\overline{x}\overline{y}\}].
\end{align*}
We can also define the interval version of a function $f(X)$ as an interval $I$ that satisfies that for every $x \in X$, $f(x) \in I$. Rigorous computation of integrals has been theoretically developed since the seminal works of Moore and many others (see \cite{Berz-Makino:high-dimensional-quadrature,Cordoba-GomezSerrano-Zlatos:stability-shifting-muskat-II,Kramer-Wedner:adaptive-gauss-legendre-verified-computation,Lang:multidimensional-verified-gaussian-quadrature,Moore-Bierbaum:methods-applications-interval-analysis,Tucker:validated-numerics-book} for just a small sample). In our computations, all arithmetic will be double precision (64 bits).

\section{The equations}
\label{sectionequations}

In this section we describe the equation that  a global  rotating solution with constant angular velocity of SQG must satisfy. We will look to the level sets of this solution rather than the solution itself. Let's assume that $\theta(x,t)$ is a smooth solution of \eqref{sqg} with initial data $\theta_0(x)$. On $\theta_0(x)$ we will assume that its level sets can be parameterized by $z_0(\alpha,\rho)$ in such a way that $$\theta_0(z_0(\alpha,\rho))=f(\rho)$$ for some smooth and even
function $f \, : \, \RR\to \RR$. The application $z_0(\alpha,\rho)$ satisfies:
\begin{enumerate}
\item It is one to one from  $\{(\alpha,\rho)\in \RR^2\,:\, -\pi\leq \alpha <\pi, \, 0<\rho<\infty\}$ to $\RR^2\setminus \{0\}$.
\item For all $\alpha$, $z_0(\alpha,0)=(0,0)$.
\item For fixed $\rho>0$, $z^i_0(\alpha,\rho)$, with $i=1,2$, are $2\pi$-periodic and $z_0(\alpha,\rho)$, with $-\pi\leq \alpha <\pi$, is a  closed and $C^1$ curve in $\RR^2$ satisfying the chord-arc condition. Also we parametrize this curve counterclockwise with $\alpha$.
\item It is differentiable with respect to $\alpha$ and $\rho$ in $\T\times [0,\infty)$,  $|z_{0\,\rho}(\alpha,\rho)|>c>0$ in $\T\times [0,\infty)$, $|z_{0\,\alpha}(\alpha,\rho)|>0$ in $\T\times (0,\infty)$ and  $z_{0\,\alpha}^\perp(\alpha,\rho) \cdot z_{0\,\rho}(\alpha,\rho)<0\quad \text{for $-\pi\leq \alpha <\pi$ and $0<\rho<\infty$}.$
\end{enumerate}

Because of the transport character of the equation \eqref{sqg} and by continuity we can assume that the level sets of the solution $\theta(x,t)$ at time $t$ can be parameterized by an application $z(\alpha,\rho,t)$ such that
\begin{align}\theta(z(\al,\rho,t),t) = f(\rho),\label{teta}\end{align}
and it satisfies  properties 1, 3 and 4.

Property 2 changes to $$z(\alpha,0,t)=c(t)\quad \text{for all $\alpha\in \T$},$$
where, here, $c(t)$ is a vector in $\RR^2$.

Differentiating \eqref{teta} with respect to  $\alpha$ and with respect to $\rho$ we have that

\begin{align*}
z_{\al}(\alpha,\rho,t) \cdot \nabla \theta(z(\alpha,\rho,t),t) = 0  \\
z_{\rho}(\alpha,\rho,t) \cdot \nabla \theta(z(\alpha,\rho,t),t) = f_\rho(\rho).
\end{align*}

Therefore

\begin{align}\label{gradiente}
\nabla \theta(z(\alpha,\rho,t)) = \frac{f_\rho(\rho)}{z_{\al}^{\perp} \cdot z_{\rho}(\alpha,\rho,t)}z_{\al}^{\perp}(\alpha,\rho,t),
\end{align}

Taking a time derivative in \eqref{teta} and using \eqref{sqg} and \eqref{gradiente} yields
\begin{align*}
0=&\frac{d }{dt}\theta(z(\alpha,\rho,t),t)=\pa_{t} \theta(z(\al,\rho,t),t) + z_{t}(\alpha,\rho,t) \cdot \nabla \theta(z(\al,\rho,t),t)  \\
=&(-u(z(\al,\rho,t),t) + z_{t}(\alpha,\rho,t) )\cdot \nabla \theta(z(\al,\rho,t),t) \\
=&\left(-u(z(\al,\rho,t),t) + z_{t}(\alpha,\rho,t )\right)\cdot z_{\al}^{\perp}(\alpha,\rho,t) \frac{f_\rho(\rho)}{z_{\al}^{\perp} \cdot z_{\rho}(\alpha,\rho,t)}.
\end{align*}
This last expression is the equation that the level sets $z(\alpha,\rho,t)$ of the solutions $\theta(x,t)$ satisfy.

Notice that since $u(x,t)=R^\perp \theta(x,t)$ we can also write $u(x,t)=-\Lambda^{-1}\nabla^\perp \theta(x,t)$ and therefore
\begin{align*}
u(z(\alpha,\rho,t),t) & = -\frac{1}{2\pi} \int_{\RR^{2}} \frac{1}{|z(\alpha,\rho,t)-y|} \nabla^{\perp} \theta(y,t) dy \\
& = -\frac{1}{2\pi} \int_0^\infty\int_{-\pi}^\pi \frac{1}{|z(\al,\rho,t)-z(\al',\rho',t)|} f_\rho(\rho') z_{\al}(\al',\rho') d\al' d\rho'.\\
\end{align*}
where we just did the change of variables $y=z(\alpha',\rho',t)$ and used \eqref{gradiente}.

Conversely, if $z(\alpha,\rho,t)$ satisfies the equation
\begin{align}\label{leveleq}
&\left(-u(z(\al,\rho,t),t) + z_{t}(\alpha,\rho,t )\right)\cdot z_{\al}^{\perp}(\alpha,\rho,t) \frac{f_\rho(\rho)}{z_{\al}^{\perp} \cdot z_{\rho}(\alpha,\rho,t)}=  0 \\
&z(\alpha,\rho,0)=  z_0(\alpha,\rho)\nonumber
\end{align}
with
\begin{align}
u(z(\alpha,\rho,t),t)=-\frac{1}{2\pi} \int_0^\infty\int_{-\pi}^\pi \frac{1}{|z(\al,\rho,t)-z(\al',\rho',t)|} f_\rho(\rho') z_{\al}(\al',\rho') d\al' d\rho',\label{ulevel}
\end{align}
we can prove that  the  function  $\theta \,:\, \RR^2\times \RR\to \RR $ defined by \eqref{teta}
 is a solution of the equation \eqref{sqg}.

Let us assume now that $\supp(f_\rho) \subset [1-a,1]$, where $0 < a < 1$ will be chosen later. Then, in order to find a solution for SQG, we can solve the equation
\begin{align*}
\left(z_{t}(\alpha,\rho,t) -u(z(\alpha,\rho,t),t)\right) \cdot z_{\al}^{\perp}(\alpha,\rho,t) & = 0,
\end{align*}
with $u(z(\alpha, \rho,t),t)$ given by \eqref{ulevel}, in the domain $\rho \in (1-a,1)$, $\alpha\in \T$. After that we extend the solution $z(\alpha,\rho,t)$ to the domain $0\leq \rho<\infty$, $-\pi\leq \alpha <\pi$ in a smooth way and finally define $\theta(x,t)$ through equation \eqref{teta}.

In addition, we will assume that our solution rotates with angular velocity $\lambda$, counterclockwise. Hence

\begin{align*}
z(\alpha,\rho,t) = \mathcal{O}(t)x(\alpha, \rho),
\quad
\mathcal{O}(t) =
\left(
\begin{array}{cc}
\cos(\lambda t) & -\sin(\lambda t) \\
\sin(\lambda t) & \cos(\lambda t)
\end{array}
\right)
\end{align*}

This implies on one hand

\begin{align*}
z_t(\alpha,\rho,t) = \mathcal{O}_t(t) x(\alpha,\rho) = \lambda
\left(
\begin{array}{cc}
-\sin(\lambda t) & -\cos(\lambda t) \\
\cos(\lambda t) & -\sin(\lambda t)
\end{array}
\right)x(\alpha,\rho), \quad
z_{\al}(\rho,\alpha,t) = \mathcal{O}(t) x_{\al}(\alpha,\rho),
\end{align*}

and

\begin{align*}
&z_{t} \cdot z_{\al}^{\perp}(\alpha,\rho,t)  = \langle \mathcal{O}_t(t) x(\alpha,\rho), (\mathcal{O}(t) x_{\al}(\alpha,\rho))^{\perp} \rangle\\
&= \langle \mathcal{O}_{t}(t) x(\alpha,\rho), \mathcal{O}(t)x_{\al}^{\perp}(\alpha,\rho) \rangle = \langle \mathcal{O}^{T} \mathcal{O}_{t}(t) x(\al,\rho), x_{\al}^{\perp}(\al,\rho) \rangle
= -\lambda \langle x(\alpha,\rho), x_{\al}(\alpha,\rho) \rangle.
\end{align*}

On the other hand from \eqref{ulevel} we see that
\begin{align*}
u(z(\alpha,\rho,t),t)=-\mathcal{O}\frac{1}{2\pi} \int_0^\infty\int_{-\pi}^\pi \frac{1}{|x(\al,\rho)-x(\al',\rho')|} f_\rho(\rho') x_{\al}(\al',\rho') d\al' d\rho'\equiv \mathcal{O}\overline{u}(x(\alpha,\rho)).
\end{align*}
Thus
\begin{align*}
\langle u(z(\alpha,\rho,t),t), z_{\al}^{\perp}(\alpha,\rho,t) \rangle = \langle \mathcal{O} (t)\overline{u}(x(\alpha,\rho)), \mathcal{O}(t) x_{\al}^{\perp}(\alpha,\rho) \rangle = \langle \overline{u}(x(\alpha,\rho)), x_{\al}^{\perp} \rangle,
\end{align*}

which yields having to solve, for the pair $(x(\alpha,\rho), \lambda)$,

\begin{align}\label{equationx}
-\lambda x (\alpha,\rho)\cdot x_{\al}(\alpha,\rho) + \frac{1}{2\pi} x_{\al}^{\perp}(\alpha,\rho) \cdot \int_{0}^\infty \int_{-\pi}^\pi \frac{f_\rho(\rho')}{|x(\al,\rho)-x(\al',\rho')|} x_{\al}(\al',\rho') d\al' d\rho'=0.
\end{align}
We now write $x(\al,\rho)$ in polar coordinates
\begin{align}
x(\alpha,\rho) & = r(\al,\rho)(\cos(\al),\sin(\al)). \label{x}
\end{align}
This choice restricts the class of functions we are considering. However this restriction   will not be strong enough and we will be able to find a solution.

We have the following relations:
\begin{align*}
x_{\al}(\al,\rho)  & = r_{\al}(\al,\rho)(\cos(\al),\sin(\al)) + r(\al,\rho)(-\sin(\al),\cos(\al)) = r_{\al}(\al,\rho) n(\al) + r(\al,\rho) t(\al) \\
x_{\al}^{\perp}(\al,\rho)  & = r_{\al}(\al,\rho) n^{\perp}(\al) + r(\al,\rho) t^{\perp}(\al) = r_{\al}(\al,\rho) t(\al) - r(\al,\rho) n(\al)\\
x_{\al}(\al',\rho') \cdot x_{\al}^{\perp}(\al,\rho) & = r_{\al}(\al,\rho)r_{\al}(\al',\rho') n(\al') \cdot t(\al) - r_{\al}(\al',\rho')r(\al,\rho) n(\al) \cdot n(\al') \\
& + r_{\al}(\al,\rho)r(\al',\rho') t(\al') \cdot t(\al) - r(\al,\rho)r(\al',\rho') n(\al) \cdot t(\al'),
\end{align*}
where
\begin{align*}
n(\al) \cdot n(\al') & = \cos(\al-\al') \\
t(\al) \cdot t(\al') & = \cos(\al-\al') \\
n(\al') \cdot t(\al) & = -\sin(\al-\al') \\
n(\al) \cdot t(\al') & = \sin(\al-\al').
\end{align*}
Moreover, we have that
\begin{align*}
x(\al,\rho) \cdot x_{\al}(\al,\rho) = r(\al,\rho) n(\al) \cdot (r_{\al}(\al,\rho)n(\al) + r(\al,\rho) t(\al)) = r(\al,\rho) r_{\al}(\al,\rho).
\end{align*}
Therefore, equation \eqref{equationx} reads
\begin{align}\label{rotequation}
F[r,\,\lambda]=0 \quad \text{in $\rho\in (1-a,1)$, $\alpha\in\T$},
\end{align}
with
\begin{align}\label{functionalF}
&F[r,\,\lambda] \\&\equiv
\lambda r_{\al}(\al,\rho)  - \frac{1}{2\pi} \int_0^\infty \int_{-\pi}^\pi \frac{f_\rho(\rho')}{|x(\al,\rho)-x(\al',\rho')|}\cos(\al-\al')(r_{\al}(\al',\rho')-r_{\al}(\al,\rho)) d\al' d\rho' \nonumber\\
& + \frac{r_{\al}(\al,\rho)}{2\pi r(\al,\rho)} \int_0^\infty \int_{-\pi}^\pi \frac{f_\rho(\rho')}{|x(\al,\rho)-x(\al',\rho')|}\cos(\al-\al')(r(\al',\rho')-r(\al,\rho)) d\al' d\rho'\nonumber \\
& - \frac{1}{2\pi r(\al,\rho)} \int_0^\infty \int_{-\pi}^\pi \frac{f_\rho(\rho')}{|x(\al,\rho)-x(\al',\rho')|}\sin(\al-\al')(r(\al,\rho)r(\al',\rho') + r_{\al}(\al,\rho)r_{\al}(\al',\rho')) d\al' d\rho',\nonumber
\end{align}
where $x(\alpha,\rho)$ is given by \eqref{x} and we have added and subtracted $$\frac{r_{\al}(\al,\rho)}{2\pi r(\al,\rho)} \int_0^\infty \int_{-\pi}^\pi \frac{f_\rho(\rho')}{|x(\al,\rho)-x(\al',\rho')|}\cos(\al-\al')r(\al,\rho) d\al' d\rho'$$ for cosmetic reasons.

The rest of the paper consists in finding a nontrivial solution  $(r(\alpha,\rho),\lambda)$ of \eqref{rotequation} and with $r_\rho(\alpha,\rho)>c>0.$

\section{Main theorem and Crandall-Rabinowitz (C-R) theorem}\label{CR}
This section is devoted to state the main theorem of the paper. But firstly, we will fix the function $f(\rho)$. The derivative of this  function will be given by the expression
\begin{align}\label{definitiondef}
 &\frac{a}{2}f_\rho\left(\rho\right) \\&=
\left\{
\scriptsize
\begin{array}{cc}
 -\frac{1}{2\beta^{9}}(126 \beta^4 - 420 \beta^3 (1 + \tilde{\rho}) + 540 \beta^2 (1 + \tilde{\rho})^2 - 315 \beta (1 + \tilde{\rho})^3 + 70 (1 + \tilde{\rho})^4)(1 + \tilde{\rho})^5 & \text{ if } -1 \leq \tilde{\rho} \leq -1 + \beta \\
 -0.5 & \text{ if } -1 + \beta \leq \tilde{\rho} \leq 1 - \beta \\
\frac{1}{2\beta^{9}}(126 \beta^4 + 420 \beta^3 (-1 + \tilde{\rho}) + 540 \beta^2 (-1 + \tilde{\rho})^2 + 315 \beta (-1 + \tilde{\rho})^3 + 70 (-1 + \tilde{\rho})^4) (-1 + \tilde{\rho})^5 & \text{ if } 1 - \beta \leq \tilde{\rho} \leq 1 \\
\end{array}
\right.\nonumber
\end{align}
where $\beta$ will be chosen later and $\tilde{\rho} = \frac{2}{a}(\rho-1)+1$. We take $f(1)=0$. This expression will be used in order to compute some integrals where $f_\rho(\rho)$ arises. Next we describe the main properties of this function:
\begin{enumerate}
\item The function $f(\rho)$ defined in $(1-a,1)$ admits a $C^4$-extension (we still call it $f$) to $\RR$. This extension is given by
\begin{align*}
f(\rho)=\left\{\begin{array}{ccc} 1 & \rho\leq 1-a\\ f(\rho) & 1-a< \rho <1 \\ 0 & \rho \geq 1\end{array}\right\}.
\end{align*}
\item It is strictly decreasing in $(1-a,1)$.
\item The derivative $f_\rho(\rho)$ is  constant for $1-a+\frac{a\beta}{2}\leq \rho \leq 1-\frac{a\beta}{2}$.
\end{enumerate}

\begin{rem}
Due to the computer-assisted nature of some parts of the proof, the choice of $f(\rho)$ and all the parameters of the problem $(n,a,\beta)$ need to be explicit. A similar strategy works for other $n$-fold symmetric solutions, and more regular solutions can be obtained by choosing more regular (explicit) profiles $f$ if the computer-assisted parts of the proof yield suitable numbers.
\end{rem}

\begin{theorem}\label{main} Let $a = 0.05$ and $\beta= \frac{2}{512} = 2^{-8}$ and consider the domain $$\Omega_a\equiv \{ (\alpha,\, \rho)\,:\,\alpha\in \T,\, 1-a<\rho<1\}$$ and  $f\in C^{4}([1-a,1])$ as in \eqref{definitiondef}. Then there exists a branch of nontrivial smooth solutions, with $3-$fold symmetry, of equation \eqref{rotequation}, in $H^{4,3}(\Omega_a)$, bifurcating from $r(\alpha,\rho)=\rho$ and  $\lambda=\lambda_3$ for some $\lambda_3\in \RR$.
\end{theorem}

\begin{rem} Nontrivial solutions means that the function $r(\alpha, \rho)$ depends on $\alpha$ in a nontrivial way. See Definition \ref{Hkl} for a precise definition of the space $H^{k,l}({\Omega_a})$.
\end{rem}

From section \ref{sectionequations} is clear that Theorem \ref{globalsqg} follows from Theorem \ref{main}.

The proof of Theorem \ref{main} relies on the Crandall-Rabinowitz theorem. We recall here the statement of this theorem from \cite{Crandall-Rabinowitz:bifurcation-simple-eigenvalues} for expository purposes.

\begin{theorem}[Crandall-Rabinowitz] Let $X$, $Y$ be Banach spaces, $V$ a neighborhood of
0 in $X$ and
\begin{align*}
F \,:\, &   V\times (-1,\,1) \rightarrow Y\\
& \,\qquad (r,\,\mu)\,\,\,\,\rightarrow F[r,\,\mu]
\end{align*}
have the properties
\begin{align}
\emph{$F[0,\, \mu] = 0$ for any $|\mu| < 1$.} \\
\emph{The partial derivatives $\pa_\mu F$, $\pa_r F$ and $\pa^2 _{\mu r}F$ exist and are continuous.} \\
\emph{$\mathcal{N}(\pa_r F[0, 0])$ and $Y/\mathcal{R}(\pa_r F[0, 0])$ are one-dimensional.} \\
\emph{$\pa^2_{\mu r} F[0, 0]r_0\not\in \mathcal {R}(\pa_r F[0, 0])$, where $\mathcal{N}(\pa_r F[0, 0]) = \text{span }r_0.$} \label{condTransversality}
\end{align}
(Here $\mathcal{N}$ and $\mathcal{R}$ denote the kernel and range respectively). If $Z$ is any complement of $\mathcal{N}(\pa_r F(0, 0))$ in $X$, then there is a neighborhood $U$ of (0, 0) in
$X \times \RR$, an interval $(-b, b)$, and continuous functions
\begin{align*}
\phi\, : \, (-b, b) \rightarrow \RR && \psi\, :\, (-b, b) \rightarrow Z\end{align*}
such that $\phi(0) = 0, \psi(0) = 0$ and
$$F^{-1}(0)\cap U=\{ (\xi r_0+\xi\psi(\xi), \phi(\xi))\,:\, |\xi|<b\}\cup \{(0,t)\,:\, (0,t)\in U\}.$$
\end{theorem}

We will check in the following section the hypotheses of the C-R theorem.

\section{Checking the hypotheses of the C-R theorem for the equation \eqref{rotequation}}\label{checking}

In this section we will check the hypotheses of the C-R theorem in suitable Banach spaces $X$ and $Y$ in order to find a nontrivial branch of the solution $(r(\alpha,\rho), \lambda)$ of \eqref{rotequation}. In order to be able to apply this theorem in the way that it is written in section \ref{CR} we will define new variables:
\begin{align}
\overline{r}(\alpha,\rho)& \equiv r(\alpha,\rho)-\rho\\
\mu & \equiv \lambda-\lambda_3\\
\overline{F}[\overline{r},\mu]& \equiv  F[\overline{r}+\rho,\lambda_3+\mu]\label{functionalfb}
\end{align}
with $\lambda_3$ to be fixed later. Thus we understand \eqref{rotequation} as an equation for $(\overline{r},\mu)$ rather that for $(r,\lambda)$. In fact, we look for solutions of
\begin{align}\label{rbarra}
\overline{F}[\overline{r},\mu]=0 \quad \text{in $\Omega_a$}.
\end{align}
Let us also define the spaces $H^{k,l}(\Omega_a)$ for $k,\,l \in \N$ and $k\geq l$ as follows:
\begin{align}\label{Hkl}
\left\{ r\in L^2(\Omega_a) \, : \, ||r||^2_{L^2(\Omega_a)}+ ||\pa^{l}_\rho r||_{L^2(\Omega_a)}^2+\sum_{j=0}^l||\pa^{k-j}_\alpha \pa_\rho^{j} r||^{2}_{L^2(\Omega_a)}<\infty\right\}.
\end{align}


We will work in the space $H^{4,3}$ and notice that $H^{4,3} \subset C^2$. This is shown below. The reason why we work with this space is because the functional $\overline{F}$ takes 1 derivative in $\al$ but no derivatives in $\rho$. Due to this anisotropy we would not be able to apply the C-R theorem by solely using homogeneous spaces (see \cite{Lannes:water-waves-book} for additional information).

\begin{lemma}
Let $\Omega \in [0,1] \times [-\pi,\pi]$ and $f: \Omega \rightarrow \mathbb{R} \in H^{4,3}(\Omega)$. Then:

\begin{align*}
\|f\|_{L^{\infty}(\Omega)} + \|\pa_{x}^{2} f\|_{L^{\infty}(\Omega)} + \|\pa_{\al}^{2} f\|_{L^{\infty}(\Omega)} \leq C \|f\|_{H^{4,3}(\Omega)}
\end{align*}

\end{lemma}

\begin{proof}
If $f \in H^{4,3}(\Omega)$, then:
\begin{align*}
\|f\|_{L^{2}(\Omega)}^{2} + \|\pa_{x}^{3}f \|_{L^{2}(\Omega)}^{2} + \|\pa_{\al}^{4}f \|_{L^{2}(\Omega)}^{2} + \|\pa_{\al}^{3} \pa_{x}f \|_{L^{2}(\Omega)}^{2} + \|\pa_{\al}^{2} \pa_{x}^{2}f \|_{L^{2}(\Omega)}^{2}
 + \|\pa_{\al}\pa_{x}^{3}f \|_{L^{2}(\Omega)}^{2} < C
\end{align*}

On one hand, $f_{\al} \in H^{3}(\Omega)$, since

\begin{align*}
\|f_{\al}\|_{H^{3}(\Omega)}^{2} & = \|f_{\al}\|_{L^{2}(\Omega)}^{2} +  \|\pa_{x}^{3}f_{\al} \|_{L^{2}(\Omega)}^{2} +  \|\pa_{\al}^{3}f_{\al} \|_{L^{2}(\Omega)}^{2} \\
& \lesssim \|f\|_{L^{2}(\Omega)}^{2} + \|\pa_{\al}^{4} f\|_{L^{2}(\Omega)}^{2} + \|\pa_{x}^{3}\pa_{\al} f \|_{L^{2}(\Omega)}^{2} +  \|\pa_{\al}^{4}f \|_{L^{2}(\Omega)}^{2} \lesssim \|f\|_{H^{4,3}(\Omega)}
\end{align*}

This implies that $f_{\al} \in C^{1+\gamma}(\overline{\Omega})$, yielding $\|f_{\al \al}\|_{L^{\infty}(\overline{\Omega})} \leq C$ and $\|f_{\al x}\|_{L^{\infty}(\overline{\Omega})} \leq C$ for some constant $C$.

On the other hand, if we define

\begin{align*}
g(\al) = \int_{0}^{1} \pa_{x}^{2} f(x,\al) dx,
\end{align*}

we claim that $g \in H^{1}([-\pi,\pi])$. In order to see this, we can compute

\begin{align*}
\|g\|_{L^{2}([-\pi,\pi])}^{2} = \int_{-\pi}^{\pi} \left(\int \pa_{x}^{2} f(x,\al) dx\right)^{2} d\al \leq \int_{-\pi}^{\pi} \int_{0}^{1} |\pa_{x}^{2} f(x,\al)|^{2} dx d\al \leq \|f\|_{H^{4,3}(\Omega)}.
\end{align*}

In addition

\begin{align*}
\|g_{\al}\|_{L^{2}([-\pi,\pi])}^{2} = \int_{-\pi}^{\pi} \left(\int \pa_{x}^{2} \pa_{\al} f(x,\al) dx\right)^{2} d\al \leq \int_{-\pi}^{\pi} \int_{0}^{1} |\pa_{x}^{2} \pa_{\al} f(x,\al)|^{2} dx d\al \leq \|f\|_{H^{4,3}(\Omega)}.
\end{align*}

Therefore

\begin{align*}
\|g\|_{L^{\infty}(\Omega)} \leq \|g\|_{L^{\infty}([-\pi,\pi])} \leq \|g\|_{H^{1}([-\pi,\pi])} \leq C \|f\|_{H^{4,3}(\Omega)}.
\end{align*}

We have that

\begin{align*}
\pa_{x}^{2} f(x,\al) = \pa_{x}^{2} f(x,\al) - \int_{0}^{1} \pa_{x}^{2} f(x',\al) dx' + g(\al),
\end{align*}

which we can bound in the following way:

\begin{align*}
\|\pa_{x}^{2} f\|_{L^{\infty}(\Omega)} \leq \left\|\int_{0}^{1} \pa_{x}^{2} f(x,\al) - \pa_{x}^{2} f(x',\al) dx'\right\|_{L^{\infty}(\Omega)} + \|g\|_{L^{\infty}(\Omega)}.
\end{align*}

In addition:

\begin{align*}
\int_{0}^{1} \pa_{x}^{2} f(x,\al) -  \pa_{x}^{2} f(x',\al) dx'
= \int_{0}^{1} \int_{x'}^{x} \pa_{x}^{3} f(x'',\al) dx'' dx' = h(\al,x)
\end{align*}

We now fix $x$ and we show that $\|g(\al,x)\|_{H^{1}([-\pi,\pi])}$ is uniformly bounded. We achieve that by using the following estimate:

\begin{align*}
\|h(\al,x)\|_{L^{2}([-\pi,\pi])} & \leq \int_{0}^{1} \int_{x'}^{x}\left(\int_{-\pi}^{\pi} |\pa_{x}^{3} f(x'',\al)|^{2}d\al\right)^{\frac12} dx'' dx' \\
& \leq \int_{0}^{1}|x-x'|^{\frac12}\left(\int_{x'}^{x} \int_{-\pi}^{\pi}|\pa_{x}^{3}f(x'',\al)|^{2}\text{sign}(x-x') d\al dx''\right)^{\frac12} dx' \\
& \leq C\left(\int_{0}^{1}\left(\int_{0}^{1} \int_{-\pi}^{\pi}|\pa_{x}^{3}f(x'',\al)|^{2} d\al dx''\right)^{\frac12} dx'\right) \\
& \leq C \|f\|_{H^{4,3}(\Omega)},
\end{align*}

where $C$ is independent of $x$. We can do the same procedure with $\pa_{\al} h(\al,x)$, getting

\begin{align*}
\|h(\al,x)\|_{L^{\infty}([-\pi,\pi])} \leq \|h(\al,x)\|_{H^{1}([-\pi,\pi])} \leq C\|f\|_{H^{4,3}(\Omega)}
\end{align*}

Taking the supremum over $x$ yields the desired result.

\end{proof}

The theorem we will prove is the following:
\begin{theorem}\label{thm3} Let $f$ and $a$ be as in Theorem \ref{main}. Then there exist $(\overline{r}_0(\alpha,\rho), \lambda_3)\in H^{4,3}(\Omega_a)\times \RR$,
 an interval $(-b, b)$, and continuous functions
\begin{align*}
\phi\, : \, (-b, b) \rightarrow \RR && \psi\, :\, (-b, b) \rightarrow Z\end{align*}
with $\phi(0) = 0, \psi(0) = 0$, such that, if $Z$ is any complement of $\text{span}\{\overline{r}_0\}$ in $H^{4,3}(\Omega_a)$,

  $$\overline{F}[\xi \overline{r}_0+ \xi \psi(\xi),\phi(\xi)]=0,$$
for $|\xi|<b$.

In addition these solutions will have $3-$fold symmetry.
\end{theorem}
Here it is important to remark that this theorem provides a nontrivial solution $r(\alpha,\rho)=\rho +\xi \overline{r}_0+\xi \psi(\xi)$ of \eqref{rotequation} with $\lambda=\lambda_3+\phi(\xi)$ satisfying $|r_\rho(\alpha,\rho)|>c>0$ if we take $\xi$ small enough. Theorem \ref{main} follows from Theorem \ref{thm3}.

\subsection{Step 1. The functional setting and the hypothesis 1}
Our first step is to define the spaces we will work with in order to apply the C-R theorem. The spaces $H^{k,l}_{3, \text{even}}(\Omega_a)$ and $H^{k,l}_{3, \text{odd}}(\Omega_a)$  will be given by
\begin{align*}
\left\{ \overline{r}\in H^{k,l}(\Omega_a) \,:\, \overline{r}(\alpha,\rho)=\sum_{m=1}^\infty \hat{r}_m(\rho)\cos(3m\alpha)\right\},\end{align*} and \begin{align*}
\left\{ \overline{r}\in H^{k,l}(\Omega_a) \,:\, \overline{r}(\alpha,\rho)=\sum_{m=1}^\infty \hat{r}_m(\rho)\sin(3m\alpha)\right\},
\end{align*}
respectively.

One of the purposes to introducing these spaces, which only represent frequencies multiples of $3$, is to be able to show the $3-$fold symmetry of the solution. Our starting space $X$ will be $H^{4,3}_{3, \text{even}}(\Omega_a)$.

 The target space $Y$ will be $H^{3,3}_{3, \text{odd}}(\Omega_a)$. Notice that a function in $H^{4,3}(\Omega_a)$ belongs to $C^{2,2}(\overline{\Omega_a})=C^1(\overline{\Omega_a})$. Finally we take the  neighbourhood $V$ of $0$ in $X$ to be
\begin{align*}
V \equiv \left\{ \overline{r}\in H^{4,3}_{3, \text{even}}(\Omega_a)\, : \, ||\overline{r}||_{H^{4,3}_{3, \text{even}}(\Omega_a)}<\delta  \right\}
\end{align*}
for $\delta>0$. The parameter $\delta$ will be fixed later (small enough).

Given these definitions  we need to show the following lemma:
\begin{lemma}\label{lema42}Let $\overline{F}[\overline{r},\mu]$ be as in \eqref{functionalfb}. Then, for fixed $a\in (0,1)$, there exists $\delta>0$ small enough so that
\begin{align*}
\overline{F} \, : \,  V\times [-1, 1]\to H^{3,3}_{3,\, \emph{odd}}(\Omega_a).
\end{align*}
\end{lemma}
\begin{proof}
Here we recall the definition of the functional
\begin{align*}
&F[r,\,\lambda] \\&\equiv
\lambda r_{\al}(\al,\rho)  - \frac{1}{2\pi} \int_0^\infty \int_{-\pi}^\pi \frac{f_\rho(\rho')}{|x(\al,\rho)-x(\al',\rho')|}\cos(\al-\al')(r_{\al}(\al',\rho')-r_{\al}(\al,\rho)) d\al' d\rho' \nonumber\\
& + \frac{r_{\al}(\al,\rho)}{2\pi r(\al,\rho)} \int_0^\infty \int_{-\pi}^\pi \frac{f_\rho(\rho')}{|x(\al,\rho)-x(\al',\rho')|}\cos(\al-\al')(r(\al',\rho')-r(\al,\rho)) d\al' d\rho'\nonumber \\
& - \frac{1}{2\pi r(\al,\rho)} \int_0^\infty \int_{-\pi}^\pi \frac{f_\rho(\rho')}{|x(\al,\rho)-x(\al',\rho')|}\sin(\al-\al')(r(\al,\rho)r(\al',\rho') + r_{\al}(\al,\rho)r_{\al}(\al',\rho')) d\al' d\rho'\\
&\equiv F_1[r,\lambda]+F_2[r]+F_3[r]+F_4[r],
\end{align*}
with $r=\rho+\overline{r}$ and $x=r(\cos(\alpha),\sin(\alpha))$.
It is easy to check that  $F_1[r,\lambda]=\lambda \pa_\alpha r$  belongs to $H^{3,3}_{3,\, \text{odd}}(\Omega_a)$.

Next we show that  \begin{align}\label{Fenh3} F_i[r]\in H^{3,3}(\Omega_a),\quad i=2,3,4.\end{align} In order to do it we notice that, since $||\cdot||_{C^2(\overline{\Omega_a})}\leq C||\cdot  ||_{H^{4,3}(\Omega_a)}$,
we can choose, for fixed $a\in (0,1)$, $\delta$ small enough to have  that $r(\alpha,\rho)>c_0(a,\delta)>0$ and $r_\rho(\alpha,\rho)>c_1(a,\delta)>0$, for every $(\alpha,\rho)\in \Omega_a$. By comparing equations \eqref{equationx} and \eqref{rotequation} we see that
\begin{align*}
\left(F_2[r]+F_3[r]+F_4[r]\right)(\alpha,\rho)=&-\frac{x_\alpha^\perp(\alpha,\rho)}{2\pi r(\alpha,\rho)}\cdot \int_{0}^\infty\int_{-\pi}^\pi f_\rho(\rho')\frac{x_\alpha(\alpha',\rho')}{|x(\alpha,\rho)-x(\alpha',\rho')|}d\alpha' d\rho'\\
=& \frac{x_\alpha^\perp(\alpha,\rho)}{2\pi r(\alpha,\rho)}\cdot \int_{0}^\infty\int_{-\pi}^\pi f_\rho(\rho')\frac{x_\alpha(\alpha,\rho)-x_\alpha(\alpha',\rho')}{|x(\alpha,\rho)-x(\alpha',\rho')|}d\alpha' d\rho'.
\end{align*}
with $x(\alpha,\rho)=r(\alpha,\rho)(\cos(\alpha),\sin(\alpha))$ and $r(\alpha, \rho)=\rho+\overline{r}(\alpha,\rho)$. We will extend the restriction of the function $f_\rho$ to the positive real axis  $f_\rho|_{\RR^+}$ to $\RR$ by zero. We still call this extension $f_\rho$. Thus, a change of variables yields,
\begin{align*}
\left(F_2[r]+F_3[r]+F_4[r]\right)(\alpha,\rho)& =\underbrace{\frac{x_\alpha^\perp(\alpha,\rho)}{2\pi r(\alpha,\rho)}}_{g(\alpha,\rho)}\cdot \underbrace{\int_{-\infty}^\infty\int_{-\pi}^\pi f_\rho(\rho-\rho')\frac{x_\alpha(\alpha,\rho)-x_\alpha(\alpha-\alpha',\rho-\rho')}{|x(\alpha,\rho)-x(\alpha-\alpha',\rho-\rho')|}d\alpha' d\rho'}_{B[f,x](\alpha,\rho)}.
 \end{align*}
We will use the following notation. For a general function $h(\alpha,\rho)$ we define \begin{align*} h=&h(\alpha,\rho)\\ h'= & h(\alpha',\rho')\\h''=& h(\alpha-\alpha',\rho-\rho')\\ \Delta h =&h-h'\\ \Delta h'=&h-h''\end{align*}
Thus we can write
\begin{align*}
B[f,x]=\int_{-\infty}^\infty \int_{-\pi}^\pi f_\rho''\frac{\Delta x'_\alpha}{|\Delta x'|}d\alpha'd\rho'.
\end{align*}

Next we look at \begin{align}\label{deriv}\sum_{j=0}^3 \pa^{3-j}_\alpha\pa^{j}_\rho \left(g(\alpha,\rho)B[f,x](\alpha,\rho)\right).\end{align}

We will consider two groups of terms. Group 1 consists of the terms
\begin{align*}
g \pa^{3}_\alpha B,\,\, \,\,\pa_\alpha g \pa_\alpha^{2}B,\,\,\,\, \pa_\alpha g \pa^{2}_{\alpha\rho}B, \,\,\,\, \pa_\rho g \pa^{2}_{\alpha} B,\,\,\,\, g \pa^{2}_\alpha\pa_\rho B,\,\,\,\, \pa_\alpha g\pa^{2}_\rho B,
\,\,\,\, \pa_\rho g \pa^{2}_{\alpha\rho}B,\,\,\,\, g\pa^{2}_\rho\pa_\alpha B,\,\,\,\, \pa_\rho g \pa^{2}_\rho B, \,\,\,\, g\pa^{3}_\rho B.
\end{align*}
Group 2  consists of the terms
\begin{align*}
\pa^{3}_{\alpha}g B,\,\, \,\, \pa^{2}_\alpha g\pa_\alpha B,\,\, \,\, \pa^{2}_\alpha g\pa_\rho B,\,\, \,\, \pa^{2}_{\alpha\rho}g\pa_\alpha B,\,\, \,\, \pa^{2}_\rho\pa_\alpha g B, \,\, \,\,  \pa^{2}_\rho g \pa_\alpha B,\,\, \,\, \pa^{2}_{\alpha\rho}g\pa_\rho B,\,\, \,\, \pa^{2}_\rho g\pa_\rho B,\,\, \,\,  \pa^{3}_\rho g B.
\end{align*}
It is easy to check that expression \eqref{deriv} is given by a linear combination of the terms in group 1 and group 2. On one hand, since in group 1 there is no more than one derivative acting on $g$ and
$$||g||_{L^\infty(\overline{\Omega_a})},\quad ||\pa_\alpha g||_{L^\infty(\overline{\Omega_a})}, \quad ||\pa_\rho g||_{L^\infty(\overline{\Omega_a})} \leq C(\delta, a)$$
in order to bound its terms we need to estimate
\begin{align*}
 \pa^{3}_\alpha B,\,\, \,\, \pa_\alpha^{2}B,\,\,\,\,  \pa^{2}_{\alpha\rho}B, \,\,\,\,  \pa^{2}_\alpha\pa_\rho B,\,\,\,\, \pa^{2}_\rho B,\,\,\,\,  \pa^{2}_\rho\pa_\alpha B,\,\,\,\,  \pa^{3}_\rho B.
\end{align*}
On the other hand, since
\begin{align*}
\pa^{2}_\alpha g,\quad \pa^{2}_\rho g,\quad \pa^{2}_{\alpha\rho}g, \quad \pa^{3}_\alpha g,\quad \pa^{3}_\rho g,\quad \pa^{2}_\alpha\pa_\rho g,\quad \pa^{2}_\rho \pa_\alpha g,
\end{align*}
have $L^2$-norms bounded by some constant depending on $\delta$ and $a$, in order to estimate the terms in group 2 we just need to control the $L^\infty$ norms of $B$, $\pa_\alpha B$, $\pa_\rho B$.  These norms are bounded by $||B||_{H^3(\overline{\Omega_a})}$.

In addition since the $L^2$-norm of $B[f,x]$ is easy to control we just have to estimate the $L^2$-norms of the derivatives of order 3 of $B[f,x]$.
\begin{lemma}\label{d3B} Let $r = \rho + \overline{r}$, where $\overline{r} \in V$,  and $x = r(\cos(\al),\sin(\al))$, the derivatives $\pa^{3}_{\sigma_1,\sigma_2,\sigma_3}B[f,x]$ where $\sigma_i$ is either $\alpha$ or $\rho$, with $i=1,2,3$, are in $L^2$ with norm bounded by a constant $C$ just depending on $\delta$, $a$ and $||f||_{C^4}$.
\end{lemma}
\begin{proof}
This lemma  will be proved by using the following lemma:
\begin{lemma}\label{cuerdaarco} Let $r = \rho + \overline{r}$, where $\overline{r} \in V$,  and $x = r(\cos(\al),\sin(\al))$, then there exists a constant $c(a,\delta)>0$ such that
\begin{align*}
|x(\alpha,\rho)-x(\alpha-\alpha',\rho-\rho')|^2\geq c(a,\delta)\left(\alpha'^2+\rho'^2\right).
\end{align*}
\end{lemma}
\begin{proof}
Because of the definition we have that
\begin{align*}
|\Delta x'|^2=r^2+r''^2-2rr''\cos(\alpha')=(\Delta r')^2+4 r r''\sin^2\left(\frac{\alpha'}{2}\right).
\end{align*}
Now we notice that $r=\rho+\overline{r}$ and then, for $\rho\in (1-a,1)$, there exists $c_0(a,\delta)>0$ such that $r\geq c_0(a,\delta)$, where $c_0(a,\delta)$ is increasing with $a$ and decreasing with $\delta$. Since $\rho-\rho'$  belongs to $(1-a,1)$, $r''\geq c_0(a,\delta)$ too. Both inequalities together yield
\begin{align*}
|\Delta x'|^2\geq (r-r'')^2+ c_0(a,\delta)\sin^2\left(\frac{\alpha'}{2}\right).
\end{align*}
In addition, $r-r'=\rho'+ \Delta\overline{r}'$, and $|\overline{r}'|\leq C(\delta)(|\alpha'|+|\rho'|)$ where $C(\delta)\to 0$  when $\delta\to 0$. Then
\begin{align*}
(r-r'')^2\geq\rho'^2-C(\delta)\left(\alpha'^2+\rho'^2\right)
\end{align*}
where $C(\delta)\to 0$  when $\delta\to 0$. Therefore we obtain that, by making $\delta$ small enough
\begin{align*}
|\Delta x'|^2\geq c(a,\delta)\left(\alpha'^2+\rho'^2\right).
\end{align*}
\end{proof}
Let $\partial$ mean differentiation with respect to either $\alpha$ or $\rho$. Then, the derivatives $\pa^{3} \left(f_\rho''\frac{\Delta x'_\alpha}{|\Delta x'|}\right)$,  consist of terms of the form
\begin{align*}
&\pa^{3}f_\rho''\frac{\Delta x'_\alpha}{|\Delta x|}, \quad f''_\rho \pa^{3}\Delta x_\alpha'\frac{1}{|\Delta x'|},\quad f_\rho''\Delta x_\alpha' \pa^{3}\frac{1}{|\Delta x'|},\\
& \pa^{2}f''_\rho \pa \Delta x'_\alpha \frac{1}{|\Delta x'|},\quad \pa^{2}f''_\rho \Delta x_\alpha'\pa \frac{1}{|\Delta x'|},\quad \pa f''_\rho \pa^{2}\Delta x'_\alpha\frac{1}{|\Delta x'|},\\
&  f''_\rho\pa^2 \Delta x'_\alpha \pa\frac{1}{|\Delta x|},  \quad \pa f''_\rho \Delta x'_\alpha \pa^{2}\frac{1}{|\Delta x'|},\quad  f''_\rho \pa \Delta x'_\alpha \pa^{2}\frac{1}{|\Delta x'|}\\
&\pa f''_\rho \pa \Delta x'_\alpha \pa \frac{1}{|\Delta x'|}.
\end{align*}
Since $f\in C^4$, Lemma \ref{cuerdaarco}, the fact that
$$\int_{\rho-1}^{\rho-(1-a)}\intpi\frac{1}{\sqrt{\alpha'^2+\rho'^2}}d\alpha'd\rho'\leq C $$
and that $||\overline{r}||_{C^2}\leq C(\delta)$ we obtain that the terms in $\pa^{3}B$ coming from  $\pa^{3} f''_\rho\frac{\Delta x_\alpha'}{|\Delta x'|}$ and $\pa^{2} f_\rho'' \pa \Delta x_\alpha' \frac{1}{|\Delta x'|}$ are in $L^\infty$ with $L^\infty-$norm bounded by some constant $C\left(a,\delta,||f||_{C^4(\overline{\Omega_a})}\right)$.

The terms in $\pa^{3}B$ coming from $f_\rho''\pa^{3}\Delta x_\alpha'\frac{1}{|\Delta x'|}$ and $\pa f_\rho''\pa^{2}\Delta x'_\alpha \frac{1}{|\Delta x'|}$ all can be bounded in $L^2$ in the following way. Again we will use the bound for $f$ in $C^4$ and Lemma \ref{cuerdaarco}. Let us focus on
\begin{align}\label{i1i2}
\int_{-\infty}^\infty \int_{-\pi}^\pi f''_\rho \frac{\pa^3\Delta x_\alpha}{|\Delta x'|}d\alpha' d\rho'=\pa^{3}x_\alpha\int_{\rho-1}^{\rho-(1-a)}\int_{-\pi}^\pi f''_\rho \frac{1}{|\Delta x'|}d\alpha'd\rho'-\int_{-\infty}^{\infty}\int_{-\pi}^\pi f''_\rho \frac{\pa^{3}x''_\alpha}{|\Delta x'|}d\alpha'd\rho'\equiv I_1+I_2.
\end{align}
Since
\begin{align*}
\left|\int_{\rho-1}^{\rho-(1-a)}\int_{-\pi}^\pi f''_\rho \frac{1}{|\Delta x'|}d\alpha'd\rho'\right|\leq C\left(a,\delta, ||f||_{C^4}\right)
\end{align*}
we have that $||I_1||_{L^2(\overline{\Omega_a})}\leq C\left(a,\delta, ||f||_{C^4}\right) ||\pa^{3}x_\alpha||_{L^2(\overline{\Omega_a})}$.

In order to bound $I_2$ we notice that, after a change of variables,
\begin{align*}
|I_2|=\left|\int_{1-a}^1\int_{-\pi}^\pi \frac{f'_\rho}{|\Delta x|}\pa^{3}x_\alpha' d\alpha'd\rho'\right|\leq C(a,\delta)\int_{1-a}^1\int_{-\pi}^\pi \frac{|f'_\rho|}{|\sqrt{(\alpha-\alpha')^2+(\rho-\rho')^2}|}|\pa^{3}x_\alpha'| d\alpha'd\rho'
\end{align*}
and therefore Young's inequality applies to yield $$||I_2||_{L^2(\overline{\Omega_a})}\leq C\left(a,\delta, ||f||_{C^4}\right) ||\pa^{3}x_\alpha||_{L^2(\overline{\Omega_a})}.$$

It remains to bound the terms with derivatives acting on the factor $\frac{1}{|\Delta x'|}$.

We first will deal with the terms in $\pa^{3}B$  with a factor $\pa\frac{1}{|\Delta x'|}$ i.e., the terms coming from $\pa^{2}_\rho f''_\rho\Delta x_\alpha' \pa \frac{1}{|\Delta x'|}$, $f''_\rho\pa^{2}\Delta x_\alpha'\pa \frac{1}{|\Delta x'|}$ and $\pa f''_\rho \pa \Delta x'_\alpha \pa \frac{1}{|\Delta x'|}$. Just a computation shows that
\begin{align*}
\pa\frac{1}{|\Delta x'|}=-\frac{\pa\Delta x'\cdot \Delta x'}{|\Delta x'|^3}
\end{align*}
and therefore by Lemma \ref{cuerdaarco} and because of  $x\in C^2$ we have that
\begin{align*}
\left|\pa\frac{1}{|\Delta x'|}\right|\leq C(a,\delta)\frac{1}{\sqrt{\alpha'^2+\rho'^2}}.
\end{align*}
Therefore the terms in $\pa^{3}B$ coming from $\pa^{2}_\rho f''_\rho \Delta x_\alpha' \pa \frac{1}{|\Delta x'|}$, and $\pa f''_\rho \pa \Delta x'_\alpha \pa \frac{1}{|\Delta x'|}$ are actually bounded in $L^\infty$. The term coming from $f''_\rho\pa^{2}\Delta x_\alpha'\pa \frac{1}{|\Delta x'|}$ is bounded as we did before for $I_1+I_2$ in \eqref{i1i2}.

The term with two derivatives of the factor $\frac{1}{|\Delta x'|}$ which causes more difficulties is $f_\rho''\pa \Delta x_\alpha \pa^{2}\frac{1}{|\Delta x'|}$. We will use the following embedding: since $x_\alpha\in H^{3,3}(\Omega_a)=H^3(\Omega_a)$ we know that $x_\alpha\in C^{1+\gamma}(\overline{\Omega_a})$ with $C^{1+\gamma}(\overline{\Omega_a})-$norm bounded for some constant $C(\delta)$. Then, since,
\begin{align*}
\pa^{2}\frac{1}{|\Delta x'|}=-\frac{\pa^2\Delta x'\cdot \Delta x' }{|\Delta x'|^3}-\frac{\pa \Delta x' \cdot \pa \Delta x' }{|\Delta x'|^3}+3\frac{\pa\Delta x' \cdot \Delta x' \pa\Delta x'\cdot\Delta x'}{|\Delta x'|^5}
\end{align*}
we can use the previous embedding to estimate
\begin{align*}
\left|\pa\Delta x'_\alpha\pa^{2}\frac{1}{|\Delta x'|}\right|\leq C(a,\delta)\left(\frac{1}{\left(\alpha'^2+\rho'^2\right)^{1-\frac{\alpha}{2}}}|\pa^{2}\Delta x'|+\frac{1}{\sqrt{\alpha'^2+\rho'^2}}|\pa\Delta x'_\alpha|\right).
\end{align*}
Therefore by using again Young's inequality we bound the term $$\int_{-\infty}^\infty \int_{-\pi}^\pi f_\rho'\Delta x_\alpha' \pa^{2}\frac{1}{|\Delta x'|}d\alpha' d\rho'\leq C(a,\delta,||f||_{C^4}).$$
Finally we compute three derivatives of the factor $\frac{1}{|\Delta x'|}$. The terms arising from these derivatives are linear combinations of terms with the structures
\begin{align*}
& \frac{\pa^{3}\Delta x'\cdot \Delta x'}{|\Delta x'|^3},\quad\frac{\pa^{2}\Delta x'\cdot \pa\Delta x'}{|\Delta x'|^3}, \frac{\pa^{2}\Delta x'\cdot \Delta x'\pa\Delta x'\cdot \Delta x'}{|\Delta x'|^5},\\ & \frac{\pa\Delta x'\cdot \pa \Delta x'\pa\Delta x'\cdot \Delta x'}{|\Delta x'|^5},\quad
\frac{\pa\Delta x'\cdot \Delta x'\pa\Delta x'\cdot \Delta x'\pa\Delta x'\cdot \Delta x'}{|\Delta x'|^7}
\end{align*}
and then, a similar analysis we did before helps us to prove that
\begin{align*}
\int_{-\infty}^\infty\int_{-\pi}^\pi f_\rho'\Delta x'_\alpha \pa^{3}\frac{1}{|\Delta x'|} d\alpha' d\rho',
\end{align*}
is bounded in $L^2$ for a constant $C(a,\delta,||f||_{C^4}$). This concludes the proof of Lemma \ref{d3B}.\end{proof}
Thus we have proven that \eqref{Fenh3} holds. This finishes the proof of Lemma \ref{lema42}.
\end{proof}

Therefore, in order to prove that, $$\overline{F}\,:\, V\times (-1,1)\to H^{3,3}_{3, \text{odd}}(\Omega_a)$$
we just need to show that if \begin{align*} \overline{r}(-\alpha,\rho) = &\overline{r}(\alpha,\rho)\end{align*}
and
\begin{align*} \overline{r}\left(\alpha+\frac{2n\pi}{3},\rho\right) = &\overline{r}(\alpha,\rho)\end{align*} for $n\in \N$, then
\begin{align*} \overline{F}(-\alpha,\rho) =- &\overline{F}(\alpha,\rho)\end{align*}
and
\begin{align*} \overline{F}\left(\alpha+\frac{2n\pi}{3},\rho\right) = &\overline{F}(\alpha,\rho)\end{align*} for $n\in \N$. These two properties are easy to check.

The last part of this section will be to check that the hypothesis 1 in the C-R theorem holds. This fact is a consequence of radial functions being stationary solutions of the SQG equation but let us check it on \eqref{functionalF}. If we take $\overline{r}=0$, i.e., $r=\rho$, the only term in \eqref{functionalF} that is not trivially zero is the last integral. In order to check that this integral is zero we just notice that the integrand is odd in $\alpha$.

\subsection{Step 2. The partial derivatives of the functional $F$}

We need to prove the existence and the continuity of the Gateaux derivatives $\pa_{\overline{r}} \overline{F}[\overline{r},\lambda]$,$\pa_\lambda \overline{F}[\overline{r},\lambda]$ and $\pa^{2}_{\overline{r}\lambda}\overline{F}[\overline{r},\lambda]$. We have the following lemma

\begin{lemma}\label{partialderivatives} For all $\overline{r}\in V^\delta$ and $\mu\in \RR$ the partial derivatives  $\pa_{\overline{r}} \overline{F}[\overline{r},\lambda]$,$\pa_\lambda \overline{F}[\overline{r},\lambda]$ and $\pa^{2}_{\overline{r}\lambda}\overline{F}[\overline{r},\lambda]$ exist and are continuous. In addition
\begin{align*}
&\partial_r\overline{F}[0,\mu]\tilde{r}(\alpha,\rho)=\partial_rF[\rho,\lambda]\tilde{r}(\alpha,\rho)\\&=\lambda \tilde{r}_{\al}(\al,\rho)  - \frac{1}{2\pi} \int \int \frac{f_\rho(\rho')\cos(\al-\al')(\tilde{r}_{\al}(\al',\rho')-\tilde{r}_{\al}(\al,\rho))}{\sqrt{\rho^2+(\rho')^2-2\rho\rho'\cos(\al-\al')}} d\al' d\rho' \\
& + \frac{\tilde{r}_{\al}(\al,\rho)}{2\pi} \int \int \frac{f_\rho(\rho')\cos(\al-\al')(\rho'-\rho)}{\rho\sqrt{\rho^2+(\rho')^2-2\rho\rho'\cos(\al-\al')}} d\al' d\rho' \\
& - \frac{1}{2\pi} \int \int \frac{f_\rho(\rho')(\rho-\rho'\cos(\al-\al'))\sin(\al-\al')(\rho\tilde{r}(\al',\rho') - \rho'\tilde{r}(\al,\rho))}{(\rho^2+(\rho')^2-2\rho\rho'\cos(\al-\al'))^{\frac32}} d\al' d\rho' \\
&\equiv I_{0}[\tilde{r},\lambda](\alpha,\rho)  + I_{1}[\tilde{r}](\alpha,\rho) + I_{2}[\tilde{r}](\alpha,\rho) + I_{3}[\tilde{r}](\alpha,\rho).
\end{align*}
\end{lemma}
\begin{proof} The lemma is trivial for the derivatives involving $\lambda$.  The continuity of the derivative with respect to $\overline{r}$ also follows since $f\in C^4$ and is compactly supported.
\end{proof}

\subsection{Step 3. Analysis of the linear operator} \label{checking3}
Now we have to study the dimension of both the kernel and image of the operator $\pa_r F[\rho,\lambda]$. We will first show that for a certain value of $\lambda$ that we will call $\lambda_3$ the kernel of the operator $\pa_r F[\rho,\lambda_3]$ is one dimensional. After that we will  show that the codimension of the image of $\pa_r F[\rho,\lambda_3]$ is also one dimensional. This will finish the checking of the hypothesis 3 in the C-R theorem. Propositions \ref{kernel} and \ref{codimension} are the main results of this section.

\begin{prop}\label{kernel}There exists a  pair $(\tilde{r}_0,\lambda_3)\in H^{4,3}_{3,\text{even}}(\Omega_a)\times \RR$, with $\tilde{r}_0$ not identically zero, such that
\begin{align}\label{kerneleq}
\pa_r F[\rho,\lambda_3]\tilde{r}_0(\alpha,\rho)=0.
\end{align}
Moreover $\tilde{r}_0$ is unique modulo multiplication by constant.
\end{prop}

\begin{proof} The proof of Proposition \ref{kernel} consists of the following steps:
\begin{enumerate}
\item \textbf{The equation for the  radial part.} We introduce in \eqref{kerneleq} the $m$-fold ansatz:
\begin{align*}
 \tilde{r}(\al,\rho) = \rho B(\rho) \cos(3m\al)
\end{align*}
and we obtain an equation for the pair $(B(\rho),\lambda)$, which we will write in the following form:
\begin{align}\label{lineal}\Theta^{m} B-\lambda B=0\quad \text{in $(1-a,1)$}.\end{align}(See the equation \eqref{simplificada} below).
\item \textbf{Existence of solutions of equation \eqref{kerneleq}.} We solve the equation \eqref{lineal} for $m=3$  and find a solution   $(B^3, \lambda_3)\in H^3((1-a,1))\times\RR$ of \eqref{lineal}. Therefore $\tilde{r}_0=\rho B^3(\rho)\cos(3\alpha)\in H^{4,3}_{3,\text{even}}(\Omega_a)$ satisfies \eqref{kerneleq}.
\item \textbf{Uniqueness for the equation \eqref{kerneleq}.} We notice that we still need to show uniqueness  for \eqref{kerneleq}, since, until now, we have that, given $\lambda_3$ there is a unique $B^3$ such that \eqref{lineal} holds. But this fact does not imply that there is only one solution (modulo multiplication by constants), $\tilde{r}_0$, to \eqref{kerneleq}. Indeed, we need to show that the equation $$\pa_r F[\rho,\lambda_3]\left( b^{3m}(\rho)\cos(3m\alpha)\right)=0\quad \text{for $m>1$}$$
    implies $b^{3m}(\rho)=0$ for $m>1$.
\end{enumerate}
\subsubsection{The equation for the radial part}
Taking $\tilde{r}(\alpha, \rho)=\rho B(\rho)\cos(m\alpha)$ we have that $\pa_r F[\rho,\lambda]\tilde{r}(\alpha, \rho)$ is given by the following terms:

\begin{align*}
I_{0}[\tilde{r},\lambda](\alpha,\rho) & = -\lambda m  \rho B(\rho) \sin(m\al) \\
I_{2}[\tilde{r}](\alpha,\rho) & = -\frac{m}{2\pi} B(\rho)\sin(m\al) \inti \intpi \frac{f_\rho(\rho')\cos(\al-\al')(\rho'-\rho)}{\sqrt{\rho^2+(\rho')^2-2\rho\rho'\cos(\al-\al')}} d\al' d\rho' \\
& = -\frac{m}{2\pi} B(\rho)\sin(m\al) \inti f_\rho(\rho')(\rho'-\rho) \left(\intpi \frac{\cos(x)}{\sqrt{\rho^2+(\rho')^2-2\rho\rho'\cos(x)}} dx\right) d\rho' \\
\end{align*}

We move on to $I_1[\tilde{r},\lambda](\alpha,\rho)$. We have that

\begin{align*}
 \tilde{r}_{\al}(\al',\rho') - \tilde{r}_{\al}(\al,\rho) & = -m(\rho'B(\rho')\sin(m\al') - \rho B(\rho) \sin(m\al)) \\
& = -m(\rho'B(\rho')\sin(m\al)\cos(m(\al-\al'))\\& - \underbrace{\rho'B(\rho')\cos(m\al)\sin(m(\al-\al'))}_{\text{will integrate to } 0} - \rho B(\rho) \sin(m\al)).
\end{align*}

Therefore

\begin{align*}
 I_1[\tilde{r}](\alpha,\rho) & = -\frac{m}{2\pi}\sin(m\al) \inti \intpi f_\rho(\rho') \frac{(\rho B(\rho) - \rho' B(\rho') \cos(mx))\cos(x)}{\sqrt{\rho^2+(\rho')^2-2\rho\rho'\cos(x)}}dx d\rho'
\end{align*}

Finally, we develop $I_3[\tilde{r}](\alpha,\rho)$. Using:

\begin{align*}
 \rho \tilde{r}(\al',\rho') - \rho' \tilde{r}(\al,\rho) & = \rho \rho'(B(\rho')\cos(m\al') - B(\rho)\cos(m\al)) \\
& = \rho \rho'(B(\rho')\sin(m(\al-\al'))\sin(m\al)\\ & + \underbrace{B(\rho')\cos(m(\al-\al'))\cos(m\al)  - B(\rho)\cos(m\al)}_{\text{will integrate to } 0})
\end{align*}

This implies that

\begin{align*}
 I_3[\tilde{r}](\alpha,\rho) & = -\frac{1}{2\pi} \sin(m\al) \inti f_\rho(\rho')B(\rho') \left(\intpi \frac{\sin(mx)(\rho-\rho'\cos(x))\rho\rho'\sin(x)}{(\rho^2+(\rho')^2-2\rho\rho'\cos(x))^{\frac32}} dx\right) d\rho'
\end{align*}

Integrating by parts, using that

\begin{align*}
 \frac{\rho\rho'\sin(x)}{(\rho^2+(\rho')^2-2\rho\rho'\cos(x))^{\frac32}} = - \pa_{x}\left(\frac{1}{(\rho^2+(\rho')^2-2\rho\rho'\cos(x))^{\frac12}}\right)
\end{align*}

we get that $I_3[\tilde{r}](\alpha,\rho)$ is given by:

\begin{align*}
 I_3[\tilde{r}](\alpha,\rho) & = -\frac{m}{2\pi} \sin(m\al) \inti f_\rho(\rho')B(\rho') \left(\intpi \frac{\cos(mx)(\rho-\rho'\cos(x))}{\sqrt{\rho^2+(\rho')^2-2\rho\rho'\cos(x)}} dx\right) d\rho' \\
 &  -\frac{1}{2\pi} \sin(m\al) \inti f_\rho(\rho')\rho'B(\rho') \left(\intpi \frac{\sin(mx)\sin(x)}{\sqrt{\rho^2+(\rho')^2-2\rho\rho'\cos(x)}} dx\right) d\rho' \\
\end{align*}

Putting all the pieces together and dividing by $\sin(m\al)$, the equation we want to solve is:

\begin{align}\label{sinsimplificar}
 & B(\rho)\left( -\lambda m \rho -\frac{m}{2\pi} \inti f_\rho(\rho')\rho' \left(\intpi \frac{\cos(x)}{\sqrt{\rho^2+(\rho')^2-2\rho\rho'\cos(x)}} dx\right) d\rho'\right) \\
& + \frac{2m}{2\pi} \inti f_\rho(\rho')\rho'B(\rho') \left(\intpi \frac{\cos(mx)\cos(x)}{\sqrt{\rho^2+(\rho')^2-2\rho\rho'\cos(x)}} dx\right) d\rho'\nonumber \\
& -\frac{m}{2\pi} \inti f_\rho(\rho')\rho B(\rho') \left(\intpi \frac{\cos(mx)}{\sqrt{\rho^2+(\rho')^2-2\rho\rho'\cos(x)}} dx\right) d\rho'\nonumber \\
 &  -\frac{1}{2\pi} \inti f_\rho(\rho')\rho'B(\rho') \left(\intpi \frac{\sin(mx)\sin(x)}{\sqrt{\rho^2+(\rho')^2-2\rho\rho'\cos(x)}} dx\right) d\rho' = 0.\nonumber
\end{align}

The inner integrals can be explicitly calculated in terms of EllipticE and EllipticK functions for any $m$. We can simplify the equation \eqref{sinsimplificar} in the following way. Letting $s(\rho,\rho') = \frac{\rho}{\rho'}$, we obtain

\begin{align*}
\frac{1}{\sqrt{\rho^2 + (\rho')^{2} - 2\rho \rho' \cos(x)}} = \frac{1}{\rho'} \frac{1}{\sqrt{1+\left(\frac{\rho}{\rho'}\right)^{2}-2\frac{\rho}{\rho'}\cos(x)}} = \frac{1}{\rho'} \frac{1}{\sqrt{1+s^2-2s\cos(x)}},
\end{align*}

thus equation \eqref{sinsimplificar} reads:

\begin{align}\label{sinsimplificar2}
 & B(\rho)\left( -\lambda \rho -\frac{1}{2\pi} \inti f_\rho(\rho')\left(\intpi \frac{\cos(x)}{\sqrt{1+s^2-2s\cos(x)}} dx\right) d\rho'\right) \\
& + \frac{2}{2\pi} \inti f_\rho(\rho')B(\rho') \left(\intpi \frac{\cos(mx)\cos(x)}{\sqrt{1+s^2-2s\cos(x)}} dx\right) d\rho'\nonumber \\
& -\frac{1}{2\pi} \inti f_\rho(\rho')\frac{\rho}{\rho'} B(\rho') \left(\intpi \frac{\cos(mx)}{\sqrt{1+s^2-2s\cos(x)}} dx\right) d\rho'\nonumber \\
 &  -\frac{1}{2m\pi} \inti f_\rho(\rho')B(\rho') \left(\intpi \frac{\sin(mx)\sin(x)}{\sqrt{1+s^2-2s\cos(x)}} dx\right) d\rho' = 0.\nonumber
\end{align}

We focus on the term

\begin{align*}
\int_{-\pi}^{\pi} \frac{2\cos(mx)\cos(x)-s\cos(mx)-\frac{1}{m}\sin(mx)\sin(x)}{\sqrt{1+s^2-2s\cos(x)}}dx \equiv T(s).
\end{align*}

We remark that $\cos(mx) = \frac{1}{m} \pa_{x} \sin(mx)$. This implies, on the one hand:

\begin{align*}
\int_{-\pi}^{\pi} \frac{2\cos(mx)\cos(x)}{\sqrt{1+s^2-2s\cos(x)}}dx
& = \frac{1}{m} \int_{-\pi}^{\pi} \frac{2\cos(x)\pa_{x}\sin(mx)}{\sqrt{1+s^2-2s\cos(x)}}dx \\
& = \frac{1}{m} \int_{-\pi}^{\pi} \frac{2\sin(x)\sin(mx)}{\sqrt{1+s^2-2s\cos(x)}}dx
+ \frac{1}{m} \int_{-\pi}^{\pi} \frac{2s\cos(x)\sin(mx)\sin(x)}{(1+s^2-2s\cos(x))^{3/2}}dx \\
\end{align*}

On the other

\begin{align*}
s\int_{-\pi}^{\pi} \frac{\cos(mx)}{\sqrt{1+s^2-2s\cos(x)}}dx
= \frac{s^2}{m}\int_{-\pi}^{\pi} \frac{\sin(mx)\sin(x)}{(1+s^2-2s\cos(x))^{3/2}}dx
\end{align*}

Therefore, $T(s)$ can be transformed into

\begin{align*}
T(s) & = \frac{1}{m} \int_{-\pi}^{\pi} \frac{\sin(x)\sin(mx)}{\sqrt{1+s^2-2s\cos(x)}}dx
+ \frac{1}{m} \int_{-\pi}^{\pi} \frac{2s\cos(x)\sin(mx)\sin(x)}{(1+s^2-2s\cos(x))^{3/2}}dx
\\ & -\frac{s^2}{m}\int_{-\pi}^{\pi} \frac{\sin(mx)\sin(x)}{(1+s^2-2s\cos(x))^{3/2}}dx
\\ & = \frac{1}{m} \int_{-\pi}^{\pi} \frac{\sin(mx)\sin(x)}{(1+s^2-2s\cos(x))^{3/2}}dx
= \frac{1}{s}\int_{-\pi}^{\pi} \frac{\cos(mx)}{\sqrt{1+s^2-2s\cos(x)}}dx.
\end{align*}

Substituting into \eqref{sinsimplificar2}, we have to solve:

\begin{align}\label{sinsimplificar3}
 & B(\rho)\left( -\lambda \rho -\frac{1}{2\pi} \inti f_\rho(\rho')\left(\intpi \frac{\cos(x)}{\sqrt{1+s^2-2s\cos(x)}} dx\right) d\rho'\right) \\
& + \frac{1}{2\pi} \inti f_\rho(\rho')B(\rho') \frac{\rho'}{\rho}\int_{-\pi}^{\pi} \frac{\cos(mx)}{\sqrt{1+s^2-2s\cos(x)}}dx d\rho' = 0.\nonumber
\end{align}

From now on, we will call

\begin{align*}
I(\rho) & = -\frac{1}{2\pi} \inti f_\rho(\rho')\left(\intpi \frac{\cos(x)}{\sqrt{1+\left(\frac{\rho}{\rho'}\right)^2-2\left(\frac{\rho}{\rho'}\right)\cos(x)}} dx\right) d\rho' \\
T^mB(\rho) & = \frac{1}{2\pi} \inti f_\rho(\rho')B(\rho') \frac{\rho'}{\rho}\int_{-\pi}^{\pi} \frac{\cos(mx)}{\sqrt{1+\left(\frac{\rho}{\rho'}\right)^2-2\left(\frac{\rho}{\rho'}\right)\cos(x)}}dx d\rho'
\end{align*}

We will also define

\begin{align*}
K^{m}(s) & = \frac{1}{2\pi s} \int_{-\pi}^{\pi} \frac{\cos(mx)}{\sqrt{1+s^2-2s\cos(x)}}dx  \\
T^{m} B(\rho) & = \int_{-\infty}^{\infty} f_\rho(\rho') B(\rho')K^{m}\left(\frac{\rho}{\rho'}\right) d\rho'
\end{align*}

This allows us to write \eqref{sinsimplificar3} as:

\begin{align}
\label{simplificada}
\tilde{I}(\rho)B(\rho) + \tilde{T}^mB(\rho) = \lambda B(\rho),\quad \text{in $(1-a,1)$,}
\end{align}

where $\tilde{T}^m = \frac{1}{\rho} T^m$ and $\tilde{I}=\frac{1}{\rho}I(\rho)$. Thus, using the notation of \eqref{lineal},
$$\Theta^m B(\rho)\equiv  \tilde{I}(\rho)B(\rho) + \tilde{T}^mB(\rho).$$

\subsubsection{Existence of an element in the kernel of $\pa_rF[\rho,\lambda_3]$}
In this part we will study the equation \eqref{simplificada} in order to obtain an element in the kernel or $\pa_r F[\rho,\lambda_3]$ for some value $\lambda_3\in \RR$ . We shall show the following proposition:
\begin{prop} \label{btres} There exists a solution $(B^3,\lambda_3)\in H^3((1-a),1)\times \RR$ to the equation \eqref{simplificada}. In addition, $\lambda_3$ is simple.
\end{prop}
We remark that this proposition yields the next corollary:
\begin{corollary}The function $\tilde{r}_0(\alpha,\rho)=\rho B^3(\rho)\cos(3\alpha)$ belongs to $H^{4,3}_{3,\text{even}}(\Omega_a)$ and solves \eqref{kerneleq}.
\end{corollary}
\begin{proof}
The proof of Proposition \ref{btres} is divided in two parts. In the first one we deal with the operator $\tilde{T}$ in  \eqref{simplificada}. In the second one we  show existence of pair $(B^3,\lambda_3)\in H^3((1-a,1))\times \RR$ solving \eqref{simplificada} and that $\lambda_3$ is simple.
\begin{enumerate}
\item \textbf{Study of the operator $\tilde{T}^m$.}
This part is devoted to studying the operator $\tilde{T}^m$ and its derivatives until order 3. Here we recall its definition:
\begin{align*}
\tilde{T}^m B(\rho)=\frac{1}{\rho}\int_{-\infty}^{\infty} f_\rho(\rho') B(\rho')K^{m}\left(\frac{\rho}{\rho'}\right) d\rho'
\end{align*}
with
\begin{align*}
K^{m}(s) & = \frac{1}{2\pi s} \int_{-\pi}^{\pi} \frac{\cos(mx)}{\sqrt{1+s^2-2s\cos(x)}}dx .
\end{align*}
The main results here are Corollary \ref{lemacompacidad} and Lemma \ref{Tstarcompacidad} that state that the operator $\tilde{T}^m$ and its adjoint  $\tilde{T}^{m*}$ are compact operators acting from $H^k$ to $H^{k+1}$ for $k=0,1,2$.

Let's  compute the derivatives of $\tilde{T}^{3m}$.

\begin{lemma}\label{Tderivadas}Let $B\in C^3$ then the following equalities hold:
\begin{align*}
\pa_\rho \tilde{T}^{m}B(\rho)=&\tilde{T}^{m}_1\pa_\rho B (\rho)+\frac{1}{\rho}\inti \pa^{2}_\rho f(\rho') B(\rho')\frac{\rho'}{\rho}K^{m}\left(\frac{\rho}{\rho'}\right)d\rho',\\
\pa^{2}_\rho \tilde{T}^{m}B(\rho)= & T^{m}_2\pa^{2}_\rho B(\rho)+\frac{1}{\rho}\sum_{j=1}^2\inti \pa^{j+1}_\rho f(\rho')\pa^{2-j}_\rho B(\rho')\left(\frac{\rho'}{\rho}\right)^2K^{m}\left(\frac{\rho}{\rho'}\right)d\rho'\\
\pa^{3}_\rho \tilde{T}^{m}B(\rho)= & T^{m}_3\pa^{3}_\rho B(\rho)+\sum_{j=1}^3\inti \pa^{j+1}_\rho f(\rho')\pa^{3-j}_\rho B(\rho')\left(\frac{\rho'}{\rho}\right)^3K^{m}\left(\frac{\rho}{\rho'}\right)d\rho'
\end{align*}
where
\begin{align}
\tilde{T}^{m}_1 B=&\frac{1}{\rho}\inti f_\rho(\rho')B(\rho')\frac{\rho'}{\rho}K^{m}\left(\frac{1}{\gamma}\right)d\rho'\\
\tilde{T}^{m}_2B(\rho)=&\frac{1}{\rho}\inti f_\rho(\rho')B(\rho')\left(\frac{\rho'}{\rho}\right)^2K^{m}\left(\frac{\rho}{\rho'}\right)d\rho'\\
\tilde{T}^{m}_3B(\rho)=&\frac{1}{\rho}\inti f_\rho(\rho')B(\rho')\left(\frac{\rho'}{\rho}\right)^3K^{m}\left(\frac{\rho}{\rho'}\right)d\rho'.
\end{align}

\end{lemma}
\begin{proof}
 We notice that after the change of variable $\gamma=\frac{\rho'}{\rho}$ we have that
\begin{align*}
\tilde{T}^{m}B(\rho)=\inti g(\rho\gamma) K^{m}\left(\frac{1}{\gamma}\right)d\gamma
\end{align*}
where $g=f_\rho B$. Taking one derivative we have that
\begin{align*}
\pa_\rho \tilde{T}^{m}B(\rho)= & \inti (\pa_\rho g)(\gamma \rho)\gamma K^{m}\left(\frac{1}{\gamma}\right)d\gamma\\
= & \frac{1}{\rho}\inti \pa_\rho g (\rho')\frac{\rho'}{\rho}K^{m}\left(\frac{1}{\gamma}\right)d\rho'\\
=&\frac{1}{\rho}\inti \pa^{2}_\rho f(\rho') B(\rho')\frac{\rho'}{\rho}K^{m}\left(\frac{1}{\gamma}\right)d\rho+\frac{1}{\rho}\inti  f_\rho(\rho') \pa_\rho B(\rho')\frac{\rho'}{\rho}K^{m}\left(\frac{1}{\gamma}\right)d\rho',
\end{align*}
so that
\begin{align*}
\pa_\rho \tilde{T}^{m}B(\rho)=\tilde{T}^{m}_1\pa_\rho B (\rho)+\frac{1}{\rho}\inti \pa^{2}_\rho f(\rho') B(\rho')\frac{\rho'}{\rho}K^{m}\left(\frac{1}{\gamma}\right)d\rho'.
\end{align*}
Computing in a similar way, by  taking two derivatives, we have that
\begin{align*}
\pa_\rho^{2}\tilde{T}^{m}B(\rho)=&\frac{1}{\rho}\inti \pa_\rho ^{2}g(\rho')\left(\frac{\rho'}{\rho}\right)^2K^{m}\left(\frac{\rho}{\rho'}\right)d\rho'\\
\pa_\rho^{3}\tilde{T}^{m}B(\rho)=&\frac{1}{\rho}\inti \pa_\rho ^{3}g(\rho')\left(\frac{\rho'}{\rho}\right)^3K^{m}\left(\frac{\rho}{\rho'}\right)d\rho'.
\end{align*}
And we can write
\begin{align*}
\pa^{2}_\rho \tilde{T}^{m}B(\rho)= & \tilde{T}^{m}_2\pa^{2}_\rho B(\rho)+\frac{1}{\rho}\sum_{j=1}^2\inti \pa^{j+1}_\rho f(\rho')\pa^{2-j}_\rho B(\rho')\left(\frac{\rho'}{\rho}\right)^2K^{m}\left(\frac{\rho}{\rho'}\right)d\rho'\\
\pa^{3}_\rho \tilde{T}^{m}B(\rho)= & \tilde{T}^{m}_3\pa^{3}_\rho B(\rho)+\sum_{j=1}^3\inti \pa^{j+1}_\rho f(\rho')\pa^{3-j}_\rho B(\rho')\left(\frac{\rho'}{\rho}\right)^3K^{m}\left(\frac{\rho}{\rho'}\right)d\rho'.
\end{align*}
\end{proof}

Some of the properties of the operator $\tilde{T}^m$ come from the sign of its kernel. We study this sign in the following lemma.

\begin{lemma}\label{positivo}
Let $T^{m}\left(\frac{\rho}{\rho'}\right)$ be defined as:

\begin{align*}
T^{m}\left(\frac{\rho}{\rho'}\right) = \frac{1}{2\pi}\int_{-\pi}^{\pi} \frac{\cos(mx)}{\sqrt{1+\left(\frac{\rho}{\rho'}\right)^2-2\left(\frac{\rho}{\rho'}\right)\cos(x)}}dx
\end{align*}

Then, we have that, for every $(\rho,\rho') \in \mathbb{R}^{2}, \rho \neq \rho'$
\begin{enumerate}
\item $T^{m}\left(\frac{\rho}{\rho'}\right) > 0$
\item $T^{m+1}\left(\frac{\rho}{\rho'}\right) < T^{m}\left(\frac{\rho}{\rho'}\right)$.
\end{enumerate}
\end{lemma}

\begin{proof}
Let $r = \frac{\rho}{\rho'}$. We have that

\begin{align*}
T^{m}\left(\frac{\rho}{\rho'}\right) = \frac{1}{2\pi}\int_{-\pi}^{\pi} \frac{\cos(mx)}{(1+r^2-2r\cos(x))^{\frac12}}dx
\end{align*}

We do first the $r < 1$ case. The $r > 1$ case follows from the property

\begin{align*}
T^{m}\left(\frac{\rho}{\rho'}\right) = \left(\frac{\rho}{\rho'}\right)^{3}T^{m}\left(\frac{\rho'}{\rho}\right)
\end{align*}

\begin{align*}
T^{m}\left(\frac{\rho}{\rho'}\right) & = \frac{1}{1+r} \int_{-\pi}^{\pi} \frac{\cos(mx)}{\left(1 - \frac{4r}{(1+r)^2}\cos^{2}\left(\frac{x}{2}\right)\right)^{\frac12}}
= \frac{1}{2\pi}\frac{2}{1+r} \int_{-\frac{\pi}{2}}^{\frac{\pi}{2}} \frac{\cos(2mx)}{\left(1 - \frac{4r}{(1+r)^2}\cos^{2}\left(x\right)\right)^{\frac12}} \\
& = \frac{1}{2\pi}\frac{4}{1+r} \sum_{k=0}^{\infty}\int_{0}^{\frac{\pi}{2}} \cos(2mx)\left(\frac{4r}{(1+r)^2}\right)^{k} \cos^{2k}\left(x\right)\frac{1}{k!}\left(1/2\right)_{k} dx \\
& = \frac{1}{2\pi}\frac{4}{1+r} \sum_{k=0}^{\infty}\left(\frac{4r}{(1+r)^2}\right)^{k} \frac{1}{k!}\left(1/2\right)_{k} \int_{0}^{\frac{\pi}{2}} \cos(2mx) \cos^{2k}\left(x\right) dx \\
& = \frac{1}{2\pi} \frac{4}{1+r} \sum_{k=m}^{\infty}\left(\frac{4r}{(1+r)^2}\right)^{k} \frac{1}{k!}\left(1/2\right)_{k} \frac{\pi}{2^{2k+1}} \frac{\Gamma(2k+1)}{\Gamma(1+k+m)\Gamma(1+k-m)} \\
& = \sum_{k=m}^{\infty}A_{k} \frac{1}{\Gamma(1+k+m)\Gamma(1+k-m)}. \\
\end{align*}

Since $A_{k} > 0$ for every $k$, this shows the first item. Next, we compute for $r < 1$:

\begin{align*}
T^{m}\left(\frac{\rho}{\rho'}\right) - T^{m+1}\left(\frac{\rho}{\rho'}\right) & = \frac{A_{m}}{\Gamma(1+2m)}\\ & + \sum_{k=m+1}^{\infty}A_{k} \left(\frac{1}{\Gamma(1+k+m)\Gamma(1+k-m)} - \frac{1}{\Gamma(2+k+m)\Gamma(k-m)}\right) \\
& = \frac{A_{m}}{\Gamma(1+2m)} + \sum_{k=m+1}^{\infty}\frac{A_k}{\Gamma(1+k+m)\Gamma(1+k-m)} \left( \frac{1+2m}{1+k+m}\right) > 0.\\
\end{align*}

This completes the proof of the lemma.
\end{proof}

In order to prove the compactness of the operator $\tilde{T}^m$ we will use the following decomposition:

\begin{lemma}\label{log+c1}The function $T^m(s)$ satisfies
$$T^{m}(s)=-\frac{2}{\sqrt{s}}\log(|1-s|)+E^m(s)$$
where $E^m(s)$ is a $C^1-$function
with $$||E^m||_{C^1}\leq C(m).$$
\end{lemma}
\begin{proof}
We will split $T^m(s)$ in two parts
\begin{align*}
T^{m}(s)=&\int_{-\pi}^\pi \frac{1}{\sqrt{(1-s)^2+4s\sin^2\left(\frac{x}{2}\right)}}dx +\int_{-\pi}^\pi \frac{\cos(mx)-1}{(1-s)^2+4s\sin^2\left(\frac{x}{2}\right)}dx\\
\\ \equiv & T_1(s)+T^{m}_2(s).
\end{align*}
We now focus on the term $T_1(s)$. By making the change $y=\sin\left(\frac{x}{2}\right)$ yields
\begin{align*}
T_1(s)=&4\int_{0}^1 \frac{1}{\sqrt{(1-s)^2+4s\sin^2\left(\frac{x}{2}\right)}}\frac{dy}{\sqrt{1-y^2}}\\
&=\frac{2}{\sqrt{s}}\int_{0}^\frac{1}{\ep}\frac{1}{\sqrt{1+z^2}}\frac{dz}{\sqrt{1-\ep^2z^2}},
\end{align*}
with $\ep=\frac{|1-s|}{2\sqrt{s}}.$ We will break  $T_1(s)$ in two parts
\begin{align*}
T_1(s)=&\frac{2}{\sqrt{s}}\int_{0}^\frac{1}{\ep}\frac{1}{\sqrt{1+z^2}}\left(\frac{1}{\sqrt{1-\ep^2z^2}}-1\right)dz\\
+& \frac{2}{\sqrt{s}}\int_{0}^\frac{1}{\ep}\frac{1}{\sqrt{1+z^2}}dz\\
\equiv & T_{11}(s)+T_{12}(s).
\end{align*}
The integral in the term $T_{12}(s)$ can be computed exactly. We obtain that
\begin{align}\label{T12}
T_{12}(s)=\frac{2}{\sqrt{s}}\arcsinh\left(\frac{1}{\ep}\right),
\end{align}
where we recall that
\begin{align*}
\arcsinh(x)=\log\left(x+\sqrt{x^2+1}\right)
\end{align*}
and then,
\begin{align*}
\arcsinh\left(\frac{1}{\ep}\right)=\log\left(\frac{2\sqrt{s}}{|1-s|}+\sqrt{\frac{4s}{|1-s|^2}+1}\right)=-\log\left(|1-s|\right)+2\log(1+\sqrt{s}).
\end{align*}
Finally,
\begin{align*}
T_{12}(s)=&-\frac{2}{\sqrt{s}}\log(|1-s|)+\frac{4}{\sqrt{s}}\log(1+\sqrt{s})=-\frac{2}{\sqrt{s}}\log(|1-s|)+S(s).
\end{align*}
where $S(s)$ is an smooth function.  Next we will show that the first derivative of the function $T_{11}(s)$ is continuous.  We notice that undoing the change of variable we can write
\begin{align*}
T_{11}(s)&=4\int_0^{1}\frac{1}{\sqrt{(1-s)^2+4sy^2}}\left(\frac{1}{\sqrt{1-y^2}}-1\right)dy\\&=4\int_0^{1}\frac{1}{\sqrt{(1-s)^2+4sy^2}}\left(\frac{y^2}{\sqrt{1-y^2}
\left(1+\sqrt{1-y^2}\right)}\right)dy
\end{align*}
Since the function $$\frac{y^2}{\sqrt{(s-1)^2+4sy^2}}$$
is in $L^\infty$, by the dominated convergence theorem (DCT),   $T_{11}(s)$ is continuous at $s=1$. In addition, for $s\neq 1$, we can differentiate to get that in a weak sense
\begin{align*}
\pa_s T_{11}(s)=&-4\int_{0}^1\frac{(s-1)+2y^2}{\left((1-s)^2+4sy^2\right)^\frac{3}{2}}\left(\frac{y^2}{\sqrt{1-y^2}
\left(1+\sqrt{1-y^2}\right)}\right)dy\\
&=T_{111}(s)+T_{112}(s)
\end{align*}
where
\begin{align*}
T_{111}(s)=-4\int_{0}^1 \frac{(s-1)}{\left((1-s)^2+4sy^2\right)^\frac{3}{2}}\left(\frac{y^2}{\sqrt{1-y^2}
\left(1+\sqrt{1-y^2}\right)}\right)dy
\end{align*}
and
\begin{align*}
T_{112}(s)=-8\int_{0}^1\frac{y^4}{\left((1-s)^2+4sy^2\right)^\frac{3}{2}}\left(\frac{1}{\sqrt{1-y^2}
\left(1+\sqrt{1-y^2}\right)}\right)dy
\end{align*}
where we can prove that $T_{112}(s)$ is a continuous function at $s=1$ by DCT. To analyze $T_{111}(s)$ we split this term into two parts
\begin{align*}
T_{111}(s)=&-2\int_0^1 \frac{(s-1)y^2}{\left((1-s)^2+4sy^2\right)^\frac{3}{2}}dy\\& -4\int_{0}^1 \frac{(s-1)y^2}{\left((1-s)^2+4sy^2\right)^\frac{3}{2}}\left(\frac{1}{\sqrt{1-y^2}
\left(1+\sqrt{1-y^2}\right)}-\frac{1}{2}\right)dy.
\end{align*}
The second integral is continuous at $s=1$ again by DCT. The first integral can be computed analytically. We obtain that
\begin{align*}
-2\int_0^1 \frac{(s-1)y^2}{\left((1-s)^2+4sy^2\right)^\frac{3}{2}}dy=\frac{1-s}{3s}\left(-\frac{2\sqrt{s}}{1+s}+\arcsinh\left(\frac{2\sqrt{s}}{|1-s|}\right)\right),
\end{align*}
to show that $T_{111}(s)$ is also a continuous function at $s=1$. Therefore we have that $$T_1(s)=-\frac{2}{\sqrt{s}}\log(|1-s|)+E(s)$$ where $E(s)$ is a $C^1$ function.

For $T^m_2(s)$ we use Taylor's integral remainder formula to show that $$\cos(mx)-1=-\frac{m^2x^2}{2}+m^4x^4\int_{0}^1\gamma^3\int_{0}^1\mu^2\int_{0}^1\nu\int_{0}^1  \cos(m\gamma\mu\nu\tau x)d\nu d\mu d\gamma d\tau \equiv \frac{m^2x^2}{2}
+ m^4x^4 R^m(x),$$
and that
\begin{align*}
2\sin^2\left(\frac{x}{2}\right)=(1-\cos(x))=\frac{x^2}{2}-x^4R^1(x)
\end{align*}
Therefore
\begin{align*}
\cos(mx)-1=2m^2\sin^2\left(\frac{x}{2}\right)+x^4\tilde{R}^m(x).
\end{align*}
where $\tilde{R}^m(x)$ is a bounded function. We now can write $T^m_2(s)$ as follows
\begin{align*}
T^m_2(s)=&\int_{-\pi}^\pi \frac{2m^2\sin^2\left(\frac{x}{2}\right)}{\sqrt{(1-s)^2+4s\sin^2\left(\frac{x}{2}\right)}}dx+\int_{-\pi}^\pi \frac{x^4\tilde{R}^m(x)}{\sqrt{(1-s)^2+4s\sin^2\left(\frac{x}{2}\right)}}dx
\end{align*}

By DCT $T^m_2(s)$ is continuous at $s=1$. For $s\neq 1$ we differentiate to get
\begin{align*}
\pa_s T^m_2(s)=&-\intpi \frac{(s-1)+2\sin^2\left(\frac{x}{2}\right)}{\left((1-s)^2+4s\sin^2\left(\frac{x}{2}\right)\right)^\frac{3}{2}}\sin^2\left(\frac{x}{2}\right)dx\\&
-\intpi \frac{(s-1)+2\sin^2\left(\frac{x}{2}\right)}{\left((1-s)^2+4s\sin^2\left(\frac{x}{2}\right)\right)^\frac{3}{2}}x^4\tilde{R}^m(x)dx.
\end{align*}
The second integral on the right hand side is a continuous function. In the first one we make the change of variables $y=\sin\left(\frac{x}{2}\right)$ to get that
\begin{align*}
\intpi \frac{(s-1)+2\sin^2\left(\frac{x}{2}\right)}{\left((1-s)^2+4s\sin^2\left(\frac{x}{2}\right)\right)^\frac{3}{2}}\sin^2\left(\frac{x}{2}\right)dx=4\int_0^1 \frac{(s-1)+2y^2}{\left((1-s)^2+4sy^2\right)^\frac{3}{2}}\frac{y^2}{\sqrt{1-y^2}}dy.
\end{align*}
Here, the difficulty to show continuity comes from the integral
\begin{align*}
\int_0^1 \frac{(s-1)}{\left((1-s)^2+4sy^2\right)^\frac{3}{2}}\frac{y^2}{\sqrt{1-y^2}}dy
\end{align*}
but it has been already proven that this integral is continuous in the analysis of $T_{111}(s)$. This concludes the proof of the lemma.
\end{proof}
We also need to study the derivatives of the function $\tilde{T}^{3*}B(\rho)$. We start with the following lemma:

\begin{lemma}\label{c2}The function $(1-s)^3T^3(s)$ is $C^2$ in a neighbourhood of  $s=1$.
\end{lemma}

\begin{proof}
In order to prove the lemma we will split $T^3(s)$ in different parts and we will deal with every one separately. We first notice that we can write
\begin{align*}
T^{3}(s)&=\int_{-\pi}^{\pi} \frac{\cos(3x)}{\sqrt{(1-s)^2+4s\sin\left(\frac{x}{2}\right)^2}}dx
\end{align*}
We will use the formula
\begin{align*}
\cos(3x)=1-18\sin^2\left(\frac{x}{2}\right)+48\sin^4\left(\frac{x}{2}\right)-32\sin^6\left(\frac{x}{2}\right)
\end{align*}
and make the change $y=\sin\left(\frac{x}{2}\right)$ to get
\begin{align*}
T^{3}(s)=&2\int_{-1}^{1}\frac{1-18y^2+48y^4-32y^6}{\sqrt{(1-s)^2+4sy^2}}\frac{dy}{\sqrt{1-y^2}}\\
=& 4\int_{0}^{1}\frac{1-18y^2+48y^4-32y^6}{\sqrt{(1-s)^2+4sy^2}}\frac{dy}{\sqrt{1-y^2}}
\end{align*}
In addition we change again of variable by making $z=\frac{2\sqrt{s}}{|1-s|}y$ so that
\begin{align*}
T^{3}(s)=&\frac{2}{\sqrt{s}}\int_{0}^\frac{1}{\ep}\frac{1-18\ep^2z^2+48\ep^4z^4-32\ep^6z^6}{\sqrt{1+z^2}}\frac{dz}{\sqrt{1-\ep^2z^2}},
\end{align*}
where $\ep=\frac{|1-s|}{2\sqrt{s}}$. Now we define
\begin{align*}
T_1(s)=&\frac{2}{\sqrt{s}}\int_{0}^\frac{1}{\ep}\frac{1}{\sqrt{1+z^2}}\frac{dz}{\sqrt{1-\ep^2z^2}}\\
T_2^{3}(s)=&\frac{2}{\sqrt{s}}\int_{0}^\frac{1}{\ep}\frac{-18\ep^2z^2+48\ep^4z^4-32\ep^6z^6}{\sqrt{1+z^2}}\frac{dz}{\sqrt{1-\ep^2z^2}}.
\end{align*}
Thus $T^{3}(s)=T_1(s)+T^3_2(s)$ and $T_1(s)$ is the same function that in the proof of Lemma \ref{log+c1}. Then we know that $T_1(s)=T_{11}(s)+T_{12}(s)$, where
$$T_{12}(s)=-\frac{2}{\sqrt{s}}\log(|1-s|)+S(s),$$ with $S(s)$ a smooth function. Therefore $(1-s)^3T_{12}(s)$ is a $C^2-$function.
Also we know that $T_{11}(s)$ is a $C^1-$ function. In order to analyze two derivatives of $T_{11}(s)$ we differentiate for $s\neq 1$ to get
\begin{align*}
\pa^2_s T_{11}(s)=&-4\int_{0}^1\frac{1}{\left((1-s)+4sy^2\right)^\frac{3}{2}}\left(\frac{y^2}{\sqrt{1-y^2}
\left(1+\sqrt{1-y^2}\right)}\right)dy
\\&+12\int_{0}^1\frac{\left((s-1)^2+4y^2\right)^2}{\left((1-s)^2+4sy^2\right)^\frac{5}{2}}\left(\frac{y^2}{\sqrt{1-y^2}
\left(1+\sqrt{1-y^2}\right)}\right)dy.
\end{align*}
Multiplying by $(1-s)$ we have that
\begin{align*}
(1-s)\pa^{2}_{s} T^3_{11}(s)=-T^3_{111}(s)+6\int_{0}^1(1-s)\frac{\left((s-1)+4y^2\right)^2}{\left((1-s)^2+4sy^2\right)^\frac{5}{2}}\left(\frac{y^2}{\sqrt{1-y^2}
\left(1+\sqrt{1-y^2}\right)}\right)dy,
\end{align*}
where we have already checked that $T^3_{111}(s)$ is continuous. Next we check that this function is also continuous. This will be a consequence of the continuity of the following terms
\begin{align*}
\int_{0}^1(1-s)\frac{(s-1)^2}{\left((1-s)^2+4sy^2\right)^\frac{5}{2}}\left(\frac{y^2}{\sqrt{1-y^2}
\left(1+\sqrt{1-y^2}\right)}\right)dy\\
\int_{0}^1(1-s)\frac{(s-1)y^2}{\left((1-s)^2+4sy^2\right)^\frac{5}{2}}\left(\frac{y^2}{\sqrt{1-y^2}
\left(1+\sqrt{1-y^2}\right)}\right)dy\\
\int_{0}^1(1-s)\frac{y^4}{\left((1-s)^2+4sy^2\right)^\frac{5}{2}}\left(\frac{y^2}{\sqrt{1-y^2}
\left(1+\sqrt{1-y^2}\right)}\right)dy.
\end{align*}
The last term is continuous just by applying DCT. The other two terms can be treated in a similar way we did before to show continuity.

In addition the analysis of $T^3_{2}(s)$ does not introduce any new difficulty and we will not give the details here. This concludes the proof of the lemma.
\end{proof}

\begin{lemma}\label{commutator}The function $(1-s)T^m(s)$ satisfies
\begin{align*}
\pa_s \left((1-s)T^m(s)\right)=-T^m(s)+E^{m}(s)
\end{align*}
where $E^{m}(s)$ is a continuous function with $L^\infty-$norm independent on $m$.
\end{lemma}
\begin{proof}
By Lemma \ref{log+c1} $(1-s)T^m(s)$ is a continuous function. For $s\neq 1$ we can differentiate to get
\begin{align*}
\pa_s \left((1-s)T^m(s)\right)=-T^m(s)+E^{m}(s)
\end{align*}
with
\begin{align*}
E^m(s)=\int_{-\pi}^\pi\frac{-(1-s)^2+2(1-s)\sin^2\left(\frac{x}{2}\right)}{\left((1-s)^2+2\sin^2\left(\frac{x}{2}\right)\right)^\frac{3}{2}}\cos(mx)dx.
\end{align*}
In addition we have that
\begin{align*}
\left|E^m(s)\right|\leq \int_{-\pi}^{\pi} \frac{(1-s)^2+2|1-s|\sin^2\left(\frac{x}{2}\right)}{\left((1-s)^2+s\sin^2\left(\frac{x}{2}\right)\right)^\frac{3}{2}}dx.
\end{align*}
For the first integral we have that
\begin{align}\label{a1}
\int_{-\pi}^{\pi} &\frac{(1-s)^2}{\left((1-s)^2+4s\sin^2\left(\frac{x}{2}\right)\right)^\frac{3}{2}}dx=\frac{2}{\sqrt{s}}\int_{0}^\frac{2\sqrt{s}}{|1-s|}\frac{1}{\left(1+y^2\right)^\frac{3}{2}}
\frac{dy}{\sqrt{1-\frac{|1-s|^2}{2s}y^2}}\nonumber\\
&=\frac{2}{\sqrt{s}}\int_{0}^\frac{2\sqrt{s}}{|1-s|}\frac{1}{\left(1+y^2\right)^\frac{3}{2}}
dy +\frac{2}{\sqrt{s}}\int_{0}^\frac{2\sqrt{s}}{|1-s|}\frac{1}{\left(1+y^2\right)^\frac{3}{2}}
\left(\frac{1}{\sqrt{1-\frac{|1-s|^2}{2s}y^2}}-1\right)dy\nonumber\\
&=\frac{2}{\sqrt{s}}\int_{0}^\frac{2\sqrt{s}}{|1-s|}\frac{1}{\left(1+y^2\right)^\frac{3}{2}}
dy +4\int_{0}^1\frac{|1-s|}{\left((1-s)^2+4sy^2\right)^\frac{3}{2}}
\left(\frac{1}{\sqrt{1-y^2}}-1\right)dy.
\end{align}
In addition, the second integral, can be written
\begin{align}\label{a2}
&\int_{-\pi}^{\pi} \frac{2|1-s|\sin^2\left(\frac{x}{2}\right)}{\left((1-s)^2+2\sin^2\left(\frac{x}{2}\right)\right)^\frac{3}{2}}dx=
8\int_{0}^1\frac{|1-s|y^2}{\left(|1-s|^2+4sy^2\right)^\frac{3}{2}}\frac{dy}{\sqrt{1-y^2}}\nonumber\\
&=\frac{|1-s|}{s^\frac{3}{2}}\int_{0}^\frac{2\sqrt{s}}{|1-s|}\frac{y^2}{(1+y^2)^\frac{3}{2}}\frac{dy}{\sqrt{1-\frac{|1-s|^2}{4s}y^2}} =\frac{|1-s|}{s^\frac{3}{2}}\int_{0}^\frac{2\sqrt{s}}{|1-s|}\frac{y^2}{(1+y^2)^\frac{3}{2}}dy\nonumber\\
&+\frac{|1-s|}{s^\frac{3}{2}}\int_{0}^\frac{2\sqrt{s}}{|1-s|}\frac{y^2}{(1+y^2)^\frac{3}{2}}\left(\frac{1}{\sqrt{1-\frac{|1-s|^2}{4s}y^2}}-1\right)dy\nonumber\\
&\frac{|1-s|}{s^\frac{3}{2}}\int_{0}^\frac{2\sqrt{s}}{|1-s|}\frac{y^2}{(1+y^2)^\frac{3}{2}}dy+8\int_{0}^1\frac{|1-s|y^2}{\left(|1-s|^2+4sy^2\right)^\frac{3}{2}}\left(\frac{1}{\sqrt{1-y^2}}-1\right)dy.
\end{align}
From expression \eqref{a1}, \eqref{a2} and DCT it is easy to achieve the conclusion of the lemma.
\end{proof}

Now we can state and prove the main lemmas of this section.

\begin{lemma}\label{lemacompacidad}
 Let $f \in C^{4}(\mathbb{R}^{+})$ and $m \geq 1$. Then, $\tilde{T}^{m}$ and $\tilde{T}^m_i$ , with $i=1,2,3$ are a compact operators, acting between $L^{2}$ and $H^{1}$, with
\begin{align*}
&||\tilde{T}_i^m v||_{L^2}\leq C ||v||_{L^2},\\
&\|\tilde{T}_i^{m} v\|_{H^{1}} \leq C(m)\|v\|_{L^{2}},
\end{align*}
where the constant $C$ in the first inequality  only  depends on $\|f\|_{C^{4}}$ and $\Omega_a$ and the constant $C(m)$ in the second one also depends on $m$.
\end{lemma}
\begin{proof}
Because of Lemma \ref{log+c1} $\tilde{T}^3_i$ is bounded from $L^2$ to $L^2$. Then, by Lemma \ref{positivo} and the monotonicity of $f(\rho)$, every $\tilde{T}^m_i$ is also bounded from $L^2$ to $L^2$ by the same constant as $\tilde{T}^3_i$. In order to show the boundedness from $L^2$ to $H^1$ we first show the estimate for smooth functions. This is done just using Lemma \ref{log+c1} and the $L^2-$boundedness of the Hilbert transform. Finally we proceed by a density argument.
\end{proof}
\begin{corollary}\label{H2H3} The operator $\tilde{T}^{m}$ is bounded from $H^{k}$ to $H^{k+1}$, for $k=0,1,2$, with norm depending only on $m$, $||f||_{C^4}$ and $a$.
\end{corollary}
\begin{proof} We first prove the bound for smooth functions. This can be done by using Lemma \ref{Tderivadas} and \ref{lemacompacidad}. Then we proceed by a density argument.
\end{proof}
Finally, we will study the adjoint operator of $\tilde{T}^3$ given by the expression
\begin{align*}
\tilde{T}^{3*}B(\rho)=f_\rho(\rho)\int_{1-a}^1B(\rho')K^3\left(\frac{\rho'}{\rho}\right)\frac{d\rho'}{\rho'}.
\end{align*}
\begin{lemma}\label{Tstarderivadas}Let $B\in C^3$ then
the following equalities hold:
\begin{enumerate}
\item
\begin{align*}
&\pa_\rho \tilde{T}^{3*}B(\rho)\\&=\pa^{2}_\rho f(\rho)\int_{1-a}^1K^3\left(\frac{\rho'}{\rho}\right)\frac{d\rho'}{\rho'}-\frac{1}{\rho'}V.P.\int_{1-a}^1B(\rho'),
\pa_{\rho'}\left(K^3\left(\frac{\rho'}{\rho}\right)\right)d\rho'\end{align*}
\item
\begin{align*}
&\pa^2_\rho \tilde{T}^{3*}B(\rho)\\&=\pa^3_\rho f(\rho)\int_{1-a}^1B(\rho')K^3\left(\frac{\rho'}{\rho}\right)\frac{d\rho'}{\rho'}\\
&-\frac{\pa^2_\rho f(\rho)}{\rho}P.V.\int_{1-a}^1B(\rho')\pa_{\rho'} K^3\left(\frac{\rho'}{\rho}\right)d\rho'-\pa_\rho\left(B(1)\frac{f_\rho(\rho)K^3\left(\frac{1}{\rho}\right)}{\rho}-B(1-a)\frac{f_\rho(\rho)K^3\left(\frac{1-a}{\rho}\right)}{\rho}\right)\\
&-\frac{\pa^2_\rho f(\rho)}{\rho}\int_{1-a}^1 B'(\rho') K^3\left(\frac{\rho'}{\rho}\right)+\frac{f_\rho(\rho)}{\rho}P.V.\int_{1-a}^1 B'(\rho')\pa_{\rho'}\left(\rho'K^3\left(\frac{\rho'}{\rho}\right)\right)d\rho',\end{align*}
\item
\begin{align*}
&\pa^3_\rho \tilde{T}^{3*}(s)\\&=\pa^4_\rho f(\rho)\int_{1-a}^1B(\rho')K^{3}\left(\frac{\rho'}{\rho}\right)\frac{d\rho'}{\rho'}+\frac{\pa^3_\rho f(\rho)}{\rho}\left(-B(1)K^{3}\left(\frac{1}{\rho}\right)+B(1-a)K^{3}\left(\frac{1-a}{\rho}\right)\right)\\&+\pa^3_\rho f(\rho)\int_{1-a}^1B'(\rho')K^{3}\left(\frac{\rho'}{\rho}\right)\frac{d\rho'}{\rho}\\
&+\pa_\rho\left(\frac{\pa^2_\rho f(\rho)}{\rho}\left(-B(1)K^3\left(\frac{1}{\rho}\right)+B(1-a)K^3\left(\frac{1-a}{\rho}\right)\right)\right)\\
&+2\pa^3_\rho f(\rho)\int_{1-a}^1B''(\rho')K^3\left(\frac{\rho'}{\rho}\right)\frac{d\rho'}{\rho}++2\frac{\pa^2_\rho f(\rho)}{\rho^2}\left(-B'(1)K^3\left(\frac{1}{\rho}\right)+(1-a)K^3\left(\frac{1-a}{\rho}\right)\right)\\
&+2\pa^2_\rho f(\rho)\int_{1-a}^1B''(\rho')K^3\left(\frac{\rho'}{\rho}\right)\frac{\rho'}{\rho}\frac{d\rho'}{\rho}
-\pa^2_\rho\left(B(1)\frac{f_\rho(\rho)K^3\left(\frac{1}{\rho}\right)}{\rho}-B(1-a)\frac{f_\rho(\rho)K^3\left(\frac{1-a}{\rho}\right)}{\rho}\right)\\
&\pa_\rho^2f(\rho)\int_{1-a}^1 B''(\rho')K^3\left(\frac{\rho'}{\rho}\right)\frac{\rho'}{\rho}\frac{d\rho'}{\rho}-f_\rho(\rho)P.V.\int_{1-a}^1 \pa_{\rho'}\left(\rho'^2K^3\left(\frac{\rho'}{\rho}\right)\right)\frac{d\rho'}{\rho}.
\end{align*}
\end{enumerate}
\end{lemma}
\begin{proof}
First we notice that, after a change of variables,
\begin{align*}
\tilde{T}^{3*}B(\rho)=f_\rho(\rho)\int_\frac{{1-a}}{\rho}^\frac{1}{\rho} B(s\rho)K^3\left(s\right)\frac{ds}{s}.
\end{align*}
Then, taking a derivative yields,
\begin{align}\label{luego1}
&\pa_\rho \tilde{T}^{3*}B(\rho)=\pa^2_\rho f(\rho)\int_\frac{{1-a}}{\rho}^\frac{1}{\rho}B(\rho s)K^3(s)\frac{ds}{s}+\frac{1}{\rho}f_\rho(\rho)
\left(B(1)K^3\left(\frac{1}{\rho}\right)-B(1-a)K^3\left(\frac{1-a}{\rho}\right)\right)\nonumber\\
&+f_\rho(\rho)\int_\frac{{1-a}}{\rho}^\frac{1}{\rho}K^3(s)B'(\rho s)ds
\end{align}
and changing variable we have that
\begin{align*}
&\pa_\rho \tilde{T}^{3*}B(\rho)=\pa^2_\rho f(\rho)\int_{1-a}^1 B(\rho')K^3\left(\frac{\rho'}{\rho}\right)\frac{d\rho'}{\rho'}+\frac{1}{\rho}f_\rho(\rho)
\left(B(1)K^3\left(\frac{1}{\rho}\right)-B(1-a)K^3\left(\frac{1-a}{\rho}\right)\right)\\
&-f_\rho(\rho)\int_{1-a}^1K^3\left(\frac{\rho'}{\rho}\right)B'(\rho')\frac{d\rho'}{\rho}.
\end{align*}
And integration by parts in the last integral yields
\begin{align*}
&\pa_\rho \tilde{T}^{3*}B(\rho)=\pa^2_\rho f(\rho)\int_{1-a}^1 B(\rho') K^3\left(\frac{\rho'}{\rho}\right)\frac{d\rho'}{\rho'}+f_\rho(\rho)V.P.\int_{1-a}^1\pa_{\rho'}K^3\left(\frac{\rho'}{\rho}\right)B(\rho')\frac{d\rho'}{\rho}
\end{align*}
Taking a derivative on the equation \eqref{luego1} we obtain
\begin{align}\label{luego2}
&\pa^2_\rho\tilde{T}^{3*}B(\rho)=\pa^3_\rho f(\rho)\int_\frac{{1-a}}{\rho}^\frac{1}{\rho}B(\rho s)K^3(s)\frac{ds}{s}+\pa^2_\rho f(\rho)\frac{1}{\rho}\left(-B(1)K^3\left(\frac{1}{\rho}\right)+B(1-a)K^3\left(\frac{1-a}{\rho}\right)\right)\nonumber\\
&+\pa^2_\rho f(\rho)\int_\frac{{1-a}}{\rho}^\frac{1}{\rho}B'(\rho s)K^3(s) ds-\pa_\rho\left(B(1)\frac{f_\rho(\rho)K^3\left(\frac{1}{\rho}\right)}{\rho}-B(1-a)\frac{f_\rho(\rho)K^3\left(\frac{1-a}{\rho}\right)}{\rho}\right)\nonumber\\
& +\pa^2_\rho f(\rho)\int_{\frac{1-a}{\rho}}^\frac{1}{\rho}B'(\rho s)K^3(s)ds+\frac{f_\rho(\rho)}{\rho^2}\left(B'(1)K^3\left(\frac{1}{\rho}\right)-B'(1-a)K^3\left(\frac{1-a}{\rho}\right)\right)\nonumber\\
&+f_\rho(\rho)\int_\frac{{1-a}}{\rho}^\frac{1}{\rho}K^3(s)B''(\rho s)s ds
\end{align}
A change of variable and integration by parts then yield
\begin{align*}
&\pa^2_\rho \tilde{T}^{3*}B(\rho)=\pa^2_\rho f(\rho)\int_{1-a}^1B(\rho')K^3\left(\frac{\rho'}{\rho}\right)\frac{d\rho'}{\rho'}\\
&-\frac{\pa^2_\rho f(\rho)}{\rho}P.V.\int_{1-a}^1B(\rho')\pa_{\rho'} K^3\left(\frac{\rho'}{\rho}\right)d\rho'-\pa_\rho\left(B(1)\frac{f_\rho(\rho)K^3\left(\frac{1}{\rho}\right)}{\rho}-B(1-a)\frac{f_\rho(\rho)K^3\left(\frac{1-a}{\rho}\right)}{\rho}\right)\\
&+\frac{\pa^2_\rho f(\rho)}{\rho}\int_{1-a}^1 B'(\rho') K^3\left(\frac{\rho'}{\rho}\right)+\frac{f_\rho(\rho)}{\rho}P.V.\int_{1-a}^1 B'(\rho')\pa_{\rho'}\left(\rho'K^3\left(\frac{\rho'}{\rho}\right)\right)d\rho'.
\end{align*}
Finally we take a derivative of the equation \ref{luego2}. We have that
\begin{align*}
&\pa^{3}_\rho \tilde{T}^{3*}=\pa^4_\rho f(\rho)\int_{\frac{1-a}{\rho}}^\frac{1}{\rho} B(\rho s)K^3(s)\frac{ds}{s} +\frac{\pa^3_\rho f(\rho)}{\rho}\left(-B(1)K^{3}\left(\frac{1}{\rho}\right)+B(1-a)K^{3}\left(\frac{1-a}{\rho}\right)\right)\\&+\pa^3_\rho f(\rho)\int_\frac{1-a}{\rho}^\frac{1}{\rho}B'(\rho s) K^3(s)ds\\
&+\pa_\rho\left(\frac{\pa^2_\rho f(\rho)}{\rho}\left(-B(1)K^3\left(\frac{1}{\rho}\right)+B(1-a)K^3\left(\frac{1-a}{\rho}\right)\right)\right)\\
&+2\pa^3_\rho f(\rho)\int_\frac{1-a}{\rho}^\frac{1}{\rho}B'(\rho s)K^3(s)ds +2\frac{\pa^2_\rho f(\rho)}{\rho^2}\left(-B'(1)K^3\left(\frac{1}{\rho}\right)+(1-a)K^3\left(\frac{1-a}{\rho}\right)\right)\\
&+2\pa^2_\rho f(\rho)\int_\frac{1-a}{\rho}^\frac{1}{\rho}B''(\rho s)K^3(s)s ds -\pa^2_\rho\left(B(1)\frac{f_\rho(\rho)K^3\left(\frac{1}{\rho}\right)}{\rho}-B(1-a)\frac{f_\rho(\rho)K^3\left(\frac{1-a}{\rho}\right)}{\rho}\right)\\
&+\pa^2_\rho f(\rho)\int_\frac{1-a}{\rho}^\frac{1}{\rho}K^3(s)B''(\rho s)s ds-f_\rho(\rho)\left(-\frac{1}{\rho^3}B''(1)K^3\left(\frac{1}{\rho}\right)+\frac{(1-a)^2}{\rho^3}B''(1-a)K^3\left(\frac{1-a}{\rho}\right)\right)\\
&f_\rho(\rho)\int_\frac{1-a}{\rho}^\frac{1}{\rho}K^3(s)B'''(\rho s)s^2 ds.
\end{align*}
And  a change of variable and an integration by parts yield
\begin{align*}
&\pa^3_\rho \tilde{T}^{3*}(s)=\pa^4_\rho f(\rho)\int_{1-a}^1B(\rho')K^{3}\left(\frac{\rho'}{\rho}\right)\frac{d\rho'}{\rho'}+\frac{\pa^3_\rho f(\rho)}{\rho}\left(-B(1)K^{3}\left(\frac{1}{\rho}\right)+B(1-a)K^{3}\left(\frac{1-a}{\rho}\right)\right)\\&+\pa^3_\rho f(\rho)\int_{1-a}^1B'(\rho')K^{3}\left(\frac{\rho'}{\rho}\right)\frac{d\rho'}{\rho}\\
&+\pa_\rho\left(\frac{\pa^2_\rho f(\rho)}{\rho}\left(-B(1)K^3\left(\frac{1}{\rho}\right)+B(1-a)K^3\left(\frac{1-a}{\rho}\right)\right)\right)\\
&+2\pa^3_\rho f(\rho)\int_{1-a}^1B''(\rho')K^3\left(\frac{\rho'}{\rho}\right)\frac{d\rho'}{\rho}++2\frac{\pa^2_\rho f(\rho)}{\rho^2}\left(-B'(1)K^3\left(\frac{1}{\rho}\right)+(1-a)K^3\left(\frac{1-a}{\rho}\right)\right)\\
&+2\pa^2_\rho f(\rho)\int_{1-a}^1B''(\rho')K^3\left(\frac{\rho'}{\rho}\right)\frac{\rho'}{\rho}\frac{d\rho'}{\rho}
-\pa^2_\rho\left(B(1)\frac{f_\rho(\rho)K^3\left(\frac{1}{\rho}\right)}{\rho}-B(1-a)\frac{f_\rho(\rho)K^3\left(\frac{1-a}{\rho}\right)}{\rho}\right)\\
&\pa_\rho^2f(\rho)\int_{1-a}^1 B''(\rho')K^3\left(\frac{\rho'}{\rho}\right)\frac{\rho'}{\rho}\frac{d\rho'}{\rho}-f_\rho(\rho)P.V.\int_{1-a}^1 \pa_{\rho'}\left(\rho'^2K^3\left(\frac{\rho'}{\rho}\right)\right)\frac{d\rho'}{\rho}.
\end{align*}
\end{proof}

\begin{lemma}\label{Tstarcompacidad} The operator $\tilde{T}^{3*}$ is bounded from $H^{k}$ to $H^{k+1}$, for $k=0,1,2$, with norm depending only on $||f||_{C^4}$ and $a$.
\end{lemma}
\begin{proof}
In order to prove this lemma we prove the estimates for smooth functions. This can be done by using Lemma \ref{Tstarderivadas}. Indeed, in order to control the evaluations of $B$ and its derivatives we use  the standard Sobolev embedding $H^k\subset C^{l}$, for $k>l+ \frac{1}{2}$ in one dimension ($H^3\subset C^2$). In addition we use Lemmas  \ref{log+c1}, \ref{c2}, together with the $L^2$-boundedness of the  Hilbert transform to control the terms involving the derivative of the kernel $K^3$. Finally we proceed by a density argument.
\end{proof}
Finally, we finish this section by studying the regularity of $\tilde{I}(\rho)$.
\begin{lemma}\label{regularidadI}
Let $f \in C^{4}(\mathbb{R}^{+})$. Then, $\tilde{I}(\rho)\in C^{3}(1-a,1)$.
\end{lemma}
\begin{proof}
 The proof follows easily by the fact that $f \in C^{4}$ and Lemma \ref{log+c1}.
\end{proof}

\item \textbf{Existence of a solution for the equation \eqref{simplificada}.}
 We will regard our operator $\Theta^3$  as a perturbation of its symmetric part $\Theta_S^3$ (which we can do by taking $a$ and $\beta$ small enough). Then we estimate the first and second eigenvalues (and the first eigenfunction) and see that there is a gap. If the antisymmetric perturbation is small enough, then there is still a gap and we are good. The strategy is to use a computer-assisted proof for the estimation of the eigenvalue-eigenvector pair and the norms of the different operators that appear.

Let $(B_{sj}, \lambda^*)$ be the approximate eigenvector-eigenvalue pair for the symmetric operator $\Theta^3_{S}$, satisfying
\begin{align}\label{casi}
\Theta_{S}^3 B_{sj} = \lambda^{*} B_{sj} + e,
\end{align}

where $e$ is small (see the next lemma for an explicit bound on $e$).

\begin{lemma}\label{lemacotae}
 Let $B_{sj} = \frac{1}{\sqrt{a-a\beta}}{1}_{[1-a+\frac{a\beta}{2},1-\frac{a\beta}{2}]}$ and $\lambda^{*} = 0.3482$. Then:
\begin{align*}
 \|e\|_{L^{2}} < 0.0905 \\
|\langle e, B_{sj} \rangle | < 0.0101
\end{align*}

\end{lemma}
\begin{proof}
 The proof is computer-assisted and the code can be found in the supplementary material. We refer to the appendix for details about the implementation.
\end{proof}

 The symbol $B_{sj}^\perp$ will denote the orthogonal space to $B_{sj}$ in $L^2((1-a,1))$, i.e.
 $$B_{sj}^\perp\equiv \{ v\in L^2((1-a,1)),\,\:\, \langle B_{sj}, v\rangle_{L^2((1-a,1))}=0\}.$$

 Let us find a solution $(B^{3},\lambda)$ such that
\begin{align}\label{eqnB3}
\Theta^3 B^{3} = \lambda B^{3}, \quad B= B_{sj} + v, \quad v \in B_{sj}^{\perp}.
\end{align}
We expect $v$ to be small if $(B_{sj},\lambda^{*})$ is an accurate enough approximation to the true eigenpair. Plugging this ansatz for  $B^{3}$ into the previous equation and using the equation \eqref{casi}, we obtain
\begin{align*}
\Theta^3 v & = \Theta^3 B^{3} - \Theta^3 B_{sj} \\
& = \lambda B_{sj} + \lambda v - \lambda^{*} B_{sj} - e - \Theta^3_{A} B_{sj} \\
& = (\lambda - \lambda^{*})B_{sj} + \lambda v  - e - \Theta^3_{A} B_{sj},
\end{align*}

where $\Theta^{3}_{A} = \Theta^{3} - \Theta^{3}_{S}$ is the antisymmetric part of $\Theta^{3}$. From now on, the pairing $\langle \cdot, \cdot \rangle $ is assumed to be taken in $L^2((1-a,1))$. Let us consider the functional equation

\begin{align*}
\langle \Theta^3 v - \lambda v, u\rangle = - \langle e + \Theta^3_{A} B_{sj}, u \rangle \quad \forall u \in B_{sj}^{\perp},
\end{align*}

and let $c^{*}$ be defined by

\begin{align*}
c^{*} = \inf_{v \in B_{sj}^{\perp}} \frac{\langle \Theta^3 v, v\rangle }{\|v\|_{L^{2}}^{2}}.
\end{align*}

A rigorous bound for $c^{*}$ will be given in the next lemma:

\begin{lemma} \label{lemacotascestrella}
 Let  $B_{sj} = \frac{1}{\sqrt{a-a\beta}}{1}_{[1-a+\frac{a\beta}{2},1-\frac{a\beta}{2}]}$. Then:
\begin{align*}
 c^{*} \geq 0.8526
\end{align*}

\end{lemma}
\begin{proof}
 The proof is computer-assisted and the code can be found in the supplementary material. We refer to the appendix for details about the implementation.
\end{proof}

Then, if $\lambda < c^{*}$ the operator $\Theta^3 - \lambda$ is coercive in $B_{sj}^{\perp}$. By Lax-Milgram, for every $\lambda < c^{*}$ there exists a function $v^{\lambda} \in B_{sj}^{\perp}$ such that
\begin{align}\label{lm}
\langle \Theta^3 v^{\lambda} - \lambda v^{\lambda}, u\rangle = - \langle e + \Theta^3_{A} B_{sj}, u \rangle \quad \forall u \in B_{sj}^{\perp}.
\end{align}
Thus, there exists a function $d(\lambda)$ such that
\begin{align}
\label{ecuacionalphadelambda}
\Theta^{3} v^{\lambda} & = (\lambda - \lambda^{*})B_{sj} + \lambda v^{\lambda}  - e - \Theta^{3}_{A} B_{sj} + d(\lambda) B_{sj}.
\end{align}
By computing the scalar product of \eqref{ecuacionalphadelambda} with $v^{\lambda}$ we obtain
\begin{align*}
(c^{*} - \lambda) \|v^{\lambda}\|_{L^{2}}^{2} \leq \langle (\Theta^3 - \lambda) v^{\lambda}, v^{\lambda} \rangle = - \langle \Theta^3_{A} B_{sj}, v^{\lambda} \rangle - \langle e, v^{\lambda} \rangle
\leq |\langle \Theta^3_{A} B_{sj} + e, v^{\lambda}\rangle|
\leq \|\Theta^3_{A} B_{sj} + e\|_{L^{2}} \|v^{\lambda}\|_{L^{2}}.
\end{align*}
This inequality implies
\begin{align}\label{cotav}
\|v^{\lambda}\|_{L^{2}} \leq \frac{\|\Theta^3_{A} B_{sj} + e\|_{L^{2}}}{c^{*} - \lambda}.
\end{align}
In addition, using that
\begin{align*}
\langle \Theta^3 v^{\lambda}, B_{sj} \rangle & = \langle \Theta^3_{S} v^{\lambda}, B_{sj} \rangle + \langle \Theta^3_{A} v^{\lambda}, B_{sj} \rangle  = \langle v^{\lambda}, \Theta^3_{S} B_{sj} \rangle - \langle  v^{\lambda}, \Theta^3_{A} B_{sj} \rangle  = \langle v^{\lambda},e \rangle - \langle  v^{\lambda}, \Theta^3_{A} B_{sj} \rangle,
\end{align*}
taking the scalar product with $B_{sj}$ of the  equation \eqref{ecuacionalphadelambda} yields
\begin{align*}
d(\lambda) \|B_{sj}\|_{L^{2}}^{2} & = \langle \Theta^3 v^{\lambda}, B_{sj} \rangle + (\lambda^{*} - \lambda) \|B_{sj}\|_{L^{2}}^{2} + \langle e, B_{sj} \rangle \\
& = \langle v^{\lambda},e \rangle - \langle  v^{\lambda}, \Theta^3_{A} B_{sj} \rangle + (\lambda^{*} - \lambda) \|B_{sj}\|_{L^{2}}^{2} + \langle e, B_{sj} \rangle.
\end{align*}

An important fact, that will be used, is the continuity of this function $d(\lambda)$:
\begin{lemma}\label{ddl}The function $d(\lambda)$ defined in \eqref{ecuacionalphadelambda}
is continuous.
\end{lemma}
\begin{proof}
Let $\lambda,\gamma <c^*$. By definition of $d(\lambda)$ we have that
\begin{align*}
|d(\lambda)-d(\gamma)|\leq C\left(||B_{sj}||_{L^2},||e||_{L^2},||\Theta^3||_{L^2\to L^2}\right)||v^\lambda-v^\gamma||_{L^2}+||B_{sj}||_{L^2}^2|\lambda-\gamma|.
\end{align*}
For every $u\in B^\perp_{sj}$ we know from \eqref{lm} that
\begin{align*}
\langle \Theta^3 v^{\lambda}-\lambda v^{\lambda}, u\rangle = & \langle f, u \rangle\\
\langle \Theta^3 v^{\gamma}-\gamma v^{\gamma}, u\rangle = & \langle f, u \rangle,
\end{align*}
where $f\in L^2$. And then
\begin{align*}
\langle \Theta^{3}(v^\lambda-v^\gamma)-\lambda (v^\lambda-v^\gamma),u\rangle=(\gamma-\lambda)\langle v^\gamma, u  \rangle.
\end{align*}
Taking $u=v^\lambda-v^\gamma\in B^\perp_{sj}$ yields
\begin{align*}
(c^*-\lambda)||v^\lambda-v^\gamma||_{L^2}^2\leq |\gamma-\lambda|\,||v^\gamma||_{L^{2}}\,||v^\lambda-v^\gamma||_{L^2}.
\end{align*}
Thus
\begin{align*}
||v^\lambda-v^\gamma||_{L^2}\leq \frac{C}{c^*-\lambda}|\gamma-\lambda|\,||v^\gamma||_{L^2}.
\end{align*}
We achieve the conclusion of the lemma from the inequality \eqref{cotav}.
\end{proof}
Our purpose is now to show that there exists a value of $\lambda =\lambda^{**} < c^{*}$ such that $d(\lambda^{**})=0$. We have the following upper bound for $d(\lambda)$:
\begin{align}\label{ddlambda}
d(\lambda) & \geq -\frac{\|v^{\lambda}\|_{L^{2}} \|e - \Theta^3_{A} B_{sj}\|_{L^{2}}}{\|B_{sj}\|_{L^{2}}^{2}}  + (\lambda^{*} - \lambda) - \frac{|\langle e, B_{sj} \rangle|}{\|B_{sj}\|_{L^{2}}^{2}} \\
& \geq -\frac{\|\Theta^3_{A} B_{sj} + e\|_{L^{2}}}{c^{*} - \lambda}\frac{\|e - \Theta^3_{A} B_{sj}\|_{L^{2}}}{\|B_{sj}\|_{L^{2}}^{2}}  + (\lambda^{*} - \lambda) - \frac{|\langle e, B_{sj} \rangle|}{\|B_{sj}\|_{L^{2}}^{2}}\nonumber.
\end{align}
We observe that $\lim_{\lambda \to -\infty}d(\lambda) = +\infty$. Hence, there exists a $\lambda$ sufficiently small for which $d(\lambda) > 0$. Similarly:
\begin{align}\label{cotad}
d(\lambda) & \leq \frac{\|\Theta^3_{A} B_{sj} + e\|_{L^{2}}}{c^{*} - \lambda}\frac{\|e - \Theta^3_{A} B_{sj}\|_{L^{2}}}{\|B_{sj}\|_{L^{2}}^{2}}  + (\lambda^{*} - \lambda) + \frac{|\langle e, B_{sj} \rangle|}{\|B_{sj}\|_{L^{2}}^{2}}
\end{align}
The objective is to use the previous inequality to show that there exist a value of $\lambda$ smaller than $c^*$, such that $d(\lambda)<0$. We will use the following lemma.
\begin{lemma}
\label{lemavertice}
Let $h(x) = \mathcal{A} - x + \frac{\mathcal{B}}{\mathcal{C}-x}$, where $\mathcal{A},\mathcal{B},\mathcal{C} > 0$. If

\begin{align*}
\mathcal{C} > \mathcal{A} + 2\sqrt{\mathcal{B}},
\end{align*}

then there exists an $x < \mathcal{C}$ such that $h(x) < 0$.
\end{lemma}
\begin{proof}
It is easy to notice that $h(x)$ is concave up since $h''(x) > 0$ for $x < \mathcal{C}$. Thus, it is enough to check that the minimum of $h$ is negative. We calculate the point $x_{m}$ where the minimum is attained.

\begin{align*}
h'(x_{m}) = 0 \Leftrightarrow (\mathcal{C}-x_{m})^{2} = \mathcal{B} \Leftrightarrow x_{m} = \mathcal{C} - \sqrt{\mathcal{B}}.
\end{align*}

Evaluating at $x = x_{m}$:

\begin{align*}
h(x_{m}) = \mathcal{A} - \mathcal{C} + 2\sqrt{\mathcal{B}} < 0,
\end{align*}

by hypothesis.
\end{proof}

\begin{lemma} \label{lemacotasABC}
 Let $B_{sj} = \frac{1}{\sqrt{a-a\beta}}{1}_{[1-a+\frac{a\beta}{2},1-\frac{a\beta}{2}]}$, $\lambda^{*} = 0.3482$ and let us define:

\begin{align*}
\mathcal{A} = \lambda^{*} + \frac{|\langle e, B_{sj} \rangle|}{\|B_{sj}\|_{L^{2}}^{2}}, \quad \mathcal{B} = \frac{\|\Theta^3_{A} B_{sj} + e\|_{L^{2}}\|e - \Theta^{3}_{A} B_{sj}\|_{L^{2}}}{\|B_{sj}\|_{L^{2}}^{2}}
\end{align*}

Then, we have the following bounds:
\begin{align*}
\mathcal{A} < 0.3583 \\
\sqrt{\mathcal{B}} < 0.1534 \\
\end{align*}

\end{lemma}
\begin{proof}
 The proof is computer-assisted and the code can be found in the supplementary material. We refer to the appendix for details about the implementation.
\end{proof}

Finally, by applying Lemma \ref{lemavertice} to the right hand side of \eqref{cotad}  with

\begin{align*}
\mathcal{A} = \lambda^{*} + \frac{|\langle e, B_{sj} \rangle|}{\|B_{sj}\|_{L^{2}}^{2}}, \quad \mathcal{B} = \frac{\|\Theta^3_{A} B_{sj} + e\|_{L^{2}}\|e - \Theta^{3}_{A} B_{sj}\|_{L^{2}}}{\|B_{sj}\|_{L^{2}}^{2}}, \quad \mathcal{C} = c^{*}
\end{align*}

and by the bounds given by Lemma \ref{lemacotasABC}, we obtain that there exists a $\lambda_0$ for which $d(\lambda_0) < 0$. More precisely, we have that

\begin{align*}
 \lambda_{0} & \leq \frac{\mathcal{C}+\mathcal{A}}{2} - \sqrt{\left(\frac{\mathcal{C}-\mathcal{A}}{2}\right)^{2} - \mathcal{B}} = \frac{\mathcal{A}\mathcal{C}+\mathcal{B}}{\frac{\mathcal{A}+\mathcal{C}}{2} + \sqrt{\left(\frac{\mathcal{A}-\mathcal{C}}{2}\right)^{2} - \mathcal{B}}} \\
& = \frac{c^{*} + \lambda^{*} + \frac{|\langle e, B_{sj} \rangle|}{\|B_{sj}\|_{L^{2}}^{2}}}{2} - \sqrt{\left(\frac{c^{*} - \lambda^{*} - \frac{|\langle e, B_{sj} \rangle|}{\|B_{sj}\|_{L^{2}}^{2}}}{2}\right)^{2}-\frac{\|\Theta^3_{A} B_{sj} + e\|_{L^{2}}\|e - \Theta^{3}_{A} B_{sj}\|_{L^{2}}}{\|B_{sj}\|_{L^{2}}^{2}}}
\end{align*}

By the continuity of $d(\lambda)$ proved in Lemma \ref{ddl}, there has to be a $\lambda^{**} \leq \lambda_{0} < c^{*}$
for which $d(\lambda^{**}) = 0$ and therefore

\begin{align*}
\Theta v^{\lambda^{**}} & = (\lambda^{**} - \lambda^{*})B_{sj} + \lambda^{**} v^{\lambda^{**}}  - e - \Theta_{A} B_{sj}.
\end{align*}

which means that $( B_{sj} + v^{\lambda^{**}}, \lambda^{**})$ is the eigenvalue-eigenvector pair, $(B^3,\lambda_3)\in L^2((1-a,1))\times \RR$, we were looking for.

\begin{corollary}
\label{corollarylambda0}
 We have just shown that $\lambda_3 \leq \lambda_{0} < 0.4117$.
\end{corollary}
\begin{proof}
 Follows from the numerical bounds obtained in Lemma \ref{lemacotasABC} and Lemma \ref{lemavertice}.
\end{proof}

Next we show that $B^3$ is unique modulo multiplication by constants. In order to prove it, let us assume that   $w\in  L^2((1-a,1))$ satisfies
\begin{align*}
\Theta^3 w=\lambda_3 w.
\end{align*}
Now we write $w=B_{sj}+(w-B_{sj})$ and we notice we can decompose $w-B_{sj}=\alpha B_{sj}+v$ where $\alpha\in \RR$ and $v\in B_{sj}^\perp$. Then $w=(1+\alpha)B_{sj}+ v$ and by linearity we have that $\frac{w}{1+\alpha}=B_{sj}+ u$, with $u=\frac{v}{1+\alpha}\in B_{sj}^\perp$, is also a solution. Then the uniqueness of the solutions in the Lax-Milgram theorem implies $\frac{w}{1+\alpha}=B^3$. If $\alpha=-1$ the previous argument fails. But, in this case, $w\in B_{sj}^\perp$ and then $c^*\leq \lambda_3$ which has  already been proven to be false.

In conclusion, we have shown that dim($\mathcal{N}(\Theta^{3} - \lambda^{3})) = 1$.

\begin{lemma}
\label{lemmaImayorquelambda}
We have that $\tilde{I}(\rho)-\lambda_3 > 0$. In particular, we have the following bounds:

\begin{align*}
 \tilde{I}(\rho)-\lambda_3 > 0.8526
\end{align*}

\end{lemma}
\begin{proof}
 The proof is computer-assisted and the code can be found in the supplementary material. We refer to the appendix for details about the implementation.
\end{proof}

Then it remains to prove the regularity of $B^3$, the solution of equation \eqref{eqnB3}. To do this we will bootstrap using  Lemma \ref{lemacompacidad}. Since
\begin{align}\label{lades}
\tilde{I}(\rho)B^3(\rho)-\lambda_3 B^3(\rho)=-\tilde{T}^3B^3(\rho)
\end{align}
by Lemma \ref{lemacompacidad} the function $(\tilde{I}(\rho)-\lambda_3) B^3(\rho)\in H^1((1-a,1))$. Since $\tilde{I}\in C^3$  and Lemma \ref{lemmaImayorquelambda}
 we have that $B^3(\rho)$ is in $H^1$. Let's take two derivatives on equation \eqref{lades}, by Lemma \ref{Tderivadas} we have that
 \begin{align*}
 \tilde{I}(\rho)\pa^{2}_\rho B^3(\rho)-\lambda_3\pa^{2}_\rho B^3(\rho)=&-\sum_{j=1}^2\pa_\rho^{j)} \tilde{I}(\rho)\pa^{2-j)}_\rho B^3(\rho)-\pa_\rho T^{3}_1\pa_\rho B^3(\rho)\\&-\pa_\rho\left(\frac{1}{\rho}\inti \pa^{2}_\rho f(\rho') B(\rho')\frac{\rho'}{\rho}K^{m}\left(\frac{\rho}{\rho'}\right)d\rho'\right)
 \end{align*}
 Then using that $\tilde{I}(\rho)\in C^3$, $f\in C^4$, $\tilde{I}-\lambda_3>0$, Lemma \ref{lemacompacidad} and $B^3\in H^1$ we have that $B^3\in H^2$. Finally, taking three derivatives yields
 \begin{align*}
 \tilde{I}(\rho)\pa^{3}_\rho B^3(\rho)-\lambda_3\pa^{3}_\rho B^3(\rho)=&-\sum_{j=1}^3\pa_\rho^{j)} \tilde{I}(\rho)\pa^{3-j)}_\rho B^3(\rho)-\pa_\rho T^{3}_2\pa^2_\rho B^3(\rho)\\&-\pa_\rho\left(\sum_{j=1}^2\frac{1}{\rho}\inti \pa^{j+1)}_\rho f(\rho') \pa^{2-j)}_\rho B(\rho')\left(\frac{\rho'}{\rho}\right)^2 K^{m}\left(\frac{\rho}{\rho'}\right)d\rho'\right)
 \end{align*}
  Then using, again, that $\tilde{I}(\rho)\in C^3$, $f\in C^4$, $\tilde{I}-\lambda_3>0$, Lemma \ref{lemacompacidad} and $B^3\in H^2$ we have that $B^3\in H^3.$
\end{enumerate}
This concludes the proof of Proposition \ref{btres}.

\end{proof}

A similar proof as in Proposition \ref{btres} also works to show:
\begin{prop} \label{btresadj} There exists a solution $(B^3_*,\lambda_3)\in H^3((1-a),1)\times \RR$ to the equation
\begin{align*}
\Theta^{3*}B^3_*=\lambda_3B^3_*,
\end{align*}
 where $\lambda_3$ is the same eigenvalue as in Proposition \ref{btres} and $\Theta^{3*}$ is the adjoint operator of $\Theta^{3}$.
 In addition, $\lambda_3$ is simple and we can decompose $$B_*^3=B_{sj}+ v_*^{\lambda_3},$$ where $v_*^{\lambda_3}\in \left(B_{sj}\right)^\perp$

and
 \begin{align*}
 ||v_*^{\lambda_3}||_{L^2}\leq \frac{||-\Theta^3_A B_{sj}+e||_{L^{2}}}{c^*-\lambda_3}
 \end{align*}
\end{prop}
\begin{proof} The proof that there exists a pair $(B^3_*,\lambda_3^*)\in L^2\times \RR$ satisfying
\begin{align*}
(\Theta^{3*}-\lambda_3^*)B^3_*=\lambda_3^*B^3_*
\end{align*}

runs the same steps than the proof of  Proposition \ref{btres}. The only modification is the change $\Theta^3_A\to -\Theta^3_A$. In order to check that, in fact, $\lambda_3^*=\lambda_3$ we notice that
\begin{align*}
\lambda_3\langle B^3,B^3_*\rangle=\langle \Theta^3 B^3, B^3_*\rangle =\langle B^3, \Theta^{3*}B^*_3\rangle=\lambda_3^* \langle B^3,B^3_*\rangle.
\end{align*}
Therefore it is enough to show that $\langle B^3,B^3_*\rangle \neq 0$. We prove this result in the following lemma:
\begin{lemma}\label{BBstar}The following inequality holds:
\begin{align*}\langle B^3,B^3_*\rangle >0.\end{align*}
\end{lemma}
\begin{proof}
We can decompose $B^3$ and $B^3_*$ in the following way
\begin{align*}
B^3= & B_{sj}+v^{\lambda_3}\\
B^3_*= & B_{sj}+v^{\lambda_3}_*
\end{align*}
where $v^{\lambda_3},\, v^{\lambda_3}_*\in B_{sj}^\perp$. Thus
\begin{align}\label{eqB3B3star}\langle B^3,B^3_* \rangle = ||B_{sj}||_{L^2}^2+\langle v^{\lambda_3}, v^{\lambda_3}_*\rangle\geq ||B_{sj}||_{L^2}^2-\frac{||-\Theta_A^3B_{sj}+e||_{L^2}||\Theta_A^3B_{sj}+e||_{L^2}}{(c^*-\lambda_3)^2}.\end{align}

Using the notation from Lemma \ref{lemavertice}, we can bound

\begin{align*}
 c^{*} - \lambda_{3} \geq c^{*} - \lambda_{0} = \frac{C-A}{2} + \sqrt{\left(\frac{C-A}{2}\right)^{2} - B}.
\end{align*}

This implies that the RHS of \eqref{eqB3B3star} is bounded below by

\begin{align*}
 ||B_{sj}||_{L^2}^2\left(1 - \frac{B}{\left(\frac{C-A}{2} + \sqrt{\left(\frac{C-A}{2}\right)^{2} - B}\right)^{2}}\right) > 0,
\end{align*}

where in the last inequality we have used the already checked condition that

\begin{align*}
 \frac{C-A}{2} > \sqrt{B}.
\end{align*}

\end{proof}

Finally we prove that $B^3_*$ is in $H^3$ in the same way we did for $B^3$ by using Lemma \ref{Tstarcompacidad}.
\end{proof}

\subsubsection{One-dimensionality of the kernel of $\pa_r F[\rho,\lambda_3]$}\label{uniqueness}
Until now, it has already been proven that  there exists an element in the kernel of the operator $\pa_r F[\rho,\lambda_3]$. In this section we will prove that this kernel is the span of this element. As commented in point 3 at the beginning of the proof of Proposition \ref{kernel} it is enough to prove that the equation
$$\Theta^{3n} u=\lambda_3u$$
implies   that $u=0$ for $n>1$.
\begin{lemma}\label{cotalambda1}
 Let $m>1$ and
 \begin{align*}\lambda_{3m}^{s}=\inf_{\substack{u\in L^2 \\ ||u||_{L^2}=1}} \langle \Theta^{3m}u,u\rangle\end{align*}
  Then, if the pair $(u_{3m},\lambda_{3m})\in L^2 \times \RR$ (with $u_{3m}$ not identically zero) satisfies
  \begin{align*}
  \Theta^{3m}u_{3m}=\lambda_{3m},
  \end{align*}
  we have that
\begin{align*}
 \lambda_{3m}^{s} \leq \lambda_{3m}
\end{align*}

\end{lemma}

\begin{proof}
Since we can take $u_{3m}$ with norm 1
we have that
\begin{align*}
 \lambda_{3m} = \langle \Theta^{3m} u_{3m}, u_{3m} \rangle \geq \lambda_{3m}^{s},
\end{align*}
by definition of $\lambda_{3m}^{s}$.
\end{proof}

\begin{lemma}\label{mn}
 Let $m > n$ and let $\lambda_{3m}^{s}, \lambda_{3n}^{s}$ be given by
 \begin{align*}
 \lambda_{3m}^{s}=\inf_{\substack{u\in L^2 \\ ||u||_{L^2}=1}} \langle \Theta^{3m} u,u\rangle\quad \text{and}\quad \lambda_{3n}^{s}=\inf_{\substack{u\in L^2 \\ ||u||_{L^2}=1}} \langle \Theta^{3n}u,u\rangle .
 \end{align*}
Then
\begin{align*}
 \lambda_{3m}^{s} \geq \lambda_{3n}^{s}
\end{align*}

\end{lemma}

\begin{proof}
First we will show that for every $m\geq 1$ and $u\in L^2$ the following inequality  holds $$\langle \Theta^{3m}u,u\rangle \geq \langle\Theta^{3m}|u|,|u|\rangle.$$
By using the positivity of $T^{3m}(\rho,\rho')$ proved in Lemma \ref{positivo} we have that
\begin{align*}
\langle \Theta^{3m}|u|,|u| \rangle & = \int \tilde{I}(\rho)|u(\rho)|^{2} d\rho +\int \int \frac{1}{\rho}\left(\frac{\rho'}{\rho}\right)T^{3m}(\rho,\rho')f_\rho(\rho')|u(\rho)||u(\rho')| d\rho d\rho' \\
& = \int I(\rho)|u(\rho)|^{2} d\rho + \int \int \frac{1}{\rho}\left(\frac{\rho'}{\rho}\right)T^{3m}(\rho,\rho')f_\rho(\rho')u^{+}(\rho)u^{+}(\rho') d\rho d\rho'\\ & + \int \int \frac{1}{\rho}\left(\frac{\rho'}{\rho}\right)T^{3m}(\rho,\rho')f_\rho(\rho')u^{-}(\rho)u^{-}(\rho') d\rho d\rho' \\
& + \int \int \frac{1}{\rho}\left(\frac{\rho'}{\rho}\right)T^{3m}(\rho,\rho')f_\rho(\rho')u^{+}(\rho)u^{-}(\rho') d\rho d\rho'\\& + \int \int \frac{1}{\rho}\left(\frac{\rho'}{\rho}\right)T^{3m}(\rho,\rho')f_\rho(\rho')u^{-}(\rho)u^{+}(\rho') d\rho d\rho' \\
& \leq \int I(\rho)|u(\rho)|^{2} d\rho + \int \int \frac{1}{\rho}\left(\frac{\rho'}{\rho}\right)T^{3m}(\rho,\rho')f_\rho(\rho')u^{+}(\rho)u^{+}(\rho') d\rho d\rho' \\ & + \int \int \frac{1}{\rho}\left(\frac{\rho'}{\rho}\right)T^{3m}(\rho,\rho')f_\rho(\rho')u^{-}(\rho)u^{-}(\rho') d\rho d\rho' \\
& - \int \int\frac{1}{\rho}\left(\frac{\rho'}{\rho}\right) T^{3m}(\rho,\rho')f_\rho(\rho')u^{+}(\rho)u^{-}(\rho') d\rho d\rho' \\ &- \int \int \frac{1}{\rho}\left(\frac{\rho'}{\rho}\right)T^{3m}(\rho,\rho')f_\rho(\rho')u^{-}(\rho)u^{+}(\rho') d\rho d\rho' \\
& = \int I(\rho)u(\rho)^{2} d\rho + \int \int \frac{1}{\rho}\left(\frac{\rho'}{\rho}\right)T^{3m}(\rho,\rho')f_\rho(\rho')u(\rho)u(\rho') d\rho d\rho' \\
& = \langle \Theta^{3m}u, u \rangle
\end{align*}
Then $$\inf_{u\in L^2\atop ||u||_{L^2}=1}\langle \Theta^{3m} u , u\rangle=\inf_{u\in L^2\atop ||u||_{L^2}=1,\, u\geq 0}\langle \Theta^{3m} u , u\rangle,$$
Now, for a positive $u\in L^2$, we write
 \begin{align*}
 \langle \Theta^{3m}u,u \rangle = \langle \left(\Theta^{3m}-\Theta^{3n}\right)u,u\rangle + \langle \Theta^{3m}u,u \rangle.
 \end{align*}
 where
\begin{align*}
 \left(\Theta^{3m}-\Theta^{3n}\right)u=\frac{1}{\rho}\inti f_\rho(\rho')\left(K^{3m}-K^{3n}\right)\left(\frac{\rho}{\rho'}\right) u(\rho')d\rho'
 \end{align*}
and notice that
\begin{align*}
\frac{1}{\rho}\inti f_\rho(\rho')\left(K^{3m}-K^{3n}\right)\left(\frac{\rho}{\rho'}\right) u(\rho')d\rho'\geq 0
\end{align*}
by Lemma \ref{positivo}.

We then obtain that
 \begin{align*}
 \langle \Theta^{3m}u,u \rangle \geq  \langle \Theta^{3n}u,u \rangle,
 \end{align*}
 for every positive $u\in L^2$.
 This concludes the proof of the lemma.
 \end{proof}

 Now our purpose is to find a bound from below for the number $\lambda_s^6$ defined by  $$\lambda_s^{6} =\inf_{u\in L^2 \atop ||u||_{L^2}=1} \langle \Theta^6 u, u\rangle.$$
 In order to do this we will need Lemmas \ref{lemacotae6} and \ref{lemacotascestrella6} below.

 \begin{lemma}\label{lemacotae6} Let $B^{6\text{aprox}}_s= \frac{1}{\sqrt{a-a\beta}}{1}_{[1-a+\frac{a\beta}{2},1-\frac{a\beta}{2}]}$ and $\lambda^{s}_{6\text{aprox}}= 0.573$. Then
 \begin{align*}
 \Theta^6_S B^{6\text{aprox}}_s-\lambda^{s}_{6\text{aprox}}B^{6\text{aprox}}_s=e^6
 \end{align*}
with
$$||e^6||_{L^2} < 0.0893.$$
\end{lemma}
 \begin{proof} The proof is computer-assisted and the codes can be found in the supplementary material. We refer to the appendix for the implementation.
 \end{proof}

\begin{lemma}\label{lemacotascestrella6} Let $B^{6\text{aprox}}_s$ be the approximation in Lemma \ref{lemacotae6}. Then, if we define the number $c^{6*}$ by
\begin{align*}
c^{6*}=\inf_{v\in \left(B^{6\text{aprox}}_s\right)^\perp\atop ||v||_{L^2}=1} \langle \Theta^6 v,v\rangle,
\end{align*}
the following bound holds
\begin{align*}
c^{6*}> 0.8355.
\end{align*}
\end{lemma}

\begin{proof}
The proof is computer-assisted and the codes can be founded in the supplementary material. We refer to the appendix for details on the implementation.
\end{proof}

\begin{lemma}\label{lemacotalambda6}
 Let  $\lambda_{6}^{s}\in\RR$ given by \begin{align*}
 \lambda_6^{s} =\inf_{u\in L^2 \atop ||u||_{L^2}=1} \langle \Theta^6 u, u\rangle.
 \end{align*}   Then, we have the following bound:

\begin{align*}
 \lambda_{6}^{s} > 0.4837.
\end{align*}

\end{lemma}

\begin{proof} For a generic $B\in L^2$ with $||B||_{L^2}=1$ we can decompose $B=\alpha B^{6\text{aprox}}_{s}+\beta v$, where $v\in \left(B^{6\text{aprox}}_{s}\right)^\perp$, $||v||_{L^2}=1$  and $\alpha^2+\beta^2=1$. Therefore
\begin{align*}
\langle \Theta^6_{S}B, B\rangle &=\alpha^2\lambda^{s}_{6\text{aprox}}+ \beta^2\langle \Theta^6_{S}v,v \rangle+ 2\alpha \beta\langle e^6, v\rangle \\
& \geq \alpha^2\lambda^{s}_{6\text{aprox}}+ \beta^2\langle \Theta^6_{S}v,v \rangle - 2\alpha \sqrt{1-\alpha^2}\langle e^6, v\rangle.
\end{align*}
Finally we obtain that
\begin{align*}
\langle \Theta^6_{S}B, B\rangle \geq \lambda^{s}_{6\text{aprox}}-||e^6||_{L^2},
\end{align*}
since $$\inf_{v\in\left(B^{6\text{aprox}}_{s}\right)^\perp\atop ||v||_{L^2}=1 }\langle \Theta^6_{S} v,v\rangle\geq c^{6*}>\lambda^{s}_{6\text{aprox}}$$ by Lemmas \ref{lemacotae6} and \ref{lemacotascestrella6}.
\end{proof}

Using  Lemmas \ref{cotalambda1}, \ref{mn} and \ref{lemacotalambda6} we can prove the following proposition:

\begin{prop} \label{thetaicoer}The bilinear forms $\langle (\Theta^{3m}-\lambda_3)v,v\rangle$ are coercive in $L^2$ for $m>1$. In addition
\begin{align*}
\langle (\Theta^{3m}-\lambda_3)v,v\rangle\geq \frac{1}{\lambda^{s}_6-\lambda_3}||v||_{L^2}^2.
\end{align*}
Therefore,  if $u\in L^2$ satisfies\begin{align*}
 \Theta^{3m} u = \lambda_3 u,
\end{align*}
for $m>1$ then $u = 0$.
\end{prop}

Proposition \ref{kernel} is then proven.

\end{proof}

To finish this section we study the codimension of the image of the operator $\pa_rF[\rho,\lambda_3]$.
\begin{prop}\label{codimension} The space $H^{3,3}_{3,\text{odd}}(\Omega_a)/\mathcal{R}\left(\pa_rF[\rho,\lambda_3]\right)$ has dimension one.
\end{prop}
\begin{proof}
In order to prove this proposition we will study the range of $\pa_r F[\rho,\lambda_3]$. Let  $G(\alpha,\rho)\in H^{3,3}_{3,\text{odd}}(\Omega_a)$. We shall try to find $g(\alpha,\rho)\in H^{4,3}_{3,\text{even}}(\Omega_a)$ such that
\begin{align}\label{codeq}
\pa_{r}F[\rho,\lambda_3]g(\alpha,\rho)=G(\alpha,\rho)\end{align}
 By using the expansions
\begin{align*}
g(\alpha,\rho)= &\sum_{m=1}^\infty \rho g^{3m}(\rho)\cos(3m\alpha)\\
G(\alpha,\rho)= &\sum_{m=1}^\infty \rho G^{3m}(\rho)\sin(3m\alpha)
\end{align*} in \eqref{codeq} we have that
\begin{align*}
\sum_{m=1}^\infty 3m\rho\left(\Theta^{3m}g^{3m}(\rho)-\lambda_3 g^{3m}(\rho)\right)\sin(3m\alpha)=\sum_{m=1}^{\infty}\rho G^{3m}(\rho)\sin(3m\alpha).
\end{align*}
Taking the projection onto the $3m-$mode yields
\begin{align}\label{gG}
3m\left(\Theta^{3m}g^{3m}-\lambda_3 g^{3m}(\rho)\right)=G^{3m}(\rho)\quad \text{for $m=1,2,3,...$}
\end{align}

Next we shall study the existence of solutions for the equation \eqref{gG} in $L^2$ and after that  the $H^3$-regularity.

\begin{enumerate}
\item \textbf{Existence in $L^2$.} We will deal separately with  two cases: in Lemma \ref{emm1} we take $m>1$ and in Lemma \ref{emi1} $m=1$.

\begin{lemma}\label{emm1} For $m>1$ there exists an inverse operator
\begin{align*}
\left(\Theta^{3m}-\lambda_3\right)^{-1} \,:\, L^2\to L^2
\end{align*}
with norm bounded independently of $m$.
\end{lemma}
\begin{proof}
By Lemmas \ref{cotalambda1}, \ref{mn} and the explicit bound for $\lambda^s_6$ in Lemma \ref{lemacotalambda6} we have that the bilinear form $\langle(\Theta^{3m}-\lambda_3)v,v\rangle$ is coercive in $L^2$ with $$\langle(\Theta^{3m}-\lambda_3)v,v\rangle\geq \frac{1}{\lambda_6^s-\lambda_3}||v||_{L^2}^2\quad \text{for $m>1$}.$$ Since $G^{3m}\in H^3\subset L^2$ we can apply Lax-Milgram theorem in order to obtain the existence of an inverse operator $$(\Theta^{3m}-\lambda_3)^{-1} \,:\, L^2\to L^2.$$
\end{proof}

 \begin{lemma}\label{emi1} Let $B^{3}_*$ be defined by Proposition \ref{btresadj}. Then, if $G^3\in (B^{3}_*)^\perp$, there exists a solution, $g^{3}\in L^2$, to the equation \eqref{gG} with $m=1$. However, if $G^3\in \text{span}\left\{B_*^3\right\}$ there is no function in $L^2$ satisfying the equation \eqref{gG} with $m=1$.
 \end{lemma}
 \begin{proof}
 First we notice that we have already checked that $\langle (\Theta^3-\lambda_3)v,v\rangle$ is coercive in the space $B_{sj}^\perp$. In the next lemma we prove that it is also coercive in $(B^3_*)^\perp.$
\begin{lemma}\label{c}There exists a constant $c>\lambda_3$ such that
\begin{align*}
\inf_{v\in (B^3)^\perp\atop ||v||_{L_2}=1}\langle (\Theta^3-\lambda_3)v,v\rangle\geq c-\lambda_3.
\end{align*}
\end{lemma}
\begin{proof}
We take $v\in (B_*^3)^\perp$, with $L^2-$norm equal to 1. We can decompose $v$ in the following form $v=\alpha B_{sj}+\beta h$ where $h\in B^\perp_{sj}$ and $||h||_{L^2}=1$. Since $B^3_*=B_{sj}+v^{\lambda_3}_*$ we have that
\begin{align*}
\alpha ||B_{sj}||^2_{L^2}+\beta \langle h, v^{\lambda_3}_*\rangle=&0\\
\alpha^2 ||B_{sj}||_{L^2}^2+\beta^2=1.
\end{align*}
Therefore, we can write,
\begin{align*}
\langle (\Theta^3-\lambda_3)v,v\rangle =\langle (\Theta^3_S-\lambda_3)v,v\rangle,
\end{align*}
using that $\Theta^3_S B_{sj}=\lambda^* B_{sj}+e$,
\begin{align*}
&\langle (\Theta^3_S-\lambda_3)v,v\rangle=\alpha^2\langle(\Theta^3_S-\lambda_3)B_{sj},B_{sj}\rangle+ 2\alpha\beta\langle(\Theta^3_S-\lambda_3)B_{sj},h \rangle+\beta^2\langle (\Theta^3_S-\lambda_3)h,h \rangle \\
&=\alpha^2\langle (\lambda^*-\lambda_3)B_{sj}+e,B_{sj}\rangle+2\alpha\beta\langle (\lambda^*-\lambda_3)B_{sj}+e,h \rangle+\beta^2\langle (\Theta^3_S-\lambda_3)h,h \rangle.
\end{align*}

And substituting
\begin{align*}
\alpha=-\beta \frac{\langle h,v_*^{\lambda_3} \rangle}{||B_{sj}||_{L^2}^2}
\end{align*}
yields
\begin{align}
\langle (\Theta^3_S-\lambda_3)v,v\rangle& = \beta^2\left(\frac{\langle h,v_*^{\lambda_3} \rangle^2}{||B_{sj}||_{L^2}^4}\left((\lambda^*-\lambda_3)\|B_{sj}\|_{L^{2}}^{2} + \langle e, B_{sj} \rangle\right) +\langle (\Theta^3_S-\lambda_3)h,h\rangle- 2\frac{\langle h,v_*^{\lambda_3} \rangle}{||B_{sj}||_{L^2}^2}\left\langle e,h\right\rangle\right). \label{cuadraticahv}
\end{align}
Here we recall that
\begin{align*}
||v_*^{\lambda_3}||_{L^2}\leq & \frac{||-\Theta^3_A B_{sj}+e||_{L^2}}{c^*-\lambda_3}\\
\langle (\Theta_S^3-\lambda_3)h,h\rangle\geq & (c^*-\lambda_3)
\end{align*}
and notice that
\begin{align*}
\beta^2\geq \frac{1}{1+\frac{||-\Theta_A^3B_{sj}+e||^2_{L^2}}{(c^*-\lambda_3)^2||B_{sj}||_{L^2}^2}} > 0
\end{align*}
Therefore, it is enough to show that

\begin{align*}
\left(\frac{\langle h,v_*^{\lambda_3} \rangle^2}{||B_{sj}||_{L^2}^4}\left((\lambda^*-\lambda_3)\|B_{sj}\|_{L^{2}}^{2} + \langle e, B_{sj} \rangle\right) +\langle (\Theta^3_S-\lambda_3)h,h\rangle- 2\frac{\langle h,v_*^{\lambda_3} \rangle}{||B_{sj}||_{L^2}^2}\left\langle e,h\right\rangle\right) > 0.
\end{align*}

However, the LHS can be bounded from below by

\begin{align*}
c^* - \lambda_0 + \left(\frac{||-\Theta^3_A B_{sj}+e||_{L^2}}{\|B_{sj}\|_{L^{2}}^{2}(c^*-\lambda_0)}\right)^{2}\left(((\lambda^{*} - \lambda_{0})\|B_{sj}\|_{L^{2}}^{2} -|\langle e, B_{sj} \rangle |\right)
-  2\frac{||-\Theta^3_A B_{sj}+e||_{L^2}}{\|B_{sj}\|_{L^{2}}^{2}(c^*-\lambda_0)}\|e\|_{L^{2}},
\end{align*}

since

\begin{align*}
\langle (\Theta_S^3-\lambda_3)h,h\rangle & \geq  (c^*-\lambda_3) \geq c^{*} - \lambda_0 \\
\frac{\langle h,v_*^{\lambda_3} \rangle^2}{||B_{sj}||_{L^2}^4}(\lambda^*-\lambda_3)\|B_{sj}\|_{L^{2}}^{2} & \geq 
\left(\frac{||-\Theta^3_A B_{sj}+e||_{L^2}}{(c^*-\lambda_3)\|B_{sj}\|_{L^{2}}^{2}}\right)^{2}(\lambda^*-\lambda_3)\|B_{sj}\|_{L^{2}}^{2} \\
& \geq
\left(\frac{||-\Theta^3_A B_{sj}+e||_{L^2}}{(c^*-\lambda_0)\|B_{sj}\|_{L^{2}}^{2}}\right)^{2}(\lambda^*-\lambda_0)\|B_{sj}\|_{L^{2}}^{2} \\
\frac{\langle h,v_*^{\lambda_3} \rangle^2}{||B_{sj}||_{L^2}^4}\langle e, B_{sj} \rangle & \geq 
-\left(\frac{||-\Theta^3_A B_{sj}+e||_{L^2}}{(c^*-\lambda_3)\|B_{sj}\|_{L^{2}}^{2}}\right)^{2}\left|\langle e, B_{sj} \rangle\right| \\
& \geq
-\left(\frac{||-\Theta^3_A B_{sj}+e||_{L^2}}{(c^*-\lambda_0)\|B_{sj}\|_{L^{2}}^{2}}\right)^{2}\left|\langle e, B_{sj} \rangle\right| \\
- 2\frac{\langle h,v_*^{\lambda_3} \rangle}{||B_{sj}||_{L^2}^2}\left\langle e,h\right\rangle & \geq -  2\frac{||-\Theta^3_A B_{sj}+e||_{L^2}}{\|B_{sj}\|_{L^{2}}^{2}(c^*-\lambda_3)}\|e\|_{L^{2}} \geq -  2\frac{||-\Theta^3_A B_{sj}+e||_{L^2}}{\|B_{sj}\|_{L^{2}}^{2}(c^*-\lambda_0)}\|e\|_{L^{2}}
\end{align*}

Using Lemmas \ref{lemacotascestrella} and Corollary \ref{corollarylambda0}, we get

\begin{align*}
c^{*} - \lambda_{0} \geq 0.8526 - 0.4117 = 0.4409.
\end{align*}

Via Lemmas \ref{lemacotae} and \ref{lemacotasABC} we obtain

\begin{align*}
\left(\frac{||-\Theta^3_A B_{sj}+e||_{L^2}}{(c^*-\lambda_0)\|B_{sj}\|_{L^{2}}^{2}}\right)^{2}(\lambda^*-\lambda_0)\|B_{sj}\|_{L^{2}}^{2}
\geq \left(\frac{0.1534}{0.4409}\right)^{2}(0.3482-0.4117) \geq -0.01
\end{align*}

and similarly:

\begin{align*}
-\left(\frac{||-\Theta^3_A B_{sj}+e||_{L^2}}{(c^*-\lambda_0)\|B_{sj}\|_{L^{2}}^{2}}\right)^{2}\left|\langle e, B_{sj} \rangle\right|
\geq -\left(\frac{0.1534}{0.4409}\right)^{2} 0.0101 \geq -0.002
\end{align*}

Finally, the last term can be bounded by:

\begin{align*}
-  2\frac{||-\Theta^3_A B_{sj}+e||_{L^2}}{\|B_{sj}\|_{L^{2}}^{2}(c^*-\lambda_0)}\|e\|_{L^{2}} \geq (-2) \frac{0.1534}{0.4409} 0.0905 \geq -0.013
\end{align*}

Adding all the numbers we get the desired positivity result. This finishes the proof.

\end{proof}

By Lemma \ref{c} and Lax-Milgram theorem we find $g^3\in (B_*^3)^\perp$ such that
\begin{align*}
\langle (\Theta^3-\lambda_3)g^3, v\rangle =\left\langle \frac{1}{3}G^3, v\right\rangle
\end{align*}
for all $v\in (B_*^3)^\perp$ and then, there exists a real $\gamma \in \RR$ such that
\begin{align*}
(\Theta^3-\lambda_3)g^3=\frac{1}{3} G^3+\gamma B_*^3.
\end{align*}
But taking scalar product with $B_*^3$ we have that
\begin{align*}
||B_*^3||_{L^2}^2\gamma= \langle (\Theta^3-\lambda_3)g^3, B^3_* \rangle =\langle g^3, (\Theta^{3*}-\lambda_3)B^3_*\rangle=0.
\end{align*}
This last equality implies that $\gamma=0$ and therefore
\begin{align*}
(\Theta^3-\lambda_3)g^3= \frac{1}{3}G^3.
\end{align*}
However the equation
\begin{align}\label{notiene}
(\Theta^3-\lambda_3)g^3=B_*^3
\end{align}
does not have any solution in $L^2$. In order to check it let us assume that there exists $g^3\in L^2$ such that the equation \eqref{notiene} is satisfied. Multiplying \eqref{notiene} by $B^3_*$ and integrating yields
\begin{align*}
\langle (\Theta^3-\lambda_3)g^3, B_*^3 \rangle =||B_*^3||_{L^2}^2=\langle g^3,(\Theta^{3*}-\lambda_3)B_*^3\rangle=0,
\end{align*}
which is a contradiction.
\end{proof}

\item \textbf{$H^3-$regularity.} We again deal separately with two cases: in Lemma \ref{rmm1} we take $m>1$ and in Lemma \ref{rmi1} $m=1$.

\begin{lemma}\label{rmm1} The solution $g^{3m}\in L^2$ to the equation \eqref{gG} given by Lemma \ref{emm1} is actually in $H^3$ with the bound $$||g^{3m}||_{H^3}\leq \frac{C}{3m}||G^{3m}||_{H^3},$$ where $C$ does not depend on $m$.
\end{lemma}
\begin{proof}
For $m>1$ let us consider the equation \eqref{gG}. We split the proof in two steps: in the first one we will show that $g^{3m}\in H^3$ but its $H^3-$norm will depend on $m$; in the second one we will prove that the $H^3-$norm is actually independent of $m$.
\begin{enumerate}
\item Step 1.  Since $\Theta^{3m}g^{3m}-\lambda_3 g^{3m}=\frac{1}{3m}G^{3m}$ and $\frac{1}{3m}G^{3m}$ is $H^1$, we can take a derivative on both sides to obtain
\begin{align*}
(\tilde{I}(\rho)-\lambda_3)\pa_\rho g^{3m}(\rho)=-\frac{1}{3m}\pa_\rho G^{3m}(\rho)-\pa_\rho\tilde{I}(\rho)g^{3m}(\rho)-\pa_\rho \tilde{T}^{3m}g^{3m},
\end{align*}
where we remark that $\pa_\rho \tilde{T}^{3m}g^{3m}\in L^2$ since we know that $g^{3m}\in L^2$ and $\tilde{T}^{3m}\,:\, L^2\to H^1$. The problem here is that $||\tilde{T}^{3m}||_{L^2\to H^1}$ depends on $m$. In addition, since $\tilde{I}(\rho)\in C^3$ and $\tilde{I}-\lambda_3>0$, we have that $\pa_\rho g^{3m}$ is $L^2$ with norm bounded by a constant depending on $m$. Taking 2 derivatives in the equation \eqref{gG} we have that
\begin{align*}
(\tilde{I}(\rho)-\lambda_3)\pa^{2}_\rho g^{3m}(\rho)=\frac{1}{3m}\pa_\rho G^{3m}(\rho)-\sum_{j=1}^2\pa^{j)}_\rho\tilde{I}(\rho)\pa^{2-j)}_\rho g^{3m}-\pa^{2}_{\rho}\tilde{T}^{3m}g^{3m}(\rho).
\end{align*}
By Lemma \ref{Tderivadas} we know that
\begin{align*}
\pa_\rho \tilde{T}^{3m}g^{3m}(\rho)=&\tilde{T}^{3m}_1\pa_\rho g^{3m} (\rho)+\frac{1}{\rho}\inti \pa^{2}_\rho f(\rho') g^{3m}(\rho)(\rho')\frac{\rho'}{\rho}K^{3m}\left(\frac{1}{\gamma}\right)d\rho'.
\end{align*}
Thus we can write
\begin{align*}
\pa^{2}_\rho \tilde{T}^{3m}g^{3m}(\rho)=&\pa_\rho\tilde{T}^{3m}_1\pa_\rho g^{3m} (\rho)+\pa_\rho \frac{1}{\rho}\inti \pa^{2}_\rho f(\rho') g^{3m}(\rho)(\rho')\frac{\rho'}{\rho}K^{3m}\left(\frac{1}{\gamma}\right)d\rho'.
\end{align*}
Then, by Lemma \ref{lemacompacidad} we know that $\pa_\rho\tilde{T}^{3m}_1\pa_\rho g^{3m}\in H^3$ (with norm depending on $m$) and therefore $\pa^{2}_\rho \tilde{T}^{3m}g^{3m}(\rho)\in L^2$. Again, using that $\tilde{I}\in C^3$, $\tilde{I}(\rho)-\lambda_3>0$, and $g^{3m}\in H^2$ we have that $\pa^{2}_\rho g^{3m}\in L^2$ with $L^2-$norm depending on $m$. Finally since
\begin{align*}
(\tilde{I}(\rho)-\lambda_3)\pa^{3}_\rho g^{3m}=\frac{1}{3m}\pa^{3}_\rho G^{3m}(\rho)-\sum_{j=1}^3 \pa^{j)}_\rho \tilde{I}(\rho)\pa^{3-j)}_\rho g^{3m}(\rho)-\pa^{3}_\rho \tilde{T}^{3m}g^{3m}(\rho)
\end{align*}
and
\begin{align*}
\pa^{2}_\rho \tilde{T}^{3m}g^{3m}(\rho)= & T^{3m}_2\pa^{2}_\rho B(\rho)+\frac{1}{\rho}\sum_{j=1}^2\inti \pa^{j+1)}_\rho f(\rho')\pa^{2-j)}_\rho B(\rho')\left(\frac{\rho'}{\rho}\right)^2K^{3m}\left(\frac{\rho}{\rho'}\right)d\rho',
\end{align*}
and $\tilde{T}^{3m}_2$ is compact from $L^2\to H^1$, a similar argument yields that $\pa^{3}_\rho g^{3m}\in L^2$ with $L^2-$norm depending on $m$.

\item Step 2.
Now, taking one, two and three derivatives on the equation \eqref{gG} and applying Lemma \ref{Tderivadas}  we have that,
\begin{align}
(\Theta^{3m}-\lambda_3)\pa_\rho g^{3m}(\rho)=&\frac{1}{3m}\pa_\rho G^{3m}(\rho)\nonumber-\pa_\rho \tilde{I}(\rho) g^{3m}(\rho)-\frac{1}{\rho}\inti \pa^{2}_\rho f(\rho') g^{3m}(\rho')\frac{\rho'}{\rho}K^{3m}\left(\frac{1}{\gamma}\right)d\rho'\nonumber\\ &+\frac{1}{\rho}\inti f_\rho(\rho')\pa_\rho g^{3m}(\rho')\left(1-\frac{\rho'}{\rho}\right)K^{m}\left(\frac{1}{\gamma}\right)d\rho' \label{primera}\\
(\Theta^{3m}-\lambda_3)\pa_\rho^{2} g^{3m}(\rho)=&\frac{1}{3m}\pa^{2}_\rho G^{3m}(\rho)-\sum_{j=1}^2 \pa^{j)}_\rho \tilde{I}(\rho)\pa^{2-j)}_\rho g^{3m}(\rho)\nonumber\\
&-\frac{1}{\rho}\sum_{j=1}^2\inti \pa^{j+1)}_\rho f(\rho')\pa^{2-j)}_\rho g^{3m}(\rho')\left(\frac{\rho'}{\rho}\right)^2K^{3m}\left(\frac{\rho}{\rho'}\right)d\rho'\nonumber\\
&+\frac{1}{\rho}\inti f_\rho(\rho')\pa_\rho^2 g^{3m}(\rho')\left(1-\left(\frac{\rho'}{\rho}\right)^2\right)K^{m}\left(\frac{\rho}{\rho'}\right)d\rho'\label{segunda}\\
(\Theta^{3m}-\lambda_3)\pa_\rho^{3} g^{3m}(\rho)=&\frac{1}{3m}\pa^{3}_\rho G^{3m}(\rho)-\sum_{j=1}^3 \pa^{j)}_\rho \tilde{I}(\rho)\pa^{3-j)}_\rho g^{3m}(\rho)\nonumber\\
&-\frac{1}{\rho}\sum_{j=1}^3\inti \pa^{j+1)}_\rho f(\rho')\pa^{3-j)}_\rho g^{3m}(\rho')\left(\frac{\rho'}{\rho}\right)^3K^{3m}\left(\frac{\rho}{\rho'}\right)d\rho'\nonumber\\
& +\frac{1}{\rho}\inti f_\rho(\rho')\pa_\rho^3 g^{3m}(\rho')\left(1-\left(\frac{\rho'}{\rho}\right)^3\right)K^{m}\left(\frac{\rho}{\rho'}\right)d\rho'\label{tercera}.
\end{align}

Since $g^{3m}\in H^3$, by the step 1 above, the coercivity property in  Proposition \ref{thetaicoer} applies. We first apply it to \eqref{primera}, then to \eqref{segunda} and \eqref{tercera} yielding the bound
$$||g^{3m}||_{H^3} \leq \frac{C}{m}||G^{3m}||_{H^3}$$
with $C$ independent of $m$. The only problem comes from the last term on \eqref{primera}, \eqref{segunda} and \eqref{tercera}. In order to bound these terms we apply an integration by parts, Lemma \ref{commutator} and the same argument that we used in Lemma \ref{lemacompacidad} to control the $L^2-$norm.
\end{enumerate}
\end{proof}
\begin{lemma}\label{rmi1}
The solution, $g^3\in L^2$ of the equation $(\Theta^3-\lambda_3)g^3=\frac{1}{3}G^3$ with $G^{3}\in H^3\cap (B_*^3)^\perp$ given by Lemma \ref{emi1} is actually in $H^3$.
\end{lemma}
\begin{proof}
 We can show that $g^3\in H^3$ in the same way we did in the proof of the first part of Lemma \ref{rmm1}.
\end{proof}
\end{enumerate}
Then Proposition \ref{codimension} is already proven.
\end{proof}
\subsection{Step 4. The transversality property 4}
In this section we prove the transversality condition \eqref{condTransversality} of the C-R theorem, i.e., the fourth hypothesis. In order to do this is enough to show that
\begin{align}
\label{eqtransversal}
 (\Theta^{3} - \lambda_3)b^{3} = B^{3}
\end{align}
does not have a solution in $b^{3}\in L^2$.

Let's suppose that there exists $b^3\in L^2$ such that \eqref{eqtransversal} holds. Then, taking scalar product in $L^2$ with $B^3_*$, we have that
\begin{align*}
\langle B^3, B^3_*\rangle =0.
\end{align*}
This is impossible as it was proved in Lemma \ref{BBstar}.

This concludes the proof of Theorem \ref{main}.

\section*{Acknowledgements}

AC, DC and JGS were partially supported by the grant MTM2014-59488-P (Spain), MTM2017-89976-P (Spain) and ICMAT Severo Ochoa projects SEV-2011-0087 and SEV-2015-556. AC was partially supported by the Ram\'on y Cajal program RyC-2013-14317, ERC grant 307179-GFTIPFD and Europa Excelencia program ERC 2018-092824. DC was partially supported by a Minerva Distinguished
Visitorship at Princeton University, and by the ERC Advanced Grant 788250. JGS was partially supported by an AMS-Simons Travel Grant, by the NSF through Grant NSF DMS-1763356, and by the ERC Starting Grant 852741. We thank Princeton University for computing facilities. Part of this work was done while some the authors were visiting Princeton University or ICMAT, to which they are grateful for their support.

Supplementary material has been provided and can be found at \url{https://arxiv.org/e-print/1603.03325}. The downloaded file should be renamed to, for example, 1603.03325.tar.gz, as it is in an archived format.

\appendix

\section{Asymptotics}
\label{sectionasymptotics}

Part of the computer-assisted proof involves having to compute the kernels $T^{1}, T^{3}$ and $T^{6}$, which are given by elliptic integrals. As far as we know, we are not aware of any rigorous implementation of them in any library. One possibility could be to leave the (singular) integrals as they are and integrate over a domain of one more dimension. This would be very time consuming in terms of the computer performance. Instead, we do the laborious work of deriving explicit approximations (to order 0) of the kernels by hand, with computable error bounds of order greater than 1. Once we do this, whenever we have to code either $T^{1}, T^{3}$ or $T^{6}$, we substitute it by the explicit expression found here.

We start with the elliptic integral
\begin{align*}
 I & = \int_{-\pi}^{\pi} \frac{\cos(x)}{\sqrt{1+r^2-2r\cos(x)}}dx.
\end{align*}

Taking $\displaystyle u = \frac{1-r}{1+r}$:

\begin{align*}
I & = \frac{1}{1+r} \int_{-\pi}^{\pi} \frac{\cos(x)dx}{\sqrt{\sin^{2}\left(\frac{x}{2}\right) + u^{2}\cos^{2}\left(\frac{x}{2}\right)}}= \frac{4}{1+r}\int_{0}^{\frac{\pi}{2}} \frac{\cos(2y)dy}{\sqrt{\sin^{2}(y)+u^{2}\cos^{2}(y)}} \\& = \frac{4}{1+r}\int_{0}^{\frac{\pi}{2}} \frac{1-2\sin^{2}(y)}{\sqrt{\sin^{2}(y)+u^{2}\cos^{2}(y)}}dy  = \left\{
\begin{array}{ccc}
z & = & \tan(y) \\
\frac{dz}{1+z^{2}} & = & dy
\end{array}
\right\} = \frac{4}{1+r}\int_{0}^{\infty} \frac{1-z^{2}}{(1+z^{2})^{3/2}} \frac{dz}{\sqrt{z^{2}+u^{2}}}.
\end{align*}

We remark that $u$ will be close to zero. We need to derive the asymptotics in powers of $u$ as $u \rightarrow 0$ of

\begin{align*}
I & = \frac{4}{1+r} \int_{0}^{\infty} \frac{1-z^2}{(1+z^2)^{3/2}} \frac{dz}{\sqrt{z^2+u^2}} \\
& = \frac{4}{1+r} \int_{0}^{|u|} \frac{1-z^2}{(1+z^2)^{3/2}} \frac{dz}{\sqrt{z^2+u^2}}
 + \frac{4}{1+r} \int_{|u|}^{1} \frac{1-z^2}{(1+z^2)^{3/2}} \frac{dz}{\sqrt{z^2+u^2}}
 + \frac{4}{1+r} \int_{1}^{\infty} \frac{1-z^2}{(1+z^2)^{3/2}} \frac{dz}{\sqrt{z^2+u^2}} \\
& = I_1 + I_2 + I_3.
\end{align*}

We start with $I_1$:

\begin{align*}
I_1 & = \frac{4}{1+r} \int_{0}^{1} \frac{1-u^2 w^2}{(1+u^2 w^2)^{3/2}} \frac{dw}{\sqrt{1+w^2}}.
\end{align*}

We expand $(1+u^2 w^2)^{-3/2}$ as a power series around $w = 0$:

\begin{align*}
(1+u^2 w^2)^{-3/2} = 1 - \frac32 u^2w^2 + \frac{1}{2}\frac{15}{4}(u^2 w^2)^2E^{1}_{2},
\end{align*}

where $(1 + u^2)^{-7/2}  \leq E^{1}_2 \leq 1$. A naive integration and bounding then yields:

\begin{align*}
I_1 & = \frac{4}{1+r}\left(\int_{0}^{1} \frac{1-u^2 w^2}{\sqrt{1+w^2}} dw - \frac32 u^2 \int_{0}^{1} \frac{1-u^2 w^2}{\sqrt{1+w^2}} w^2 dw + \frac{15}{8} u^4 \tilde{E}^{1}_2\right) \\
& = \frac{4}{1+r}\left(\arcsinh(1) + u^2 \left(\frac{5}{4} (-\sqrt{2} + \arcsinh(1))\right) + u^4 \left(-\frac{3}{16}\left(\sqrt{2}-3\arcsinh(1)\right) + \frac{15}{8}\tilde{E}^{1}_2\right)\right),
\end{align*}

where

\begin{align*}
\frac{1}{\sqrt{2}}\frac{1-u^2}{(1+u^2)^{7/2}}\leq \tilde{E}^{1}_{2} \leq 1.
\end{align*}

Next, we proceed with $I_3$. We can write it as:

\begin{align*}
I_3 & = \frac{4}{1+r} \int_{0}^{1} \frac{w^2-1}{(w^2+1)^{3/2}} \frac{dw}{\sqrt{1+w^2u^2}}.
\end{align*}

Expanding in series $(1+w^2u^2)^{-1/2}$:

\begin{align*}
(1+u^2 w^2)^{-1/2} = 1 - \frac12 u^2w^2 + \frac{1}{2}\frac{3}{4}(u^2 w^2)^2E^{2}_{2},
\end{align*}

where $\frac{1}{(1+u^2)^{5/2}} \leq E^{2}_2 \leq 1$. A naive integration and bounding then yields:

\begin{align*}
I_3 & = \frac{4}{1+r}\left(\int_{0}^{1} \frac{w^2-1}{(w^2+1)^{3/2}}dw - \frac12 u^{2} \int_{0}^{1} w^2\frac{w^2-1}{(w^2+1)^{3/2}}dw + \frac{3}{8} u^4 \tilde{E}^{2}_{2}\right) \\
& = \frac{4}{1+r}\left(\arcsinh(1)-\sqrt{2} + \frac{1}{4} u^2\left(-3\sqrt{2}+5\arcsinh(1)\right) + \frac{3}{8} u^4 \tilde{E}^{2}_{2}\right),
\end{align*}

and

\begin{align*}
-1 \leq \tilde{E}^{2}_{2} \leq 0.
\end{align*}

Finally, we do the asymptotics for $I_2$. This is the most careful part. We start by writing

\begin{align*}
I_2 & = \frac{4}{1+r} \frac12 \int_{u^{2}}^{1} \frac{1}{w} \frac{1-w}{(1+w)^{3/2}} \frac{dw}{\sqrt{1+\frac{u^{2}}{w}}} \\
& = \frac{4}{1+r} \frac12 \int_{u^{2}}^{1} \frac{1}{w} \frac{1}{(1+w)^{3/2}} \frac{dw}{\sqrt{1+\frac{u^{2}}{w}}} - \frac{4}{1+r} \frac12 \int_{u^{2}}^{1} \frac{1}{w} \frac{w}{(1+w)^{3/2}} \frac{dw}{\sqrt{1+\frac{u^{2}}{w}}}\\ &
 = I_{21} + I_{22}.
\end{align*}

Next, we expand the term $\left(1+\frac{u^{2}}{w}\right)^{-1/2}$, but we do not truncate. Hence

\begin{align*}
I_{21} & = \frac{2}{1+r}\left(\sum_{k \geq 0}{-1/2 \choose k}\underbrace{ u^{2k} \int_{u^{2}}^{1} \frac{1}{w^{1+k}} \frac{dw}{(1+w)^{3/2}}}_{F_k}\right).
\end{align*}

Using the following integration by parts formula for $k > 0$:

\begin{align*}
\int_{u^{2}}^{1} w^{-k}(1+w)^{-3/2} dw & = \left.(-2)(1+w)^{-1/2}w^{-k}\right|_{w=u^{2}}^{w=1}
-2k\int_{u^{2}}^{1} w^{-k-1}(1+w)^{-1/2} dw \\
& = \left.(-2)(1+w)^{-1/2}w^{-k}\right|_{w=u^{2}}^{w=1}
-2k\int_{u^{2}}^{1} w^{-k-1}(1+w)^{-3/2} dw \\ &-2k\int_{u^{2}}^{1} w^{-k}(1+w)^{-3/2} dw,
\end{align*}

which implies:

\begin{align*}
\int_{u^{2}}^{1} w^{-k-1}(1+w)^{-3/2} dw & = \left.-\frac{1}{k}(1+w)^{-1/2}w^{-k}\right|_{w=u^{2}}^{w=1}
- \frac{2k+1}{2k} \int_{u^{2}}^{1}w^{-k}(1+w)^{-3/2} dw,
\end{align*}

yielding

\begin{align*}
F_{k} & = \frac{1}{k}\left(\frac{1}{\sqrt{1+u^{2}}} - \frac{u^{2k}}{\sqrt{2}}\right) - u^{2}\left(\frac{2k+1}{2k}\right)F_{k-1}.
\end{align*}

We now get back to $I_{21}$. We have to compute

\begin{align*}
I_{21} & = \frac{2}{1+r}\left(\sum_{k \geq 0} {-1/2 \choose k} F_{k}\right) =\\ & \frac{2}{1+r}\left(F_{0} + \sum_{k \geq 1}{-1/2 \choose k} \frac{1}{k}\left(\frac{1}{\sqrt{1+u^{2}}} - \frac{u^{2k}}{\sqrt{2}}\right) - u^{2} \sum_{k\geq 1}{-1/2 \choose k} \left(\frac{2k+1}{2k}\right)F_{k-1}\right).
\end{align*}

We calculate the following explicit numbers:

\begin{align*}
F_0 & = \sqrt{2} - \frac{2}{\sqrt{1+u^{2}}} - 2\arcsinh(1) + 2\arcsinh\left(\frac{1}{|u|}\right). \\
\sum_{k \geq 1}{-1/2 \choose k} \frac{1}{k}\left(\frac{1}{\sqrt{1+u^{2}}}-\frac{u^{2k}}{\sqrt{2}}\right) & = \frac{-2\arcsinh(1)+\log(4)}{\sqrt{1+u^{2}}} + \sqrt{2}\log\left(\frac{1}{2}(1+\sqrt{1+u^{2}})\right).
\end{align*}

We treat the rightmost sum as an error. Using the fact that the terms are alternating (since $\displaystyle {-1/2 \choose k}$ is alternating and the other factors are positive), and that, for $k \geq 1$:

\begin{align*}
 F_{k-1} = \frac{1}{u^{2}} \int_{u^{2}}^{1} \left(\frac{u^{2}}{w}\right)^{k} \frac{dw}{(1+w)^{3/2}} \geq \frac{1}{u^{2}} \int_{u^{2}}^{1} \left(\frac{u^{2}}{w}\right)^{k+1} \frac{dw}{(1+w)^{3/2}} = F_{k} \quad \forall \quad 0 < u^{2} < 1,
\end{align*}

\begin{align*}
 \left| {-1/2 \choose k} \right| \frac{2k+1}{2k} = \frac{2 (1 + k)^2}{k (3 + 2 k)} \left| {-1/2 \choose k+1} \right| \frac{2(k+1)+1}{2(k+1)}
\geq \left| {-1/2 \choose k+1} \right| \frac{2(k+1)+1}{2(k+1)}
\end{align*}

we can bound the absolute value of the sum by the absolute value of its first term, yielding:

\begin{align*}
 \left| u^{2} \sum_{k\geq 1}{-1/2 \choose k} \left(\frac{2k+1}{2k}\right)F_{k-1}\right| \leq \frac{3}{4} u^{2}|F_{0}|.
\end{align*}

%

We move on to $I_{22}$. We can write it as:

\begin{align*}
I_{22} & = \frac{2}{1+r}\left(-u^{2}F_{-1} - u^{2}\sum_{k \geq 1} {-1/2 \choose k} F_{k-1}\right).
\end{align*}

Together with the explicit calculation

\begin{align*}
-u^{2}F_{-1} & = \sqrt{2} - \frac{2}{\sqrt{1+u^{2}}},
\end{align*}

and the bound on the last term by the same reason as above:

\begin{align*}
\left|u^{2}\sum_{k \geq 1} {-1/2 \choose k} F_{k-1}\right| \leq \frac12 u^{2}|F_{0}|
\end{align*}

we can conclude, after gathering all the contributions, that $I$ splits in the following way:




\begin{align*}
 I & = I_{ho,ns} + I_{ho,s} + I_{e,s} + I_{e,ns},
\end{align*}

where

\begin{align*}
 I_{ho,ns} & = \frac{4}{1+r}\left(\arcsinh(1) + u^2 \left(\frac{5}{4} (-\sqrt{2} + \arcsinh(1))\right) + u^4 \left(-\frac{3}{16}\left(\sqrt{2}-3\arcsinh(1)\right) \right)\right)\\
& + \frac{4}{1+r}\left(\arcsinh(1)-\sqrt{2} + \frac{1}{4} u^2\left(-3\sqrt{2}+5\arcsinh(1)\right)\right) \\
& + \frac{2}{1+r}\left(\sqrt{2} - \frac{2}{\sqrt{1+u^{2}}} - 2\arcsinh(1) \right) \\
& + \frac{2}{1+r}\left( \frac{-2\arcsinh(1)+\log(4)}{\sqrt{1+u^{2}}} + \sqrt{2}\log\left(\frac{1}{2}(1+\sqrt{1+u^{2}})\right) \right) \\
& + \frac{2}{1+r}\left(\sqrt{2} - \frac{2}{\sqrt{1+u^{2}}} \right) \\
& = \frac{4}{1+r}\left(\frac{1}{\sqrt{2}}\log\left(\frac{1}{2}(1+\sqrt{1+u^{2}})\right) + \frac{1}{\sqrt{1+u^{2}}}\left(-2+\left(-1+\sqrt{1+u^{2}}\right)\arcsinh(1)+\log(2)\right)\right) \\
& + \frac{4}{1+r}\left(u^2 \left(-2 \sqrt{2} + \frac{5}{2}\arcsinh(1)\right) - \frac{3}{16} u^4 (\sqrt{2} - 3 \arcsinh(1))  \right) \\
I_{ho,s} & = \frac{4}{1+r} \arcsinh\left(\frac{1}{|u|}\right)
\end{align*}

and

\begin{align*}
|I_{e,ns}| & \leq \frac{4}{1+r}\frac{15}{8} u^{4} + \frac{4}{1+r}\frac{3}{8} u^4  + \frac{2}{1+r}\frac{5}{4}u^{2}\left|\sqrt{2} - \frac{2}{\sqrt{1+u^{2}}} - 2\arcsinh(1)\right| \\
& = \frac{4}{1+r}\frac{9}{4} u^{4}  + \frac{4}{1+r}\frac{5}{8}u^{2}\left|\sqrt{2} - \frac{2}{\sqrt{1+u^{2}}} - 2\arcsinh(1)\right| \\
 |I_{e,s}| & \leq \frac{4}{1+r}\frac{5}{4}u^{2} \arcsinh\left(\frac{1}{|u|}\right).
\end{align*}

Note that in the computer implementation any error term $I_{e}$ will be implemented as the interval $[-M,M]$ whenever $|I_{e}| \leq M$.

We do the same for the other elliptic integral:

\begin{align*}
 J = \int_{-\pi}^{\pi} \frac{\cos(3x)}{\sqrt{1+s^2-2s\cos(x)}}dx.
\end{align*}

Taking $\displaystyle u = \frac{1-r}{1+r}$:

\begin{align*}
J & = \frac{1}{1+r} \int_{-\pi}^{\pi} \frac{\cos(3x)dx}{\sqrt{\sin^{2}\left(\frac{x}{2}\right) + u^{2}\cos^{2}\left(\frac{x}{2}\right)}} = \frac{4}{1+r}\int_{0}^{\frac{\pi}{2}} \frac{\cos(6y)dy}{\sqrt{\sin^{2}(y)+u^{2}\cos^{2}(y)}}  \\
& = \left\{
\begin{array}{ccc}
z & = & \tan(y) \\
\frac{dz}{1+z^{2}} & = & dy
\end{array}
\right\} = \frac{4}{1+r}\int_{0}^{\infty} \frac{1-15z^2+15z^4-z^6}{(1+z^{2})^{7/2}} \frac{dz}{\sqrt{z^{2}+u^{2}}}.
\end{align*}

We repeat the splitting that we did for $I$, this time for $J$:

\begin{align*}
J & = \frac{4}{1+r} \int_{0}^{\infty} \frac{1-15z^2+15z^4-z^6}{(1+z^2)^{7/2}} \frac{dz}{\sqrt{z^2+u^2}} \\
& = \frac{4}{1+r} \int_{0}^{|u|} \frac{1-15z^2+15z^4-z^6}{(1+z^2)^{7/2}} \frac{dz}{\sqrt{z^2+u^2}}
 + \frac{4}{1+r} \int_{|u|}^{1} \frac{1-15z^2+15z^4-z^6}{(1+z^2)^{7/2}} \frac{dz}{\sqrt{z^2+u^2}} \\
 & + \frac{4}{1+r} \int_{1}^{\infty} \frac{1-15z^2+15z^4-z^6}{(1+z^2)^{7/2}} \frac{dz}{\sqrt{z^2+u^2}}
= J_1 + J_2 + J_3.
\end{align*}

We start with $J_1$:

\begin{align*}
J_1 & = \frac{4}{1+r} \int_{0}^{1} \frac{1-15 u^2 w^2+15 u^4 w^4- u^6w^6}{(1+u^2 w^2)^{7/2}} \frac{dw}{\sqrt{1+w^2}}.
\end{align*}

We expand $(1+u^2 w^2)^{-7/2}$ as a power series around $w = 0$:

\begin{align*}
(1+u^2 w^2)^{-7/2} = 1 - \frac{7}{2} u^2w^2 F^{1}_{2},
\end{align*}

where $(1 + u^2)^{-9/2}  \leq F^{1}_2 \leq 1$. A naive integration and bounding then yields:
\begin{align*}
J_1 & = \frac{4}{1+r}\left(\int_{0}^{1} \frac{1-15 u^2 w^2+15 u^4 w^4- u^6w^6}{\sqrt{1+w^2}} dw - \frac{7}{2} u^2 \tilde{F}^{1}_2\right) \\
& = \frac{4}{1+r}\left(\arcsinh(1) - u^2 \left(\frac{15}{2} (\sqrt{2} - \arcsinh(1))\right) - u^4 \left(\frac{15}{8}\left(\sqrt{2}-3\arcsinh(1)\right)\right)\right. \\
& \left. - \frac{1}{48}u^{6}\left(13 \sqrt{2} - 15 \arcsinh(1)\right) - \frac{7}{2}u^{2}\tilde{F}^{1}_2\right),
\end{align*}


where

\begin{align*}
0 \leq \tilde{F}^{1}_{2} \leq 1,
\end{align*}

and we have used that $|u|$ is small enough to guarantee the positiveness of the integrand.

Next, we proceed with $J_3$. We can write it as:

\begin{align*}
J_3 & = \frac{4}{1+r} \int_{0}^{1} \frac{w^6-15w^4+15w^2-1}{(w^2+1)^{7/2}} \frac{dw}{\sqrt{1+w^2u^2}}.
\end{align*}

Expanding in series $(1+w^2u^2)^{-1/2}$:

\begin{align*}
(1+u^2 w^2)^{-1/2} = 1 - \frac12 u^2w^2 F^{2}_{2},
\end{align*}

where $\frac{1}{(1+u^2)^{3/2}} \leq F^{2}_2 \leq 1$. We use the fact that

\begin{align*}
\max_{w \in [0,1]} |w^6-15w^4+15w^2-1| \leq 176 - 80 \sqrt{5} < 3,
\end{align*}

to obtain, via integration and bounding:

\begin{align*}
J_3 & = \frac{4}{1+r}\left(\int_{0}^{1} \frac{w^6-15w^4+15w^2-1}{(w^2+1)^{7/2}}dw - \frac{1}{2} u^2 \tilde{F}^{2}_{2}\right) \\
& = \frac{4}{1+r}\left(\arcsinh(1)-\frac{7}{15}\sqrt{2} - \frac{1}{2} u^2 \tilde{F}^{2}_{2}\right),
\end{align*}

and

\begin{align*}
 |\tilde{F}^{2}_{2}| \leq 3.
\end{align*}

We finally move on to $J_2$. As before, we have the following formula obtained by integration by parts:

\begin{align*}
\int_{u^{2}}^{1} w^{-k-1}(1+w)^{-7/2}dw = \left.-\frac{1}{k}w^{-k}(1+w)^{-5/2}\right|_{w=u^{2}}^{w=1} - \left(\frac{2k+5}{2k}\right)\int_{u^{2}}^{1} w^{-k}(1+w)^{-7/2}dw.
\end{align*}

Defining now

\begin{align*}
G_{k} & = u^{2k} \int_{u^{2}}^{1} \frac{1}{w^{1+k}} \frac{dw}{(1+w)^{7/2}}
\end{align*}

we arrive to

\begin{align}
\label{formulaJ}
G_{k} & = \frac{1}{k}\left(\frac{1}{(1+u^{2})^{5/2}} - \frac{u^{2k}}{\sqrt{32}}\right) - u^{2}\left(\frac{2k+5}{2k}\right)G_{k-1}.
\end{align}

We can write

\begin{align*}
 J_2 & = \frac{2}{1+r}\left(\sum_{k\geq 0}  {-1/2 \choose k} G_{k} - 15 u^{2} \sum_{k\geq 0}  {-1/2 \choose k} G_{k-1} + 15 u^{4}\sum_{k\geq 0}  {-1/2 \choose k} G_{k-2} - u^{6} \sum_{k\geq 0}  {-1/2 \choose k} G_{k-3}\right) \\
& = J_{21} + J_{22} + J_{23} + J_{24}.
\end{align*}

The last three terms are easier. We deal with them first.

\begin{align*}
 J_{24} & = \frac{2}{1+r}\left(-u^{6}G_{-3} - u^{6}\sum_{k\geq 1}  {-1/2 \choose k} G_{k-3}\right).
\end{align*}

The first term can be explicitly calculated and it amounts to

\begin{align*}
 -u^{6}G_{-3} & = -\frac{1}{60}\frac{64 + 160u^{2} + 120 u^{4} - 43 \sqrt{2}(1+u^{2})^{5/2}}{(1+u^{2})^{5/2}}.
\end{align*}

The last series is alternating and convergent and we can get the following bound:

\begin{align*}
 \left| u^{6}\sum_{k\geq 1}  {-1/2 \choose k} G_{k-3}\right| \leq u^{6}\frac{1}{2}G_{-2}
= \frac{1}{120}u^{2}\left|-7\sqrt{2} + \frac{40}{(1+u^{2})^{3/2}} - \frac{24}{(1+u^{2})^{5/2}}\right|.
\end{align*}

We move on to $J_{23}$. By the same reasoning:

\begin{align*}
 J_{23} & = \frac{2}{1+r}\left(15u^{4}G_{-2} + 15 u^{4}\sum_{k\geq 1}  {-1/2 \choose k} G_{k-2}\right) \\
& = \frac{2}{1+r}\left(-\frac{7}{2\sqrt{2}} + \frac{10}{(1+u^{2})^{3/2}} -\frac{6}{(1+u^{2})^{5/2}}  + 15 u^{4}\sum_{k\geq 1}  {-1/2 \choose k} G_{k-2}\right), \\
\end{align*}

and

\begin{align*}
 \left|15 u^{4}\sum_{k\geq 1}  {-1/2 \choose k} G_{k-2}\right| \leq u^{4}\frac{15}{2}G_{-1}
= \frac{3}{8}u^{2}\left|-\sqrt{2} + \frac{8}{(1+u^{2})^{5/2}}\right|.
\end{align*}

The same applies to $J_{22}$. We have:

\begin{align*}
 J_{22} & = \frac{2}{1+r}\left(-15u^{2}G_{-1} - 15 u^{2}\sum_{k\geq 1}  {-1/2 \choose k} G_{k-1}\right) \\
& = \frac{2}{1+r}\left(\frac{3}{4}\sqrt{2} -\frac{6}{(1+u^{2})^{5/2}}  - 15 u^{2}\sum_{k\geq 1}  {-1/2 \choose k} G_{k-1}\right) \\
\end{align*}

and

\begin{align*}
 &\left|-15 u^{2}\sum_{k\geq 1}  {-1/2 \choose k} G_{k-1}\right|  \leq u^{2}\frac{15}{2}G_{0} \\
& = \frac{1}{8}u^{2}\left|73\sqrt{2} - \frac{24}{(1+u^{2})^{5/2}} - \frac{40}{(1+u^{2})^{3/2}} - \frac{120}{(1+u^{2})^{1/2}} - 120 \arcsinh(1) + 120\arcsinh\left(\frac{1}{|u|}\right)\right|.
\end{align*}

In order to compute $J_{21}$ we use formula \eqref{formulaJ}:

\begin{align*}
\sum_{k\geq 0}  {-1/2 \choose k} G_{k} & = G_{0} + \sum_{k\geq 1}  {-1/2 \choose k} G_{k} \\
& = G_{0} + \sum_{k\geq 1}  {-1/2 \choose k} \frac{1}{k}\left(\frac{1}{(1+u^{2})^{5/2}} - \frac{u^{2k}}{\sqrt{32}}\right) -  u^{2} \sum_{k\geq 1}  {-1/2 \choose k}\left(\frac{2k+5}{2k}\right)G_{k-1}\\
& = \frac{73}{30\sqrt{2}} - \frac{2}{\sqrt{1+u^{2}}} - \frac{2}{3}\frac{1}{(1+u^2)^{3/2}} - \frac{2}{5} \frac{1}{(1+u^2)^{5/2}} - 2 \arcsinh(1) + 2 \arcsinh\left(\frac{1}{|u|}\right) \\
& + \frac{1}{8}\left(\frac{16(\log(2) - \arcsinh(1))}{(1+u^2)^{5/2}} - \sqrt{2}\log(4) + 2\sqrt{2}\log(1+\sqrt{1+u^2})\right) \\
& -  u^{2} \sum_{k\geq 1}  {-1/2 \choose k}\left(\frac{2k+5}{2k}\right)G_{k-1}.
\end{align*}

The last sum can be bounded as usual by

\begin{align*}
& \left|u^{2} \sum_{k\geq 1}  {-1/2 \choose k}\left(\frac{2k+5}{2k}\right)G_{k-1}\right| \leq \frac{1}{2}\frac{7}{2}u^{2}G_{0} \\
& = \frac{7}{4}u^{2}\left|\frac{73}{30\sqrt{2}} - \frac{2}{\sqrt{1+u^{2}}} - \frac{2}{3}\frac{1}{(1+u^2)^{3/2}} - \frac{2}{5} \frac{1}{(1+u^2)^{5/2}} - 2 \arcsinh(1) + 2 \arcsinh\left(\frac{1}{|u|}\right)\right|.
\end{align*}

We finally add everything together to write $J$ as

\begin{align*}
J & = J_{ho,ns} + J_{ho,s} + J_{e,ns} + J_{e,s}
\end{align*}

where




\begin{align*}
J_{ho,ns}
& =
\frac{4}{1+r}\left(\arcsinh(1) - u^2 \left(\frac{15}{2} (\sqrt{2} - \arcsinh(1))\right)\right.\\ & \quad\qquad\qquad \left. - u^4 \left(\frac{15}{8}\left(\sqrt{2}-3\arcsinh(1)\right)\right)- \frac{1}{48}u^{6}\left(13 \sqrt{2} - 15 \arcsinh(1)\right)\right) \\
& +\frac{4}{1+r}\left(\arcsinh(1)-\frac{7}{15}\sqrt{2} \right) \\
& + \frac{2}{1+r}\left(-\frac{1}{60}\frac{64 + 160 u^2 + 120 u^4 - 43 \sqrt{2}(1 + u^2)^{5/2}}{(1+u^2)^{5/2}}\right) \\
& + \frac{2}{1+r}\left(-\frac{7}{2\sqrt{2}} - \frac{6}{(1+u^2)^{5/2}} + \frac{10}{(1+u^2)^{3/2}}\right) \\
& + \frac{2}{1+r}\left(\frac{3}{4}\left(\sqrt{2} - \frac{8}{(1+u^2)^{5/2}}\right)\right) \\
& + \frac{2}{1+r}\left(\frac{1}{8}\left(\frac{16(\log(2) - \arcsinh(1))}{(1+u^2)^{5/2}} - \sqrt{2}\log(4) + 2\sqrt{2}\log(1+\sqrt{1+u^2})\right)\right) \\
& + \frac{2}{1+r}\left(\frac{73}{30\sqrt{2}} - \frac{2}{\sqrt{1+u^{2}}} - \frac{2}{3}\frac{1}{(1+u^2)^{3/2}} - \frac{2}{5} \frac{1}{(1+u^2)^{5/2}} - 2 \arcsinh(1) \right) \\
& = \frac{4}{1+r}\left(-\frac{46}{15(1+u^2)^{5/2}} - \frac{1}{48}u^{6}(13\sqrt{2}-15\arcsinh(1)) + \arcsinh(1)\right) \\
& + \frac{4}{1+r}\left(-\frac{\arcsinh(1)}{(1+u^2)^{5/2}} + u^{4}\left(-\frac{15}{4\sqrt{2}} - \frac{2}{(1+u^2)^{5/2}} + \frac{45}{8}\arcsinh(1)\right)\right) \\
& + \frac{4}{1+r}\left(\frac{1}{6}u^{2}\left(-45\sqrt{2} + \frac{8}{(1+u^{2})^{5/2}} + 45 \arcsinh(1)\right)\right) \\
& + \frac{4}{1+r}\left(\frac{\log(2)}{(1+u^2)^{5/2}} + \frac{\log\left(\frac{1}{2}\left(1+\sqrt{1+u^2}\right)\right)}{4\sqrt{2}}\right) \\
J_{ho,s} & = \frac{4}{1+r} \arcsinh\left(\frac{1}{|u|}\right)
\end{align*}

and


\begin{align*}
|J_{e,ns}| & \leq \frac{4}{1+r}\left(\frac{7}{2}u^2 + \frac{3}{2} u^2\right)\\
& + \frac{4}{1+r}\left(\frac{1}{240}u^2\left|-7 \sqrt{2} - \frac{24}{(1+u^2)^{5/2}} + \frac{40}{(1+u^2)^{3/2}}\right|\right) \\
& + \frac{4}{1+r}\left(\frac{3}{16}u^2\left|-\sqrt{2} + \frac{8}{(1+u^2)^{5/2}}\right|\right) \\
& + \frac{4}{1+r}\left(\frac{37}{8}u^2\left|\frac{73}{30\sqrt{2}} - \frac{2}{\sqrt{1+u^{2}}} - \frac{2}{3}\frac{1}{(1+u^2)^{3/2}} - \frac{2}{5} \frac{1}{(1+u^2)^{5/2}} - 2 \arcsinh(1)\right|\right) \\
|J_{e,s}| & \leq \frac{4}{1+r}\frac{37}{4}u^{2} \arcsinh\left(\frac{1}{|u|}\right).
\end{align*}

Finally, we do $K_6$, corresponding to:

\begin{align*}
 L = \int_{-\pi}^{\pi} \frac{\cos(6x)}{\sqrt{1+s^2-2s\cos(x)}}dx.
\end{align*}

Taking $\displaystyle u = \frac{1-r}{1+r}$:

\begin{align*}
L & = \frac{1}{1+r} \int_{-\pi}^{\pi} \frac{\cos(6x)dx}{\sqrt{\sin^{2}\left(\frac{x}{2}\right) + u^{2}\cos^{2}\left(\frac{x}{2}\right)}} = \frac{4}{1+r}\int_{0}^{\frac{\pi}{2}} \frac{\cos(12y)dy}{\sqrt{\sin^{2}(y)+u^{2}\cos^{2}(y)}}  \\
& = \left\{
\begin{array}{ccc}
z & = & \tan(y) \\
\frac{dz}{1+z^{2}} & = & dy
\end{array}
\right\} = \frac{4}{1+r}\int_{0}^{\infty} \frac{1 - 66 z^2 + 495 z^4 - 924 z^6 + 495 z^8 - 66 z^{10} + z^{12}}{(1+z^{2})^{13/2}} \frac{dz}{\sqrt{z^{2}+u^{2}}}.
\end{align*}

We repeat the splitting again:

\begin{align*}
L & = \frac{4}{1+r} \int_{0}^{\infty} \frac{1 - 66 z^2 + 495 z^4 - 924 z^6 + 495 z^8 - 66 z^{10} + z^{12}}{(1+z^2)^{13/2}} \frac{dz}{\sqrt{z^2+u^2}} \\
& = \frac{4}{1+r} \int_{0}^{|u|} \frac{1 - 66 z^2 + 495 z^4 - 924 z^6 + 495 z^8 - 66 z^{10} + z^{12}}{(1+z^2)^{13/2}} \frac{dz}{\sqrt{z^2+u^2}} \\
& + \frac{4}{1+r} \int_{|u|}^{1} \frac{1 - 66 z^2 + 495 z^4 - 924 z^6 + 495 z^8 - 66 z^{10} + z^{12}}{(1+z^2)^{13/2}} \frac{dz}{\sqrt{z^2+u^2}} \\
 & + \frac{4}{1+r} \int_{1}^{\infty} \frac{1 - 66 z^2 + 495 z^4 - 924 z^6 + 495 z^8 - 66 z^{10} + z^{12}}{(1+z^2)^{13/2}} \frac{dz}{\sqrt{z^2+u^2}}
= L_1 + L_2 + L_3.
\end{align*}

We start with $L_1$:

\begin{align*}
L_1 & = \frac{4}{1+r} \int_{0}^{1} \frac{1 - 66 u^2w^2 + 495 u^4w^4 - 924 u^6w^6 + 495 u^8w^8 - 66 u^{10}w^{10} + u^{12}w^{12}}{(1+u^2 w^2)^{13/2}} \frac{dw}{\sqrt{1+w^2}}.
\end{align*}

We expand $(1+u^2 w^2)^{-13/2}$ as a power series around $w = 0$:

\begin{align*}
(1+u^2 w^2)^{-13/2} = 1 - \frac{13}{2} u^2w^2 G^{1}_{2},
\end{align*}

where $(1 + u^2)^{-15/2}  \leq F^{1}_2 \leq 1$. A naive integration and bounding then yields:
\begin{align*}
L_1 & = \frac{4}{1+r}\left(\int_{0}^{1} \frac{1 - 66 u^2w^2 + 495 u^4w^4 - 924 u^6w^6 + 495 u^8w^8 - 66 u^{10}w^{10} + u^{12}w^{12}}{\sqrt{1+w^2}} dw - \frac{13}{2} u^2 \tilde{G}^{1}_2\right) \\
& = \frac{4}{1+r}\left(\arcsinh(1) - 33 u^2  (\sqrt{2} - \arcsinh(1)) - u^4 \left(\frac{495}{8}\left(\sqrt{2}-3\arcsinh(1)\right)\right)\right. \\
& \left. - \frac{77}{4}u^{6}\left(13 \sqrt{2} - 15 \arcsinh(1)\right) -\frac{165}{128} u^8 (43 \sqrt{2} - 105 \arcsinh(1))  - \frac{33}{640} u^{10} (257 \sqrt{2} - 315 \arcsinh(1)) \right. \\
& \left.-\frac{7}{15360}u^{12}(221 \sqrt{2} - 495 \arcsinh(1)) -  \frac{13}{2}u^{2}\tilde{G}^{1}_2\right),
\end{align*}

where

\begin{align*}
0 \leq \tilde{G}^{1}_{2} \leq 1,
\end{align*}

and we have used that $|u|$ is small enough to guarantee the positiveness of the integrand.

Next, we proceed with $L_3$. We can write it as:

\begin{align*}
L_3 & = \frac{4}{1+r} \int_{0}^{1} \frac{1 - 66 w^2 + 495 w^4 - 924w^6 + 495 w^8 - 66 w^{10} + w^{12}}{(w^2+1)^{13/2}} \frac{dw}{\sqrt{1+w^2u^2}}.
\end{align*}

Expanding in series $(1+w^2u^2)^{-1/2}$:

\begin{align*}
(1+u^2 w^2)^{-1/2} = 1 - \frac12 u^2w^2 G^{2}_{2},
\end{align*}

where $\frac{1}{(1+u^2)^{3/2}} \leq G^{2}_2 \leq 1$. We use the fact that

\begin{align*}
\max_{w \in [0,1]} |1 - 66 w^2 + 495 w^4 - 924 w^6 + 495 w^8 - 66 w^{10} + w^{12}| =
|w(1)| = 64,
\end{align*}

to obtain, via integration and bounding:

\begin{align*}
L_3 & = \frac{4}{1+r}\left(\int_{0}^{1} \frac{1 - 66 w^2 + 495 w^4 - 924 w^6 + 495 w^8 - 66 w^{10} + w^{12}}{(w^2+1)^{13/2}}dw - \frac{1}{2} u^2 \tilde{G}^{2}_{2}\right) \\
& = \frac{4}{1+r}\left(\arcsinh(1)-\frac{2182}{3465}\sqrt{2} - \frac{1}{2} u^2 \tilde{G}^{2}_{2}\right),
\end{align*}

and

\begin{align*}
 |\tilde{G}^{2}_{2}| \leq 64.
\end{align*}

We finally move on to $L_2$. As before, we have the following formula obtained by integration by parts:

\begin{align*}
\int_{u^{2}}^{1} w^{-k-1}(1+w)^{-13/2}dw = \left.-\frac{1}{k}w^{-k}(1+w)^{-11/2}\right|_{w=u^{2}}^{w=1} - \left(\frac{2k+11}{2k}\right)\int_{u^{2}}^{1} w^{-k}(1+w)^{-11/2}dw.
\end{align*}

Defining now

\begin{align*}
H_{k} & = u^{2k} \int_{u^{2}}^{1} \frac{1}{w^{1+k}} \frac{dw}{(1+w)^{13/2}}
\end{align*}

we arrive to

\begin{align}
\label{formulaL}
H_{k} & = \frac{1}{k}\left(\frac{1}{(1+u^{2})^{11/2}} - \frac{u^{2k}}{\sqrt{2^{11}}}\right) - u^{2}\left(\frac{2k+11}{2k}\right)H_{k-1}.
\end{align}

We can write

\begin{align*}
 L_2 & = \frac{2}{1+r}\left(\sum_{k\geq 0}  {-1/2 \choose k} H_{k} - 66 u^{2} \sum_{k\geq 0}  {-1/2 \choose k} H_{k-1} + 495 u^{4}\sum_{k\geq 0}  {-1/2 \choose k} H_{k-2} - 924 u^{6} \sum_{k\geq 0}  {-1/2 \choose k} H_{k-3}\right) \\
 & + \frac{2}{1+r}\left(495 u^{8}\sum_{k\geq 0}  {-1/2 \choose k} H_{k-4} - 66 u^{10} \sum_{k\geq 0}  {-1/2 \choose k} H_{k-5} + u^{12}\sum_{k\geq 0}  {-1/2 \choose k} H_{k-6}\right) \\
& = L_{21} + L_{22} + L_{23} + L_{24} + L_{25} + L_{26} + L_{27}.
\end{align*}

The last six terms are easier. We deal with them first.

\begin{align*}
 L_{27} & = \frac{2}{1+r}\left(u^{12}H_{-6} + u^{12}\sum_{k\geq 1}  {-1/2 \choose k} H_{k-6}\right).
\end{align*}

The first term can be explicitly calculated and it amounts to

\begin{align*}
 u^{12}H_{-6} & = \frac{1}{22176}\frac{16384 + 90112 u^2 + 202752 u^4 + 236544 u^6 + 147840 u^8 + 44352 u^{10} - 11531 \sqrt{2} (1 + u^2)^{11/2}}{(1+u^{2})^{11/2}}.
\end{align*}

The last series is alternating and convergent and we can get the following bound:

\begin{align*}
 \left| u^{12}\sum_{k\geq 1}  {-1/2 \choose k} H_{k-6}\right|  \leq u^{12}\frac{1}{2}H_{-5}
 & = \frac{1}{221760(1+u^2)^{11/2}}u^{2} \\
& \times \left|8192 + 45056 u^2 + 101376 u^4 + 118272 u^6 + 73920 u^8 - 5419 \sqrt{2}(1 + u^2)^{11/2}\right|.
\end{align*}

We move on to $L_{26}$. By the same reasoning:

\begin{align*}
 L_{26} & = \frac{2}{1+r}\left(-66u^{10}H_{-5} - 66 u^{10}\sum_{k\geq 1}  {-1/2 \choose k} H_{k-5}\right) \\
& = \frac{2}{1+r}\left(-\frac{8192 + 45056 u^2 + 101376 u^4 + 118272 u^6 + 73920 u^8 - 5419 \sqrt{2} (1 + u^2)^{11/2}}{1680(1+u^2)^{11/2}}\right) \\
& + \frac{2}{1+r}\left(- 66 u^{10}\sum_{k\geq 1}  {-1/2 \choose k} H_{k-5}\right),
\end{align*}

\begin{align*}
 \left|66 u^{10}\sum_{k\geq 1}  {-1/2 \choose k} H_{k-5}\right| \leq 33u^{10}H_{-4}
= \frac{1}{1120}u^{2}\left|\frac{1024 + 5632 u^2 + 12672 u^4 + 14784 u^6 - 533 \sqrt{2} (1 + u^2)^{11/2}}{(1+u^2)^{11/2}}\right|.
\end{align*}

The same applies to $L_{25}$. We have:

\begin{align*}
 L_{25} & = \frac{2}{1+r}\left(495 u^{8} H_{-4} + 495 u^{8}\sum_{k\geq 1}  {-1/2 \choose k} H_{k-4}\right) \\
& = \frac{2}{1+r}\left(\frac{3(1024 + 5632 u^2 + 12672 u^4 + 14784 u^6 - 533 \sqrt{2} (1 + u^2)^{11/2})}{224(1+u^{2})^{11/2}}  + 495 u^{8}\sum_{k\geq 1}  {-1/2 \choose k} H_{k-4}\right) \\
\end{align*}

and

\begin{align*}
 &\left|495 u^{8}\sum_{k\geq 1}  {-1/2 \choose k} H_{k-4}\right|  \leq u^{8}\frac{495}{2}H_{-3} \\
& = \frac{5}{448}u^{2}\left|\frac{512 + 2816 u^2 + 6336 u^4 - 151 \sqrt{2} (1 + u^2)^{11/2}}{(1+u^2)^{11/2}}\right|.
\end{align*}

The next term is $L_{24}$. We have:

\begin{align*}
 L_{24} & = \frac{2}{1+r}\left(-924 u^{6} H_{-3} -924 u^{6}\sum_{k\geq 1}  {-1/2 \choose k} H_{k-3}\right) \\
& = \frac{2}{1+r}\left(\frac{-512 - 2816 u^2 - 6336 u^4 + 151 \sqrt{2} (1 + u^2)^{11/2})}{24(1+u^{2})^{11/2}}  - 924 u^{6}\sum_{k\geq 1}  {-1/2 \choose k} H_{k-3}\right) \\
\end{align*}

and

\begin{align*}
 &\left|924 u^{6}\sum_{k\geq 1}  {-1/2 \choose k} H_{k-3}\right|  \leq 462u^{6}H_{-2} = \frac{7}{48}u^{2}\left|-13 \sqrt{2} - \frac{576}{(1+u^2)^{11/2}} + \frac{704}{(1+u^2)^{9/2}}\right|.
\end{align*}

The term $L_{23}$ can be decomposed as:

\begin{align*}
 L_{23} & = \frac{2}{1+r}\left(495 u^{4} H_{-2} + 495 u^{4}\sum_{k\geq 1}  {-1/2 \choose k} H_{k-2}\right) \\
& = \frac{2}{1+r}\left(- \frac{65}{16 \sqrt{2}} - \frac{90}{(1+u^2)^{11/2}} + \frac{110}{(1+u^2)^{9/2}}  + 495 u^{4}\sum_{k\geq 1}  {-1/2 \choose k} H_{k-2}\right) \\
\end{align*}

and we have the bound

\begin{align*}
 &\left|495 u^{4}\sum_{k\geq 1}  {-1/2 \choose k} H_{k-2}\right|  \leq \frac{495}{2}u^{4}H_{-1} =
 \frac{45}{64} u^2 \left(-\sqrt{2} +\frac{64}{(1+u^2)^{11/2}}\right).
\end{align*}

We can split $L_{22}$ into:

\begin{align*}
 L_{22} & = \frac{2}{1+r}\left(-66 u^{2} H_{-1} - 66 u^{2}\sum_{k\geq 1}  {-1/2 \choose k} H_{k-1}\right) \\
& = \frac{2}{1+r}\left(\frac{3}{8 \sqrt{2}} - \frac{12}{(1+u^2)^{11/2}}  - 66 u^{2}\sum_{k\geq 1}  {-1/2 \choose k} H_{k-1}\right) \\
\end{align*}

and bound the second sum by

\begin{align*}
 &\left|66 u^{2}\sum_{k\geq 1}  {-1/2 \choose k} H_{k-1}\right|  \leq 33u^{2}H_{0} \\
 & = 33u^{2}\left(\frac{137969}{55440 \sqrt{2}} - \frac{2}{11(1+u^2)^{11/2}} - \frac{2}{9(1+u^2)^{9/2}}\right. \\
& \left.-\frac{2}{7(1+u^2)^{7/2}} - \frac{2}{5(1+u^2)^{5/2}} - \frac{2}{3(1+u^2)^{3/2}}\right. \\
& \left. - \frac{2}{\sqrt{1+u^{2}}} - 2 \arcsinh(1) + 2 \arcsinh\left(\frac{1}{|u|}\right)\right).
\end{align*}

In order to compute $L_{21}$ we use formula \eqref{formulaL}:

\begin{align*}
\sum_{k\geq 0}  {-1/2 \choose k} H_{k} & = H_{0} + \sum_{k\geq 1}  {-1/2 \choose k} H_{k} \\
& = H_{0} + \sum_{k\geq 1}  {-1/2 \choose k} \frac{1}{k}\left(\frac{1}{(1+u^{2})^{11/2}} - \frac{u^{2k}}{\sqrt{2^{11}}}\right) -  u^{2} \sum_{k\geq 1}  {-1/2 \choose k}\left(\frac{2k+11}{2k}\right)H_{k-1}\\
& = \frac{137969}{55440 \sqrt{2}} - \frac{2}{11(1+u^2)^{11/2}} - \frac{2}{9(1+u^2)^{9/2}} -\frac{2}{7(1+u^2)^{7/2}} - \frac{2}{5(1+u^2)^{5/2}} \\
& - \frac{2}{3(1+u^2)^{3/2}} - \frac{2}{\sqrt{1+u^{2}}} - 2 \arcsinh(1) + 2 \arcsinh\left(\frac{1}{|u|}\right) \\
& + \frac{1}{64}\left(\frac{128(\log(2) - \arcsinh(1))}{(1+u^2)^{5/2}} - \sqrt{2}\log(4) + 2\sqrt{2}\log(1+\sqrt{1+u^2})\right) \\
& -  u^{2} \sum_{k\geq 1}  {-1/2 \choose k}\left(\frac{2k+11}{2k}\right)H_{k-1}.
\end{align*}

The last sum can be bounded as usual by

\begin{align*}
& \left|u^{2} \sum_{k\geq 1}  {-1/2 \choose k}\left(\frac{2k+11}{2k}\right)H_{k-1}\right| \leq \frac{1}{2}\frac{11}{2}u^{2}H_{0} \\
& = \frac{11}{4}u^{2}\left|\frac{137969}{55440 \sqrt{2}} - \frac{2}{11(1+u^2)^{11/2}} - \frac{2}{9(1+u^2)^{9/2}} -\frac{2}{7(1+u^2)^{7/2}} - \frac{2}{5(1+u^2)^{5/2}}\right. \\
& \left.- \frac{2}{3(1+u^2)^{3/2}} - \frac{2}{\sqrt{1+u^{2}}} - 2 \arcsinh(1) + 2 \arcsinh\left(\frac{1}{|u|}\right)\right|.
\end{align*}

We finally add everything together to write $L$ as

\begin{align*}
L & = L_{ho,ns} + L_{ho,s} + L_{e,ns} + L_{e,s}
\end{align*}

where

\begin{align*}
L_{ho,ns}
& = \frac{4}{1+r}\left(\arcsinh(1) - 33 u^2  (\sqrt{2} - \arcsinh(1)) - u^4 \left(\frac{495}{8}\left(\sqrt{2}-3\arcsinh(1)\right)\right)\right) \\
& + \frac{4}{1+r}\left(- \frac{77}{4}u^{6}\left(13 \sqrt{2} - 15 \arcsinh(1)\right) -\frac{165}{128} u^8 (43 \sqrt{2} - 105 \arcsinh(1)) \right.\\ &\left. \qquad\qquad - \frac{33}{640} u^{10} (257 \sqrt{2} - 315 \arcsinh(1))\right)\\
& + \frac{4}{1+r}\left( -\frac{7}{15360}u^{12}(221 \sqrt{2} - 495 \arcsinh(1)) \right) \\
& + \frac{4}{1+r}\left(\arcsinh(1)-\frac{2182}{3465}\sqrt{2}\right) \\
& + \frac{2}{1+r}\left(u^{12}H(-6) - 66 u^{10}H(-5) + 495u^8H(-4) - 924u^6H(-3) +
  495u^4H(-2) - 66u^2H(-1)\right) \\
&  + \frac{2}{1+r}\left( H(0) - 2 \arctan\left(\frac{1}{|u|}\right)+\sum_{k\geq 1}  {-1/2 \choose k} \frac{1}{k}\left(\frac{1}{(1+u^{2})^{11/2}} - \frac{u^{2k}}{\sqrt{2^{11}}}\right)\right) \\
& = \frac{4}{1+r}\left(\arcsinh(1) + 33u^2(\arcsinh(1)-\sqrt{2}) - \frac{495}{8}u^4\left(\sqrt{2}-3\arcsinh(1)\right)\right. \\ &\left. \qquad\qquad - \frac{77}{4}u^6\left(13\sqrt{2} - 15 \arcsinh(1)\right)\right) \\
& + \frac{4}{1+r}\left(-\frac{165}{128}u^8\left(43\sqrt{2} - 105 \arcsinh(1)\right) - \frac{33}{640}u^{10}\left(257\sqrt{2} - 315 \arcsinh(1)\right)\right. \\ & \left.  \qquad\qquad - \frac{7 u^{12}}{15360}\left(221\sqrt{2} - 495\arcsinh(1)\right)\right) \\
& + \frac{4}{1+r}\left(-\frac{8}{3465(1+u^2)^{11/2}}(1627+11u^2(-604 + 3366 u^2 - 2268 u^4 + 945 u^6))
\right) \\
& + \frac{4}{1+r}\left(\frac{\log(2)-\arcsinh(1)}{(1+u^2)^{11/2}} + \frac{\log\left(\frac{1}{2}\left(1+\sqrt{1+u^2}\right)\right)}{32\sqrt{2}}\right) \\
L_{ho,s} & = \frac{4}{1+r} \arcsinh\left(\frac{1}{|u|}\right)
\end{align*}

and


\begin{align*}
|L_{e,ns}| & \leq \frac{4}{1+r}\left(\frac{13}{2}u^2 + 32 u^2\right)\\
& + \frac{4}{1+r}\left|\frac{u^{2}}{443520}\frac{8192 + 45056 u^2 + 101376 u^4 + 118272 u^6 + 73920 u^8 - 5419 \sqrt{2} (1 + u^2)^{11/2}}{(1+u^2)^{11/2}}\right| \\
& + \frac{4}{1+r}\left|\frac{u^{2}}{2240}\frac{1024 + 5632 u^2 + 12672 u^4 + 14784 u^6 - 533 \sqrt{2} (1 + u^2)^{11/2}}{(1+u^2)^{11/2}}\right| \\
& + \frac{4}{1+r}\left|\frac{5u^{2}}{896}\frac{512 + 2816 u^2 + 6336 u^4 - 151 \sqrt{2}(1 + u^2)^{11/2}}{(1+u^2)^{11/2}}\right| \\
& + \frac{4}{1+r}\left|\frac{7}{96} u^2 \left(-13 \sqrt{2} - \frac{576}{(1+u^2)^{11/2}} + \frac{704}{(1+u^2)^{9/2}}\right)\right| \\
& + \frac{4}{1+r}\left|\frac{45 u^{2}}{128}\left(-\sqrt{2} + \frac{64}{(1+u^2)^{11/2}}\right)\right| \\
& + \frac{4}{1+r}\left|\frac{143}{8}u^{2}\left(\frac{137969}{55440 \sqrt{2}} - \frac{2}{11(1+u^2)^{11/2}} - \frac{2}{9(1+u^2)^{9/2}} -\frac{2}{7(1+u^2)^{7/2}} - \frac{2}{5(1+u^2)^{5/2}}\right.\right. \\
& \left.\left.- \frac{2}{3(1+u^2)^{3/2}} - \frac{2}{\sqrt{1+u^{2}}} - 2 \arcsinh(1)\right)\right| \\
|L_{e,s}| & \leq \frac{4}{1+r}\frac{143}{4}u^{2} \arcsinh\left(\frac{1}{|u|}\right).
\end{align*}

\section{Implementation of the computer-assisted part and rigorous numerical results}
\label{sectioncomputerassisted}

In this section we will discuss the technical details about the implementation of the different integrals that appear in the proofs. We remark that we are computing explicit (but complicated) functions over a one dimensional domain. In order to perform the rigorous computations we used the C-XSC library \cite{CXSC}. The code can be found in the supplementary material.

The implementation is split into several files, and many of the headers of the functions (such as the integration methods) contain pointers to functions (the integrands) so that they can be reused for an arbitrary number of integrals with minimal changes and easy and safe debugging. For the sake of clarity, and at the cost of numerical performance and duplicity in the code, we decided to treat many simple integrals instead of a single big one.

We will only integrate in $\rho'$ (or more specifically, in $\trho'$: see below for the change of variables).  We outline the computation for $\tilde{I}(\rho)$ here but the other parts of the functions are calculated in the same way.

In order to minimize the impact of $a$ and $\beta$ being too small, we transform the original domain $[1-a,1]$ into a reference one: $[-1,1]$.

\begin{align*}
 \tilde{I}(\rho) & = -\frac{1}{2\pi \rho} \int_{1-a}^{1} f'(\rho') K^{1}\left(\frac{\rho}{\rho'}\right) d\rho' \\
& = -\frac{1}{2\pi \left(\frac{a}{2}(\trho-1)+1\right)} \int_{-1}^{1} \frac{a}{2}f_{\rho}\left(\frac{a}{2}(\trho'-1)+1\right)K^{1}\left(\frac{(\trho-1)+\frac{2}{a}}{(\trho'-1)+\frac{2}{a}}\right) d\trho' \\
& = -\frac{1}{2\pi \left(\frac{a}{2}(\trho-1)+1\right)} \int_{-1}^{1} \frac{a}{2}f_{\rho}\left(\frac{a}{2}(\trho'-1)+1\right)(I_{ho,s} + I_{ho,ns} + I_{e,s} + I_{e,ns})\left(\frac{(\trho-1)+\frac{2}{a}}{(\trho'-1)+\frac{2}{a}}\right) d\trho' \\
& \equiv \tilde{I}(\trho),
\end{align*}

where $\trho = \frac{2}{a}(\rho-1)+1$.

There are two basic classes in the programs that enclose all the necessary information used throughout the computations. The first one is called \texttt{ParameterSet} and has the following members: two doubles, \texttt{abs\_tol} and \texttt{rel\_tol}, providing the desired tolerances used to accept or reject the enclosure of the integral in the adaptive integration scheme described below; two intervals, \texttt{a} and \texttt{beta}, which are parameters of the system. We take them to be intervals since the actual value of $a$ is not representable by a computer. A \texttt{ParameterSet} also contains two intervals, \texttt{Left} and \texttt{Right}, describing the boundaries of the integration region; two integers, \texttt{region\_rho} and \texttt{region\_rhop}, denoting whether $\trho$ and $\trho'$ are in $[-1,-1+\beta]$,$[-1+\beta,1-\beta]$ or $[1-\beta,1]$ respectively. Finally, there is also an interval \texttt{rho\_normalized} indicating the value of $\trho$ (we remark that we are integrating in $\trho'$). The second data structure is called \texttt{IntegrationResult} and is composed of a \texttt{ParameterSet}, an interval \texttt{result} containing the result of the integration, an ivector (vector of intervals) \texttt{error\_by\_coordinate} which has information about the error in the different directions and an integer \texttt{flag} which is set to 1 if we ever encounter an error in the program (e.g. a division by zero due to overestimation).

We now explain how the integrals are calculated. By technical issues explained below, we will split the integration region $[-1,1]$ into smaller pieces and sum the contributions over each piece. Regardless of the domain, the integration is done in an adaptive way unless specified. We keep track of the regions over which we need to integrate in a Standard Template Library \texttt{priority\_queue}, which keeps the \texttt{IntegrationResults} sorted by absolute width of their member \texttt{result}. We operate by taking the topmost element (i.e. the one with the highest absolute width) and deciding to accept the result or reject it. This is done based on the width of the result in an absolute and a relative (to the length of the integration region) way (it has to be smaller than \texttt{abs\_tol} and \texttt{rel\_tol} respectively). In the latter case, we split the region and recomputed the integral on both subregions. The splitting is done by the midpoint. In order to avoid infinite loops -- which could potentially happen since there is uncertainty in the value of $\trho$ --, we repeat this step at most \texttt{MAX\_ELEMENTS\_EVALUATED} times. In our code, \texttt{MAX\_ELEMENTS\_EVALUATED} $= 100.000$. All integrations are done using a Gauss-Legendre quadrature of order 2, given by:

\begin{align*}
\int_{a}^{b} f(x) dx \in  \frac{b-a}{2}\left(f\left(\frac{b-a}{2}\frac{\sqrt{3}}{3} + \frac{b+a}{2}\right)+f\left(-\frac{b-a}{2}\frac{\sqrt{3}}{3} + \frac{b+a}{2}\right)\right)
+\frac{1}{4320}(b-a)^{5}f^{4}([a,b]).
\end{align*}

Once we have defined the basic classes and explained the integration method, we now turn into the discussion of the splitting of the interval $[-1,1]$. We will compute 4 different integrals depending on whether we are integrating $I_{ho,ns}$,$I_{ho,s}$,$I_{e,ns}$,$I_{e,s}$ (see Appendix \ref{sectionasymptotics}). On the one hand, the integrals of $I_{ho,ns}$ and $I_{e,ns}$ will be performed on the full interval $[-1,1]$ taking care of adjusting the regions of $\trho$ and $\trho'$ accordingly (see Figure \ref{regionsF} for a depiction of the different regions) in order to adjust the expression of $f_{\rho}$ accordingly to the region. We remark that because of the monotonicity of $f_{\rho}$, whenever we want to evaluate $f_{\rho}$ in an interval it is enough to compute it at the endpoints and take the hull of the two results. On the other hand, the integrals of $I_{ho,s}$ and $I_{e,s}$ will be split into a staircase domain and a singularity region depending on $(\rho,\rho')$. The staircase domain is shown in Figure \ref{staircase} for $N = 20, \beta = \frac{2}{N}$.

\begin{figure}[ht]
\centering
\includegraphics[width=0.6\textwidth]{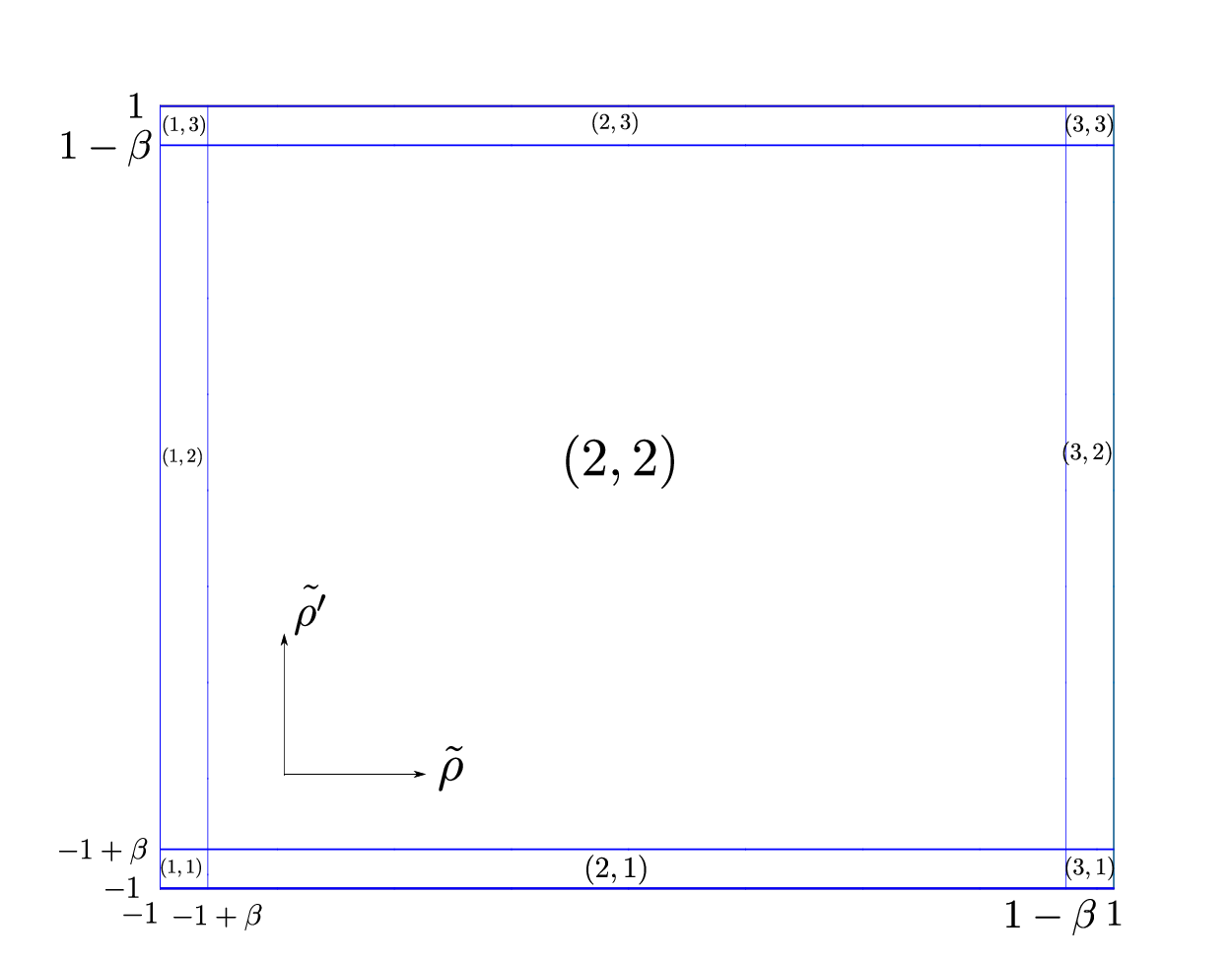}
\caption{The 9 different possibilities in $(\rho,\rho')$ leading to different values of $f'$.}
\label{regionsF}
\end{figure}

\begin{figure}[ht]
\centering
\includegraphics[width=0.6\textwidth]{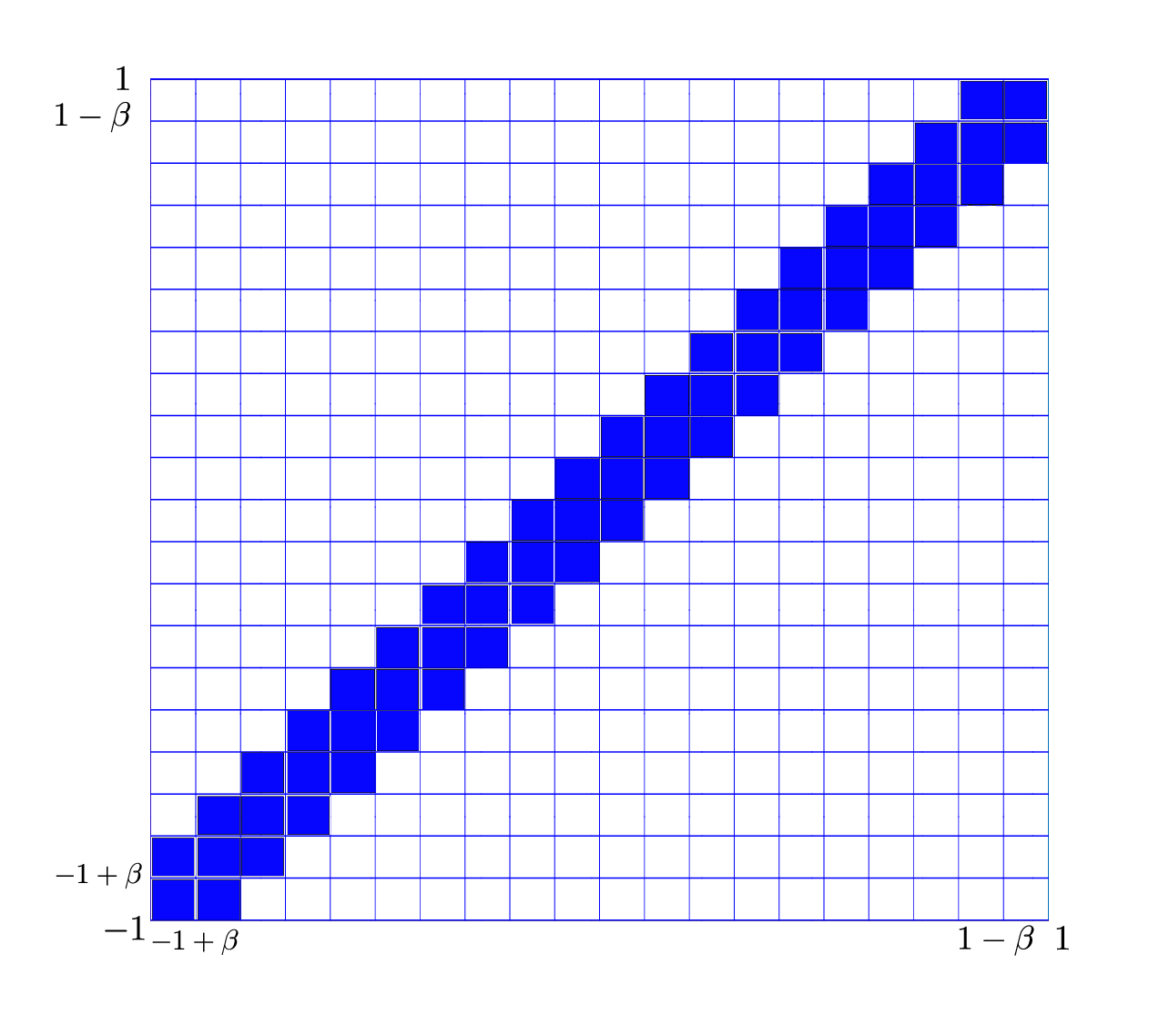}
\caption{Different integration regions. Colored: singular region, white: staircase region.}
\label{staircase}
\end{figure}

In order to integrate over the singularity region, we will integrate by factoring out everything out of the factor $\arcsinh(\frac{1}{|u|})$ and integrating explicitly.

For example, if we want to integrate

\begin{align}\label{integralarcsinh}
 \int A(\rho,\rho') \arcsinh\left(\frac{1}{|u|}\right) d\rho'
\end{align}

and we have uniform bounds on $A$ of the type

\begin{align*}
 \underline{a} \leq A(\rho,\rho') \leq \overline{a},
\end{align*}

then the integral \eqref{integralarcsinh} can be bounded by

\begin{align*}
 \underline{a}\int \arcsinh\left(\frac{1}{|u|}\right) d\rho' \leq \int A(\rho,\rho') \arcsinh\left(\frac{1}{|u|}\right) d\rho' \leq \overline{a}\int \arcsinh\left(\frac{1}{|u|}\right) d\rho',
\end{align*}

yielding the enclosure

\begin{align*}
 \int A(\rho,\rho') \arcsinh\left(\frac{1}{|u|}\right) d\rho' \in \left[\underline{a} \int \arcsinh\left(\frac{1}{|u|}\right)d\rho' ,\overline{a} \int \arcsinh\left(\frac{1}{|u|}\right)d\rho'\right].
\end{align*}

A strightforward but long calculation yields the following lemma, which is useful for that purpose:

\begin{lemma}
 Let $-1 \leq \tilde{\rho},\tilde{\rho}' \leq 1$, $0 < a < 1$. We have that:
\begin{align*}
 & \int_{c}^{d} \arcsinh\left(\frac{|\tilde{\rho} + \tilde{\rho}' - 2 + \frac{4}{a}|}{|\tilde{\rho} - \tilde{\rho}'|}\right) d\tilde{\rho}' \\
& = -c \arcsinh\left(\frac{4+a(-2+c+\tilde{\rho})}{a|c-\trho|}\right) + d \arcsinh\left(\frac{4+a(-2+d+\tilde{\rho})}{a|d-\trho|}\right) + \trho\log\left(\frac{|d-\trho|}{|c-\trho|}\right) \\
& + \frac{\sqrt{2}}{a}(-2+a-a\trho)\log\left(\frac{2+a(-1+c)+\sqrt{8+4a(-2+c+\trho)+a^2(2+(-2+c)c+(-2+\trho)\trho)}}{2+a(-1+d)+\sqrt{8+4a(-2+d+\trho)+a^2(2+(-2+d)d+(-2+\trho)\trho)}}\right) \\
& + \trho \log\left(\frac{4+a(-2+c+\trho)+\sqrt{2}\sqrt{8+4a(-2+c+\trho)+a^2(2+(-2+c)c+(-2+\trho)\trho)}}{4+a(-2+d+\trho)+\sqrt{2}\sqrt{8+4a(-2+d+\trho)+a^2(2+(-2+d)d+(-2+\trho)\trho)}}\right) \\
& = -c \log\left(\frac{4}{a}+(-2+c+\tilde{\rho}) + \sqrt{(c-\trho)^2 + \left(\frac{4}{a}+(-2+c+\tilde{\rho})\right)^2}\right) + (c-\trho)\log\left(|c-\trho|\right) \\
& +d \log\left(\frac{4}{a}+(-2+d+\tilde{\rho}) + \sqrt{(d-\trho)^2 + \left(\frac{4}{a}+(-2+d+\tilde{\rho})\right)^2}\right) - (d-\trho)\log\left(|d-\trho|\right) \\
& + \frac{\sqrt{2}}{a}(-2+a-a\trho)\log\left(\frac{2+a(-1+c)+\sqrt{8+4a(-2+c+\trho)+a^2(2+(-2+c)c+(-2+\trho)\trho)}}{2+a(-1+d)+\sqrt{8+4a(-2+d+\trho)+a^2(2+(-2+d)d+(-2+\trho)\trho)}}\right) \\
& + \trho \log\left(\frac{4+a(-2+c+\trho)+\sqrt{2}\sqrt{8+4a(-2+c+\trho)+a^2(2+(-2+c)c+(-2+\trho)\trho)}}{4+a(-2+d+\trho)+\sqrt{2}\sqrt{8+4a(-2+d+\trho)+a^2(2+(-2+d)d+(-2+\trho)\trho)}}\right). \\
\end{align*}

\end{lemma}

The computations were run on a desktop with 8 cores at 3.10 GHz and 8 GB of RAM. The different runtimes are described in each of the lemmas separately.

We will explain the algorithms and the procedures of the different lemmas from easier to harder, irrespective of the order in which the lemmas are found in the main sections of the paper.

\begin{proofsth}{Lemma \ref{lemmaImayorquelambda}}

We compute the hull (an enclosure of the range) of $\tilde{I}\left(\frac{2}{a}(\trho-1)+1\right)$ for $\trho \in [-1,1]$. To do so, we split the interval $[-1,1]$ into $N = 512$ intervals $I_{j}, j=0\ldots N-1$ of equal size and compute an enclosure $II_{j} = \tilde{I}\left(\frac{2}{a}(\trho-1)+1\right)$, for $\trho = I_{j}$. We can prove the following estimate:

\begin{align*}
 \min_{\tilde{\rho} \in [-1,1]} \tilde{I}\left(\frac{2}{a}(\trho-1)+1\right) \geq \min_{j} II_{j} > 1.2655
\end{align*}


The minimum is attained in the last region, $\trho \in \left[1-\frac{2}{N},1\right]$. A nonrigorous computation shows that indeed this minimum is attained at $\trho = 1$. The detailed breakdown of the regions can be found in the file \texttt{output/output\_Min\_I\_512.out} in the supplementary material. The tolerances \texttt{abs\_tol} and \texttt{rel\_tol} were set to $10^{-5}$. The computation took approximately 242 minutes, giving an average time of around 7 seconds per integral.

\end{proofsth}

\begin{proofsth}{Lemma \ref{lemacotae}}

For every $\trho \in [-1,1]$, we complete an enclosure of $e\left(\frac{2}{a}(\trho-1)+1\right)$. To do so, we split the interval $[-1,1]$ into $N = 512$ intervals $I_{j}, j=0\ldots N-1$ of equal size and compute an enclosure $\mathcal{E}^{3}_{j} = e\left(\frac{2}{a}(\trho-1)+1\right)$, for $\trho = I_{j}$. Finally, we can estimate an enclosure of the $L^{2}$ norm of $e$ and the scalar product with $B_{sj}$ by

\begin{align*}
 \|e\|_{L^{2}} \in \left(\frac{a}{N}\sum_{j=0}^{N-1}(\mathcal{E}^{3}_{j})^{2}\right)^{\frac12}, \\
|\langle e, B_{sj} \rangle| \in \left|\frac{a}{N}\sum_{j=0}^{N-1}(\mathcal{E}^{3}_{j}B_{sj}(I_{j}))\right| = \left|\frac{a}{N}\sum_{j=1}^{N-2}(\mathcal{E}^{3}_{j})\right|,
\end{align*}

where we have used that $B_{sj}$ is piecewise constant in each of the $I_{j}$: it is zero if $j = 0$ or $j = N-1$ and 1 otherwise. This yields in this particular case

\begin{align*}
\|e\|_{L^{2}} & < 0.0905 \\
|\langle e,B_{sj} \rangle | & < 0.0101
\end{align*}

The detailed breakdown of the regions can be found in the file \texttt{output/output\_E\_3.out} in the supplementary material. The tolerances \texttt{AbsTol} and \texttt{RelTol} were set to $10^{-5}$. The computation took approximately 563 minutes, giving an average time of around 16.5 seconds per integral.

\end{proofsth}

\begin{proofsth}{Lemma \ref{lemacotasABC}}

For every $\trho \in [-1,1]$, we complete an enclosure of $\Theta^{3}_{A}B_{sj}\left(\frac{2}{a}(\trho-1)+1\right)$. To do so, we split the interval $[-1,1]$ into $N = 512$ intervals $I_{j}, j=0\ldots N-1$ of equal size and compute an enclosure $T^{3}_{A,j} = \Theta^{3}_{A}B_{sj}\left(\frac{2}{a}(\trho-1)+1\right)$, for $\trho = I_{j}$. Finally, we can estimate an enclosure of the $L^{2}$ norm of $\Theta^{3}_{A}B_{sj}$ by

\begin{align*}
 \|\Theta^{3}_{A}B_{sj}\|_{L^{2}} \in \left(\frac{a}{N}\sum_{j=0}^{N-1}(T^{3}_{A,j})^{2}\right)^{\frac12},
\end{align*}

yielding in this particular case

\begin{align*}
\|\Theta^{3}_{A}B_{sj}\|_{L^{2}} < 0.0629
\end{align*}


The detailed breakdown of the regions can be found in the file \texttt{output/output\_Theta\_A\_N\_512.out} in the supplementary material. The tolerances \texttt{abs\_tol} and \texttt{rel\_tol} were set to $10^{-5}$. The computation took approximately 405 minutes, giving an average time of around 12 seconds per integral.

\end{proofsth}

\begin{proofsth}{Lemma \ref{lemacotascestrella}}

To obtain a bound on $c^{*}$ we will employ the following strategy. First, we can bound $c^{*}$ in the following way:

\begin{align*}
 c^{*} \geq \min \tilde{I}(\rho) + \min_{\|u\|_{L^{2}} = 1, u \in (B_{sj})^{\perp}} \langle T^{3}_{S} u, u\rangle,
\end{align*}

where $T_{S}^{3}$ is the symmetric part of $\Theta^{3} - \tilde{I}$, given by

\begin{align*}
 T_{S}^{3} u(\rho) & = \frac12 \int_{1-a}^{1}\left(f_\rho(\rho')K^3\left(\frac{\rho}{\rho'}\right)\frac{1}{\rho} +
f_\rho(\rho)K^3\left(\frac{\rho'}{\rho}\right)\frac{1}{\rho'}\right)B(\rho') d\rho'
\end{align*}

We know that $T_{S}^{3}$ is symmetric and compact. We approximate it by a finite rank operator. For any $\trho \in [-1,1]$:

\begin{align*}
 (T_{S}^{3}u)(\frac{a}{2}(\trho-1)+1) = T_{S}^{3}u(\trho) \approx T_{fin}^{3}u(\trho) = \int_{-1}^{1}\frac{a}{2}\sum_{i,j=0}^{23} (T_{fin}^{3})_{ij}u_i(\trho)u_j(\trho') u(\frac{a}{2}(\trho'-1)+1) d\trho',
\end{align*}

The matrix $T_{fin}^{3}$ is symmetric and given explicitly in Appendix \ref{appendixprojections}. The functions $u_i$ are an orthonormal basis chosen in the following way:

\begin{align*}
 u_i(\trho) & =
\left\{
\begin{array}{cc}
  \sqrt{\frac{2(2[i/3]+1)}{a\beta}}\text{Leg}([i/3],\frac{2}{\beta}(\trho+1)-1)1_{\{-1 \leq \trho \leq -1+\beta\}}& \text{ if } i \equiv 0 (mod \, 3) \\
  \sqrt{\frac{2[i/3]+1}{a-a\beta}}\text{Leg}([i/3],\frac{\trho}{1-\beta})1_{\{-1+\beta \leq \trho \leq 1-\beta\}} & \text{ if } i \equiv 1 (mod \, 3) \\
  \sqrt{\frac{2(2[i/3]+1)}{a\beta}}\text{Leg}([i/3],\frac{2}{\beta}(\trho-1)+1)1_{\{1-\beta \leq \trho \leq 1\}} & \text{ if } i \equiv 2 (mod \, 3) \\
\end{array}
\right.
\end{align*}

where Leg($a,x$) stands for the standard Legendre polynomial of order $a$, defined for $x \in [-1,1]$ by

\begin{align*}
 \text{Leg}(0,x) = 1, & \quad \text{Leg}(1,x) = x, \\
(n+1)\text{Leg}(n+1,x) & = (2n+1)x\text{Leg}(n,x) - n\text{Leg}(n-1,x), \quad n \geq 1.
\end{align*}

Note that $B_{sj}$ corresponds to $u_1$. We now decompose $T_{S}^{3}$ as:

\begin{align*}
 \min_{\|u\|_{L^{2}} = 1, u \in (B_{sj})^{\perp}} \langle T^{3}_{S} u, u\rangle
\geq \min_{\|u\|_{L^{2}} = 1, u \in (B_{sj})^{\perp}} \langle T^{3}_{fin} u, u\rangle
+ \min_{\|u\|_{L^{2}} = 1, u \in (B_{sj})^{\perp}} \langle (T^{3}_{S} - T^{3}_{fin})u, u\rangle
\end{align*}

The first term in the sum is simply the smallest eigenvalue of the matrix $T_{fin}^{3}$ without the second row and the second column. By Gershgorin's theorem  \cite{Gerschgorin:eigenvalues-theorem}, the eigenvalues of an $n \times n$ matrix $A$ lie inside the union of the disks

\begin{align*}
 D_{i} = \left\{z \in \mathbb{C}, |z-A_{ii}| \leq \sum_{\substack{j=1 \\ j\neq i}} |A_{ij}|\right\}, \quad i =1, \ldots, n.
\end{align*}

In this particular case, this implies that

\begin{align*}
 \min_{\|u\|_{L^{2}} = 1, u \in (B_{sj})^{\perp}} \langle T^{3}_{fin} u, u\rangle > -0.3125,
\end{align*}

where the leftmost disk of $T_{fin}^{3}$ is $D_{5}$. The second term can be bounded via the operator norm

\begin{align*}
 \left|\min_{\|u\|_{L^{2}} = 1, u \in (B_{sj})^{\perp}} \langle (T^{3}_{S} - T^{3}_{fin})u, u\rangle\right|
\leq \|T^{3}_{S} - T^{3}_{fin}\|_{L^{2} \rightarrow L^{2}} \\
\leq
\max_{\trho \in [-1,1]} \int_{-1}^{1} \left|\frac{a}{2}\frac12 \left(f_\rho\left(\frac{a}{2}(\trho'-1)+1\right)K^3\left(\frac{\frac{a}{2}(\trho-1)+1}{\frac{a}{2}(\trho'-1)+1}\right)\frac{1}{(\frac{a}{2}(\trho-1)+1)} \right.\right. \\
\left.\left.+
f_\rho\left(\frac{a}{2}(\trho-1)+1\right)K^3\left(\frac{\frac{a}{2}(\trho'-1)+1}{\frac{a}{2}(\trho-1)+1}\right)\frac{1}{(\frac{a}{2}(\trho'-1)+1)}\right)\right. \\
\left. - \frac{a}{2}\sum_{i,j=0}^{23} (T_{fin}^{3})_{ij}u_i(\trho)u_j(\trho')\right|
d\trho'
\end{align*}
where we have used that the operator is symmetric and the Generalized Young inequality. We computed the following bound:

\begin{align*}
  \|T^{3}_{S} - T^{3}_{fin}\|_{L^{2}} \leq 0.1004
\end{align*}

To do so, we split the subdomains $[-1,-1+\beta]$, $[-1+\beta,1-\beta]$ and $[1-\beta,1]$ for $\trho'$ into a uniform mesh of $N_1 = 512, N_2 = 510 \cdot 16, N_3 = 512$ subintervals respectively. Since we are expecting bounds of the order of the width of the integration integral and we can't do better (the integrand is not $C^{1}$), we simply compute the integrand evaluated in the full interval (a quadrature of order 0). We are careful and the singular part of the integrand is bounded first (all terms that multiply the arcsinh term) and then integrated explicitly separately. This is done whenever $\trho \in I_j, \trho' \in I_k$ and $|j-k| < 2$. Otherwise we evaluate the full integrand. For every $\trho \in I_i$ we sum over all $j$ such that $\trho' \in I_j$ to obtain the $L^1$ bound. Finally, we take the maximum over every $\trho$.

Putting everything together we obtain

\begin{align*}
 c^{*} > 1.2655 - 0.3125 - 0.1004 = 0.8526
\end{align*}

The detailed breakdown of the regions can be found in the file \texttt{output/out\_L1\_Estimates\_T3\_N\_512.out} in the supplementary material. The computation took approximately 12 hours, 53 minutes, giving an average time of around 5 seconds per subinterval in $\trho$.

\end{proofsth}

\begin{proofsth}{Lemma \ref{lemacotae6}}

The proof follows the same strategy as the proof of Lemma \ref{lemacotae}. We get the following:

\begin{align*}
||e^6||_{L^2} < 0.0893.
\end{align*}

The detailed breakdown of the regions can be found in the file \texttt{output/output\_E\_6.out} in the supplementary material. The tolerances \texttt{abs\_tol} and \texttt{rel\_tol} were set to $10^{-5}$. The computation took approximately 604 minutes, giving an average time of around 17.5 seconds per integral.

\end{proofsth}

\begin{proofsth}{Lemma \ref{lemacotascestrella6}}

The proof follows the same strategy as the proof of Lemma \ref{lemacotascestrella}. We get the following figures:

\begin{align*}
 \min_{\|u\|_{L^{2}} = 1, u \in (B_{s}^{6\text{aprox}})^{\perp}} \langle T^{6}_{fin} u, u\rangle > -0.3121.
\end{align*}

\begin{align*}
  \|T^{6}_{S} - T^{6}_{fin}\|_{L^{2}} \leq 0.1179
\end{align*}

\begin{align*}
 c^{6*} > 1.2655 - 0.3121 - 0.1179 = 0.8355
\end{align*}

As before, the leftmost disk is $D_{5}$. The detailed breakdown of the regions can be found in the file \texttt{output/out\_L1\_Estimates\_T6\_N\_512.out} in the supplementary material. The computation took approximately 13 hours, 3 minutes, giving an average time of around 5 seconds per subinterval in $\trho$.

\end{proofsth}

In summary, we have proved the following bounds (we are using that $\|B_{sj}\|_{L^{2}} = \|B_{s}^{6\text{aprox}}\|_{L^{2}} = 1$):

\begin{itemize}
 \item $c^{*} > 0.8526$
\item $\mathcal{A} = \lambda^{*} + |\langle e, B_{sj} \rangle| < 0.3583$
\item $\sqrt{\mathcal{B}} \leq \|e\|_{L^{2}} + \|\Theta_{A}^{3} B_{sj}\|_{L^{2}} < 0.1534 $
\item $\lambda_{3} \leq \lambda_{0} < 0.4117 $
\item $\tilde{I} - \lambda_{3} > 0.8538$
 \item $c^{6*} > 0.8355$
\end{itemize}

\section{Finite projections}
\label{appendixprojections}

In this section we provide the two matrices of size $24 \times 24$: $T^{3}_{fin}$ and $T^{6}_{fin}$ used to approximate the finite dimensional projections of $T^{3}$ and $T^{6}$ respectively. The matrices were computed using nonrigorous integration. The matrices (in a slight different format) can be found in the files \texttt{input/good\_projection\_N\_512.out} and \texttt{input/good\_projection\_T6\_N\_512.out}. In order to write the matrices, because of spacing issues, we will decompose $T^{3}_{fin}$ and $T^{6}_{fin}$ into the following blocks:

\begin{align*}
 T^{3}_{fin} & =
\left(
\begin{array}{cccc}
 T^{3}_{fin,1} & T^{3}_{fin,2} & T^{3}_{fin,3} & T^{3}_{fin,4}
\end{array}
\right) \\
 T^{6}_{fin} & =
\left(
\begin{array}{cccc}
 T^{6}_{fin,1} & T^{6}_{fin,2} & T^{6}_{fin,3} & T^{6}_{fin,4}
\end{array}
\right),
\end{align*}

where every block is $24 \times 6$. The exact expressions are:
\tiny
\begin{align*}
 T^{3}_{fin,1} & = \left(
\begin{array}{cccccc}
-0.00313914631  & -0.03306228140  & -0.00063757204  & -0.00130670573  & 0.00933189622  & 0.00027670636 \\
-0.03306228140  & -1.12901165788  & -0.03145756256  & -0.00822746089  & 0.00729039148  & 0.00876971358 \\
-0.00063757204  & -0.03145756256  & -0.00299812421  & -0.00022549348  & -0.00911506433  & 0.00124776818 \\
-0.00130670573  & -0.00822746089  & -0.00022549348  & -0.00021202888  & 0.00270210646  & 0.00000013750 \\
0.00933189622  & 0.00729039148  & -0.00911506433  & 0.00270210646  & -0.23717252091  & 0.00223673972 \\
0.00027670636  & 0.00876971358  & 0.00124776818  & 0.00000013750  & 0.00223673972  & -0.00020166430 \\
0.00006084465  & -0.00000985095  & -5.18229665685\cdot 10^{-8}  & -0.00005068073  & 0.00001298234  & 4.63109155689\cdot 10^{-11} \\
-0.00402764479  & 0.06018656936  & -0.00384397048  & -0.00114915910  & 0.00170736716  & 0.00102743751 \\
-4.43314493090\cdot 10^{-8}  & -0.00000908783  & 0.00005786424  & -4.60911625268\cdot 10^{-11}  & -0.00001219458  & 0.00004810665 \\
0.00033343322  & 0.00192283641  & 0.00005295158  & 0.00004143653  & -0.00062130092  & -1.82960948332\cdot 10^{-8} \\
0.00233871666  & -0.00072531006  & -0.00222568341  & 0.00067663595  & 0.04134607065  & 0.00062743044 \\
-0.00006498459  & -0.00205039765  & -0.00031825067  & -1.39917284412\cdot 10^{-8}  & -0.00051248314  & 0.00003940755 \\
0.00000533557  & 0.00000206570  & 1.15781378363\cdot 10^{-8}  & 0.00003985325  & -0.00000267070  & -6.84832886920\cdot 10^{-12} \\
-0.00155761794  & 0.00542593789  & -0.00148129715  & -0.00046479429  & -0.00053855113  & 0.00043588867 \\
9.90200697265\cdot 10^{-9}  & 0.00000190393  & 0.00000505660  & 5.62487168223\cdot 10^{-12}  & 0.00000250964  & -0.00003785447 \\
-0.00007592114  & -0.00048164344  & -0.00001327566  & 0.00000498929  & 0.00015501214  & 4.58706059629\cdot 10^{-9} \\
0.00112127249  & -0.00008203811  & -0.00106597793  & 0.00034694746  & 0.00518713553  & 0.00032698003 \\
0.00001629250  & 0.00051364377  & 0.00007249718  & 3.50789736574\cdot 10^{-9}  & 0.00012776307  & 0.00000473106 \\
0.00000125611  & 8.65982848231\cdot 10^{-8}  & -3.19226015591\cdot 10^{-9}  & -0.00000437207  & -0.00000039251  & 1.95786950003\cdot 10^{-12} \\
-0.00084878958  & 0.00138998159  & -0.00080678323  & -0.00027291240  & -0.00008220557  & 0.00025788023 \\
-2.73010821736\cdot 10^{-9}  & 1.02554282905\cdot 10^{-7}  & 0.00000119690  & -1.62202100706\cdot 10^{-12}  & 0.00000038596  & 0.00000415121 \\
0.00001432963  & 0.00008343141  & 0.00000227980  & 0.00000129772  & -0.00002784648  & -7.87753818442\cdot 10^{-10} \\
0.00066557241  & -0.00002363432  & -0.00063256170  & 0.00022255360  & 0.00154868510  & 0.00021062293 \\
-0.00000279787  & -0.00008889372  & -0.00001367678  & -6.02460704593\cdot 10^{-10}  & -0.00002311176  & 0.00000123395 \\
\end{array}
\right)
\end{align*}
\begin{align*}
 T^{3}_{fin,2} & = \left(
\begin{array}{cccccc}
 0.00006084465  & -0.00402764479  & -4.43314493090\cdot 10^{-8}  & 0.00033343322  & 0.00233871666  & -0.00006498459 \\
 -0.00000985095  & 0.06018656936  & -0.00000908783  & 0.00192283641  & -0.00072531006  & -0.00205039765 \\
 -5.18229665685\cdot 10^{-8}  & -0.00384397048  & 0.00005786424  & 0.00005295158  & -0.00222568341  & -0.00031825067 \\
 -0.00005068073  & -0.00114915910  & -4.60911625268\cdot 10^{-11}  & 0.00004143653  & 0.00067663595  & -1.39917284412\cdot 10^{-8} \\
 0.00001298234  & 0.00170736716  & -0.00001219458  & -0.00062130092  & 0.04134607065  & -0.00051248314 \\
 4.63109155689\cdot 10^{-11}  & 0.00102743751  & 0.00004810665  & -1.82960948332\cdot 10^{-8}  & 0.00062743044  & 0.00003940755 \\
 -0.00010321148  & -0.00001414574  & -2.80290333759\cdot 10^{-14}  & -0.00003848684  & 0.00001470528  & -3.67487307816\cdot 10^{-12} \\
 -0.00001414574  & -0.12947879697  & -0.00001337930  & 0.00025555410  & 0.00095216359  & -0.00022761700 \\
 -2.80290333759\cdot 10^{-14}  & -0.00001337930  & -0.00009797657  & 4.90797112113\cdot 10^{-12}  & -0.00001393579  & 0.00003651136 \\
 -0.00003848684  & 0.00025555410  & 4.90797112113\cdot 10^{-12}  & -0.00006239992  & -0.00014435698  & -3.49966215326\cdot 10^{-15} \\
 0.00001470528  & 0.00095216359  & -0.00001393579  & -0.00014435698  & -0.08886148634  & -0.00013349522 \\
 -3.67487307816\cdot 10^{-12}  & -0.00022761700  & 0.00003651136  & -3.49966215326\cdot 10^{-15}  & -0.00013349522  & -0.00005919817 \\
 0.00003020315  & 0.00000287744  & 2.98104624843\cdot 10^{-15}  & -0.00001908529  & -0.00000297094  & 1.50361388430\cdot 10^{-14} \\
 -0.00001495138  & 0.03027233689  & -0.00001418074  & 0.00009471142  & 0.00066179035  & -0.00008862582 \\
 3.14813570358\cdot 10^{-15}  & 0.00000272492  & 0.00002869680  & -1.48184820399\cdot 10^{-14}  & 0.00000282037  & 0.00001807875 \\
 0.00002961061  & -0.00006309436  & -1.21506679426\cdot 10^{-12}  & 0.00002355898  & 0.00003503703  & -7.01151105584\cdot 10^{-16} \\
 0.00001500747  & -0.00041014040  & -0.00001423998  & -0.00006730898  & 0.02375644464  & -0.00006331650 \\
 9.04914568531\cdot 10^{-13}  & 0.00005613296  & -0.00002812073  & -6.97673484952\cdot 10^{-15}  & 0.00003236439  & 0.00002238336 \\
 0.00000406644  & 0.00000066295  & -5.79173351464\cdot 10^{-14}  & 0.00002138146  & -0.00000090330  & -5.68804673351\cdot 10^{-14} \\
 -0.00001493806  & 0.00436069767  & -0.00001417756  & 0.00005029006  & -0.00032965875  & -0.00004744310 \\
 -7.69293477455\cdot 10^{-14}  & 0.00000063871  & 0.00000386310  & 5.43227769034\cdot 10^{-14}  & 0.00000086668  & -0.00002030190 \\
 -0.00000252131  & 0.00001241650  & 7.34844588938\cdot 10^{-14}  & 0.00000330787  & -0.00000788818  & 8.17156504431\cdot 10^{-14} \\
 0.00001478127  & -0.00007108881  & -0.00001403083  & -0.00003888592  & 0.00368849025  & -0.00003674960 \\
 -4.37312250174\cdot 10^{-14}  & -0.00001114858  & 0.00000239297  & 1.14001332420\cdot 10^{-13}  & -0.00000734374  & 0.00000314218 \\
\end{array}
\right)
\end{align*}
\begin{align*}
 T^{3}_{fin,3} & =
\left(
\begin{array}{cccccc}
 0.00000533557  & -0.00155761794  & 9.90200697265\cdot 10^{-9}  & -0.00007592114  & 0.00112127249  & 0.00001629250 \\
 0.00000206570  & 0.00542593789  & 0.00000190393  & -0.00048164344  & -0.00008203811  & 0.00051364377 \\
 1.15781378363\cdot 10^{-8}  & -0.00148129715  & 0.00000505660  & -0.00001327566  & -0.00106597793  & 0.00007249718 \\
 0.00003985325  & -0.00046479429  & 5.62487168223\cdot 10^{-12}  & 0.00000498929  & 0.00034694746  & 3.50789736574\cdot 10^{-9} \\
 -0.00000267070  & -0.00053855113  & 0.00000250964  & 0.00015501214  & 0.00518713553  & 0.00012776307 \\
 -6.84832886920\cdot 10^{-12}  & 0.00043588867  & -0.00003785447  & 4.58706059629\cdot 10^{-9}  & 0.00032698003  & 0.00000473106 \\
 0.00003020315  & -0.00001495138  & 3.14813570358\cdot 10^{-15}  & 0.00002961061  & 0.00001500747  & 9.04914568531\cdot 10^{-13} \\
 0.00000287744  & 0.03027233689  & 0.00000272492  & -0.00006309436  & -0.00041014040  & 0.00005613296 \\
 2.98104624843\cdot 10^{-15}  & -0.00001418074  & 0.00002869680  & -1.21506679426\cdot 10^{-12}  & -0.00001423998  & -0.00002812073 \\
 -0.00001908529  & 0.00009471142  & -1.48184820399\cdot 10^{-14}  & 0.00002355898  & -0.00006730898  & -6.97673484952\cdot 10^{-15} \\
 -0.00000297094  & 0.00066179035  & 0.00000282037  & 0.00003503703  & 0.02375644464  & 0.00003236439 \\
 1.50361388430\cdot 10^{-14}  & -0.00008862582  & 0.00001807875  & -7.01151105584\cdot 10^{-16}  & -0.00006331650  & 0.00002238336 \\
 -0.00004067472  & 0.00000301020  & -3.70438406451\cdot 10^{-15}  & -0.00001108514  & -0.00000301927  & 3.33415541431\cdot 10^{-14} \\
 0.00000301020  & -0.06718915938  & 0.00000286072  & -0.00002243646  & 0.00050351668  & 0.00002096805 \\
 -3.70438406451\cdot 10^{-15}  & 0.00000286072  & -0.00003855497  & -4.36948352896\cdot 10^{-14}  & 0.00000287165  & 0.00001047733 \\
 -0.00001108514  & -0.00002243646  & -4.36948352896\cdot 10^{-14}  & -0.00002712959  & 0.00001542937  & 5.10048792691\cdot 10^{-14} \\
 -0.00000301927  & 0.00050351668  & 0.00000287165  & 0.00001542937  & -0.05364315620  & 0.00001449165 \\
 3.33415541431\cdot 10^{-14}  & 0.00002096805  & 0.00001047733  & 5.10048792691\cdot 10^{-14}  & 0.00001449165  & -0.00002568508 \\
 0.00001920997  & 0.00000111905  & -2.13045262845\cdot 10^{-15}  & -0.00000572563  & -0.00000131717  & -1.45197814251\cdot 10^{-14} \\
 0.00000301060  & 0.01951726310  & 0.00000286489  & -0.00001103849  & 0.00040280165  & 0.00001039344 \\
 2.27927040204\cdot 10^{-14}  & 0.00000107185  & 0.00001825080  & 3.89829629231\cdot 10^{-14}  & 0.00000126051  & 0.00000538519 \\
 0.00001662305  & 0.00000597029  & 1.21914579746\cdot 10^{-14}  & 0.00001614121  & -0.00000498558  & -7.83391499939\cdot 10^{-14} \\
 -0.00000299111  & -0.00027523003  & 0.00000284750  & 0.00000806555  & 0.01654806391  & 0.00000760366 \\
 1.26943093273\cdot 10^{-14}  & -0.00000562167  & -0.00001578051  & -3.69080589571\cdot 10^{-14}  & -0.00000471758  & 0.00001533483 \\
\end{array}
\right)
\end{align*}
\begin{align*}
 T^{3}_{fin,4} & =
\left(
\begin{array}{cccccc}
 0.00000125611  & -0.00084878958  & -2.73010821736\cdot 10^{-9}  & 0.00001432963  & 0.00066557241  & -0.00000279787 \\
 8.65982848231\cdot 10^{-8}  & 0.00138998159  & 1.02554282905\cdot 10^{-7}  & 0.00008343141  & -0.00002363432  & -0.00008889372 \\
 -3.19226015591\cdot 10^{-9}  & -0.00080678323  & 0.00000119690  & 0.00000227980  & -0.00063256170  & -0.00001367678 \\
 -0.00000437207  & -0.00027291240  & -1.62202100706\cdot 10^{-12}  & 0.00000129772  & 0.00022255360  & -6.02460704593\cdot 10^{-10} \\
 -0.00000039251  & -0.00008220557  & 0.00000038596  & -0.00002784648  & 0.00154868510  & -0.00002311176 \\
 1.95786950003\cdot 10^{-12}  & 0.00025788023  & 0.00000415121  & -7.87753818442\cdot 10^{-10}  & 0.00021062293  & 0.00000123395 \\
 0.00000406644  & -0.00001493806  & -7.69293477455\cdot 10^{-14}  & -0.00000252131  & 0.00001478127  & -4.37312250174\cdot 10^{-14} \\
 0.00000066295  & 0.00436069767  & 0.00000063871  & 0.00001241650  & -0.00007108881  & -0.00001114858 \\
 -5.79173351464\cdot 10^{-14}  & -0.00001417756  & 0.00000386310  & 7.34844588938\cdot 10^{-14}  & -0.00001403083  & 0.00000239297 \\
 0.00002138146  & 0.00005029006  & 5.43227769034\cdot 10^{-14}  & 0.00000330787  & -0.00003888592  & 1.14001332420\cdot 10^{-13} \\
 -0.00000090330  & -0.00032965875  & 0.00000086668  & -0.00000788818  & 0.00368849025  & -0.00000734374 \\
 -5.68804673351\cdot 10^{-14}  & -0.00004744310  & -0.00002030190  & 8.17156504431\cdot 10^{-14}  & -0.00003674960  & 0.00000314218 \\
 0.00001920997  & 0.00000301060  & 2.27927040204\cdot 10^{-14}  & 0.00001662305  & -0.00000299111  & 1.26943093273\cdot 10^{-14} \\
 0.00000111905  & 0.01951726310  & 0.00000107185  & 0.00000597029  & -0.00027523003  & -0.00000562167 \\
 -2.13045262845\cdot 10^{-15}  & 0.00000286489  & 0.00001825080  & 1.21914579746\cdot 10^{-14}  & 0.00000284750  & -0.00001578051 \\
 -0.00000572563  & -0.00001103849  & 3.89829629231\cdot 10^{-14}  & 0.00001614121  & 0.00000806555  & -3.69080589571\cdot 10^{-14} \\
 -0.00000131717  & 0.00040280165  & 0.00000126051  & -0.00000498558  & 0.01654806391  & -0.00000471758 \\
 -1.45197814251\cdot 10^{-14}  & 0.00001039344  & 0.00000538519  & -7.83391499939\cdot 10^{-14}  & 0.00000760366  & 0.00001533483 \\
 -0.00001787015  & 0.00000149962  & 1.79534093854\cdot 10^{-14}  & -0.00000204184  & -0.00000166924  & -2.11654961276\cdot 10^{-13} \\
 0.00000149962  & -0.04435474412  & 0.00000143467  & 0.00000442712  & 0.00033271878  & -0.00000420077 \\
 1.79534093854\cdot 10^{-14}  & 0.00000143467  & -0.00001688740  & 2.11077743505\cdot 10^{-13}  & 0.00000159584  & 0.00000188523 \\
 -0.00000204184  & 0.00000442712  & 2.11077743505\cdot 10^{-13}  & -0.00001114387  & -0.00000409665  & -2.57393155040\cdot 10^{-13} \\
 -0.00000166924  & 0.00033271878  & 0.00000159584  & -0.00000409665  & -0.03758247246  & -0.00000389363 \\
 -2.11654961276\cdot 10^{-13}  & -0.00000420077  & 0.00000188523  & -2.57393155040\cdot 10^{-13}  & -0.00000389363  & -0.00001049667 \\
\end{array}
\right)
\end{align*}

\begin{align*}
 T_{fin,1}^{6} & =
\left(
\begin{array}{cccccc}
 -0.00291345537  & -0.02540749127  & -0.00041941419  & -0.00121786943  & 0.00933076659  & 0.00018206274 \\
-0.02540749127  & -0.90497888764  & -0.02424571525  & -0.00635382205  & 0.00563881494  & 0.00674636418 \\
-0.00041941419  & -0.02424571525  & -0.00278371098  & -0.00014838001  & -0.00907915550  & 0.00116335903 \\
-0.00121786943  & -0.00635382205  & -0.00014838001  & -0.00021204206  & 0.00262575863  & 0.00000013645 \\
0.00933076659  & 0.00563881494  & -0.00907915550  & 0.00262575863  & -0.23685105510  & 0.00230342523 \\
0.00018206274  & 0.00674636418  & 0.00116335903  & 0.00000013645  & 0.00230342523  & -0.00020165240 \\
0.00006084737  & -0.00000978359  & -4.86286767418\cdot 10^{-8}  & -0.00005068073  & 0.00001298540  & 4.66620681740\cdot 10^{-11} \\
-0.00403278487  & 0.06006461031  & -0.00384880693  & -0.00114914679  & 0.00170420893  & 0.00102804953 \\
-4.61137477934\cdot 10^{-8}  & -0.00000913619  & 0.00005786179  & -4.64964261953\cdot 10^{-11}  & -0.00001219898  & 0.00004810665 \\
0.00031255555  & 0.00148247120  & 0.00003482935  & 0.00004143807  & -0.00060335343  & -1.67484077527\cdot 10^{-8} \\
0.00233877011  & -0.00072279657  & -0.00222565755  & 0.00067668948  & 0.04134592791  & 0.00062737193 \\
-0.00004274417  & -0.00157496654  & -0.00029841490  & -1.52922857114\cdot 10^{-8}  & -0.00052815010  & 0.00003940615 \\
0.00000533496  & 0.00000205063  & 1.08637530482\cdot 10^{-8}  & 0.00003985325  & -0.00000267139  & -6.88851803831\cdot 10^{-12} \\
-0.00155761678  & 0.00542592858  & -0.00148129586  & -0.00046479549  & -0.00053855101  & 0.00043588952 \\
1.03004567414\cdot 10^{-8}  & 0.00000191475  & 0.00000505715  & 5.72582288790\cdot 10^{-12}  & 0.00000251061  & -0.00003785447 \\
-0.00007068683  & -0.00037123808  & -0.00000873218  & 0.00000498890  & 0.00015051244  & 4.19903510403\cdot 10^{-9} \\
0.00112127247  & -0.00008204165  & -0.00106597795  & 0.00034694746  & 0.00518713748  & 0.00032698004 \\
0.00001071653  & 0.00039444690  & 0.00006752409  & 3.83396386458\cdot 10^{-9}  & 0.00013169100  & 0.00000473141 \\
0.00000125628  & 9.07808750600\cdot 10^{-8}  & -2.99529358660\cdot 10^{-9}  & -0.00000437207  & -0.00000039232  & 1.97070264514\cdot 10^{-12} \\
-0.00084878958  & 0.00138995868  & -0.00080678324  & -0.00027291241  & -0.00008219826  & 0.00025788023 \\
-2.83996989618\cdot 10^{-9}  & 9.96003237305\cdot 10^{-8}  & 0.00000119675  & -1.65191684991\cdot 10^{-12}  & 0.00000038570  & 0.00000415121 \\
0.00001343075  & 0.00006447176  & 0.00000149956  & 0.00000129779  & -0.00002707375  & -7.21116613265\cdot 10^{-10} \\
0.00066557241  & -0.00002363043  & -0.00063256170  & 0.00022255360  & 0.00154868413  & 0.00021062293 \\
-0.00000184033  & -0.00006842433  & -0.00001282276  & -6.58453704807\cdot 10^{-10}  & -0.00002378631  & 0.00000123389 \\
\end{array}
\right)
\end{align*}

\begin{align*}
 T_{fin,2}^{6} & =
\left(
\begin{array}{cccccc}
0.00006084737  & -0.00403278487  & -4.61137477934\cdot 10^{-8}  & 0.00031255555  & 0.00233877011  & -0.00004274417 \\
-0.00000978359  & 0.06006461031  & -0.00000913619  & 0.00148247120  & -0.00072279657  & -0.00157496654 \\
-4.86286767418\cdot 10^{-8}  & -0.00384880693  & 0.00005786179  & 0.00003482935  & -0.00222565755  & -0.00029841490 \\
-0.00005068073  & -0.00114914679  & -4.64964261953\cdot 10^{-11}  & 0.00004143807  & 0.00067668948  & -1.52922857114\cdot 10^{-8} \\
0.00001298540  & 0.00170420893  & -0.00001219898  & -0.00060335343  & 0.04134592791  & -0.00052815010 \\
4.66620681740\cdot 10^{-11}  & 0.00102804953  & 0.00004810665  & -1.67484077527\cdot 10^{-8}  & 0.00062737193  & 0.00003940615 \\
-0.00010321148  & -0.00001414600  & -2.73005911389\cdot 10^{-14}  & -0.00003848684  & 0.00001470528  & -3.71351392015\cdot 10^{-12} \\
-0.00001414600  & -0.12947857817  & -0.00001337906  & 0.00025555125  & 0.00095215812  & -0.00022776080 \\
-2.73005911389\cdot 10^{-14}  & -0.00001337906  & -0.00009797657  & 4.89450772014\cdot 10^{-12}  & -0.00001393580  & 0.00003651136 \\
-0.00003848684  & 0.00025555125  & 4.89450772014\cdot 10^{-12}  & -0.00006239992  & -0.00014436957  & -4.92646031077\cdot 10^{-15} \\
0.00001470528  & 0.00095215812  & -0.00001393580  & -0.00014436957  & -0.08886148567  & -0.00013348148 \\
-3.71351392015\cdot 10^{-12}  & -0.00022776080  & 0.00003651136  & -4.92646031077\cdot 10^{-15}  & -0.00013348148  & -0.00005919817 \\
0.00003020315  & 0.00000287750  & -6.11527496408\cdot 10^{-16}  & -0.00001908529  & -0.00000297095  & 1.88925224729\cdot 10^{-14} \\
-0.00001495138  & 0.03027233185  & -0.00001418074  & 0.00009471169  & 0.00066178980  & -0.00008862601 \\
-3.31589309321\cdot 10^{-16}  & 0.00000272488  & 0.00002869680  & -1.82627377477\cdot 10^{-14}  & 0.00000282036  & 0.00001807875 \\
0.00002961061  & -0.00006309364  & -1.20807810807\cdot 10^{-12}  & 0.00002355898  & 0.00003504019  & 1.06789333275\cdot 10^{-14} \\
0.00001500747  & -0.00041014029  & -0.00001423998  & -0.00006730898  & 0.02375644411  & -0.00006331651 \\
9.12420912576\cdot 10^{-13}  & 0.00005616901  & -0.00002812073  & 1.93046228742\cdot 10^{-15}  & 0.00003236094  & 0.00002238336 \\
0.00000406644  & 0.00000066293  & -6.41351673174\cdot 10^{-14}  & 0.00002138146  & -0.00000090328  & -5.54519037512\cdot 10^{-14} \\
-0.00001493806  & 0.00436069008  & -0.00001417757  & 0.00005029007  & -0.00032965896  & -0.00004744310 \\
-8.14811085733\cdot 10^{-14}  & 0.00000063872  & 0.00000386310  & 5.45200014040\cdot 10^{-14}  & 0.00000086669  & -0.00002030190 \\
-0.00000252131  & 0.00001241640  & 7.05782109673\cdot 10^{-14}  & 0.00000330787  & -0.00000788873  & 8.42888401493\cdot 10^{-14} \\
0.00001478127  & -0.00007109159  & -0.00001403083  & -0.00003888592  & 0.00368849050  & -0.00003674960 \\
-4.67295111503\cdot 10^{-14}  & -0.00001115478  & 0.00000239297  & 1.14232496671\cdot 10^{-13}  & -0.00000734314  & 0.00000314218 \\
\end{array}
\right)
\end{align*}

\begin{align*}
 T_{fin,3}^{6} & =
\left(
\begin{array}{cccccc}
 0.00000533496  & -0.00155761678  & 1.03004567414\cdot 10^{-8}  & -0.00007068683  & 0.00112127247  & 0.00001071653 \\
0.00000205063  & 0.00542592858  & 0.00000191475  & -0.00037123808  & -0.00008204165  & 0.00039444690 \\
1.08637530482\cdot 10^{-8}  & -0.00148129586  & 0.00000505715  & -0.00000873218  & -0.00106597795  & 0.00006752409 \\
0.00003985325  & -0.00046479549  & 5.72582288790\cdot 10^{-12}  & 0.00000498890  & 0.00034694746  & 3.83396386458\cdot 10^{-9} \\
-0.00000267139  & -0.00053855101  & 0.00000251061  & 0.00015051244  & 0.00518713748  & 0.00013169100 \\
-6.88851803831\cdot 10^{-12}  & 0.00043588952  & -0.00003785447  & 4.19903510403\cdot 10^{-9}  & 0.00032698004  & 0.00000473141 \\
0.00003020315  & -0.00001495138  & -3.31589309321\cdot 10^{-16}  & 0.00002961061  & 0.00001500747  & 9.12420912576\cdot 10^{-13} \\
0.00000287750  & 0.03027233185  & 0.00000272488  & -0.00006309364  & -0.00041014029  & 0.00005616901 \\
-6.11527496408\cdot 10^{-16}  & -0.00001418074  & 0.00002869680  & -1.20807810807\cdot 10^{-12}  & -0.00001423998  & -0.00002812073 \\
-0.00001908529  & 0.00009471169  & -1.82627377477\cdot 10^{-14}  & 0.00002355898  & -0.00006730898  & 1.93046228742\cdot 10^{-15} \\
-0.00000297095  & 0.00066178980  & 0.00000282036  & 0.00003504019  & 0.02375644411  & 0.00003236094 \\
1.88925224729\cdot 10^{-14}  & -0.00008862601  & 0.00001807875  & 1.06789333275\cdot 10^{-14}  & -0.00006331651  & 0.00002238336 \\
-0.00004067472  & 0.00000301019  & -9.94013687038\cdot 10^{-15}  & -0.00001108514  & -0.00000301927  & 3.09891987571\cdot 10^{-14} \\
0.00000301019  & -0.06718916229  & 0.00000286072  & -0.00002243655  & 0.00050351656  & 0.00002096811 \\
-9.94013687038\cdot 10^{-15}  & 0.00000286072  & -0.00003855497  & -3.90779130797\cdot 10^{-14}  & 0.00000287165  & 0.00001047733 \\
-0.00001108514  & -0.00002243655  & -3.90779130797\cdot 10^{-14}  & -0.00002712959  & 0.00001542937  & 4.95308748195\cdot 10^{-14} \\
-0.00000301927  & 0.00050351656  & 0.00000287165  & 0.00001542937  & -0.05364315613  & 0.00001449165 \\
3.09891987571\cdot 10^{-14}  & 0.00002096811  & 0.00001047733  & 4.95308748195\cdot 10^{-14}  & 0.00001449165  & -0.00002568508 \\
0.00001920997  & 0.00000111906  & 2.94183176863\cdot 10^{-15}  & -0.00000572563  & -0.00000131717  & -1.48589541887\cdot 10^{-14} \\
0.00000301060  & 0.01951725986  & 0.00000286489  & -0.00001103849  & 0.00040280094  & 0.00001039344 \\
2.46715880694\cdot 10^{-14}  & 0.00000107186  & 0.00001825080  & 4.12671067396\cdot 10^{-14}  & 0.00000126051  & 0.00000538519 \\
0.00001662305  & 0.00000597031  & 6.81443231533\cdot 10^{-15}  & 0.00001614121  & -0.00000498558  & -8.56313692173\cdot 10^{-14} \\
-0.00000299111  & -0.00027522999  & 0.00000284750  & 0.00000806555  & 0.01654806366  & 0.00000760366 \\
1.62019892263\cdot 10^{-14}  & -0.00000562168  & -0.00001578051  & -4.16453813480\cdot 10^{-14}  & -0.00000471758  & 0.00001533483 \\
\end{array}
\right)
\end{align*}

\begin{align*}
 T_{fin,4}^{6} & =
\left(
\begin{array}{cccccc}
 0.00000125628  & -0.00084878958  & -2.83996989618\cdot 10^{-9}  & 0.00001343075  & 0.00066557241  & -0.00000184033 \\
9.07808750600\cdot 10^{-8}  & 0.00138995868  & 9.96003237305\cdot 10^{-8}  & 0.00006447176  & -0.00002363043  & -0.00006842433 \\
-2.99529358660\cdot 10^{-9}  & -0.00080678324  & 0.00000119675  & 0.00000149956  & -0.00063256170  & -0.00001282276 \\
-0.00000437207  & -0.00027291241  & -1.65191684991\cdot 10^{-12}  & 0.00000129779  & 0.00022255360  & -6.58453704807\cdot 10^{-10} \\
-0.00000039232  & -0.00008219826  & 0.00000038570  & -0.00002707375  & 0.00154868413  & -0.00002378631 \\
1.97070264514\cdot 10^{-12}  & 0.00025788023  & 0.00000415121  & -7.21116613265\cdot 10^{-10}  & 0.00021062293  & 0.00000123389 \\
0.00000406644  & -0.00001493806  & -8.14811085733\cdot 10^{-14}  & -0.00000252131  & 0.00001478127  & -4.67295111503\cdot 10^{-14} \\
0.00000066293  & 0.00436069008  & 0.00000063872  & 0.00001241640  & -0.00007109159  & -0.00001115478 \\
-6.41351673174\cdot 10^{-14}  & -0.00001417757  & 0.00000386310  & 7.05782109673\cdot 10^{-14}  & -0.00001403083  & 0.00000239297 \\
0.00002138146  & 0.00005029007  & 5.45200014040\cdot 10^{-14}  & 0.00000330787  & -0.00003888592  & 1.14232496671\cdot 10^{-13} \\
-0.00000090328  & -0.00032965896  & 0.00000086669  & -0.00000788873  & 0.00368849050  & -0.00000734314 \\
-5.54519037512\cdot 10^{-14}  & -0.00004744310  & -0.00002030190  & 8.42888401493\cdot 10^{-14}  & -0.00003674960  & 0.00000314218 \\
0.00001920997  & 0.00000301060  & 2.46715880694\cdot 10^{-14}  & 0.00001662305  & -0.00000299111  & 1.62019892263\cdot 10^{-14} \\
0.00000111906  & 0.01951725986  & 0.00000107186  & 0.00000597031  & -0.00027522999  & -0.00000562168 \\
2.94183176863\cdot 10^{-15}  & 0.00000286489  & 0.00001825080  & 6.81443231533\cdot 10^{-15}  & 0.00000284750  & -0.00001578051 \\
-0.00000572563  & -0.00001103849  & 4.12671067396\cdot 10^{-14}  & 0.00001614121  & 0.00000806555  & -4.16453813480\cdot 10^{-14} \\
-0.00000131717  & 0.00040280094  & 0.00000126051  & -0.00000498558  & 0.01654806366  & -0.00000471758 \\
-1.48589541887\cdot 10^{-14}  & 0.00001039344  & 0.00000538519  & -8.56313692173\cdot 10^{-14}  & 0.00000760366  & 0.00001533483 \\
-0.00001787015  & 0.00000149962  & 2.77623635774\cdot 10^{-14}  & -0.00000204184  & -0.00000166924  & -1.91830702198\cdot 10^{-13} \\
0.00000149962  & -0.04435472907  & 0.00000143468  & 0.00000442712  & 0.00033271917  & -0.00000420077 \\
2.77623635774\cdot 10^{-14}  & 0.00000143468  & -0.00001688740  & 2.06623856601\cdot 10^{-13}  & 0.00000159584  & 0.00000188523 \\
-0.00000204184  & 0.00000442712  & 2.06623856601\cdot 10^{-13}  & -0.00001114387  & -0.00000409665  & -2.62477024413\cdot 10^{-13} \\
-0.00000166924  & 0.00033271917  & 0.00000159584  & -0.00000409665  & -0.03758247285  & -0.00000389363 \\
-1.91830702198\cdot 10^{-13}  & -0.00000420077  & 0.00000188523  & -2.62477024413\cdot 10^{-13}  & -0.00000389363  & -0.00001049667 \\
\end{array}
\right)
\end{align*}

\normalsize

 \bibliographystyle{abbrv}
 \bibliography{references}

\def\cprime{$'$}
\begin{thebibliography}{10}

\bibitem{Azzam-Bedrossian:bmo-uniqueness-active-scalar-equations}
J.~Azzam and J.~Bedrossian.
\newblock Bounded mean oscillation and the uniqueness of active scalar
  equations.
\newblock {\em Trans. Amer. Math. Soc.}, 367(5):3095--3118, 2015.

\bibitem{Berz-Makino:high-dimensional-quadrature}
M.~Berz and K.~Makino.
\newblock New methods for high-dimensional verified quadrature.
\newblock {\em Reliable Computing}, 5(1):13--22, 1999.

\bibitem{Burbea:motions-vortex-patches}
J.~Burbea.
\newblock Motions of vortex patches.
\newblock {\em Lett. Math. Phys.}, 6(1):1--16, 1982.

\bibitem{Cannone-Xue:self-similar-solutions-sqg}
M.~Cannone and L.~Xue.
\newblock Remarks on self-similar solutions for the surface quasi-geostrophic
  equation and its generalization.
\newblock {\em Proc. Amer. Math. Soc.}, 143(6):2613--2622, 2015.

\bibitem{Castro-Cordoba:infinite-energy-sqg}
A.~Castro and D.~C{\'o}rdoba.
\newblock Infinite energy solutions of the surface quasi-geostrophic equation.
\newblock {\em Adv. Math.}, 225(4):1820--1829, 2010.

\bibitem{Castro-Cordoba-GomezSerrano:existence-regularity-vstates-gsqg}
A.~Castro, D.~C{\'o}rdoba, and J.~G{\'o}mez-Serrano.
\newblock Existence and regularity of rotating global solutions for the
  generalized surface quasi-geostrophic equations.
\newblock {\em Duke Math. J.}, 165(5):935--984, 2016.

\bibitem{Castro-Cordoba-GomezSerrano:analytic-vstates-ellipses}
A.~Castro, D.~C{\'o}rdoba, and J.~G{\'o}mez-Serrano.
\newblock Uniformly rotating analytic global patch solutions for active
  scalars.
\newblock {\em Annals of PDE}, 2(1):1--34, 2016.

\bibitem{Castro-Cordoba-GomezSerrano:uniformly-rotating-smooth-euler}
A.~Castro, D.~C\'{o}rdoba, and J.~G\'{o}mez-Serrano.
\newblock Uniformly rotating smooth solutions for the incompressible 2{D}
  {E}uler equations.
\newblock {\em Arch. Ration. Mech. Anal.}, 231(2):719--785, 2019.

\bibitem{Chae:qg-equation-triebel-lizorkin}
D.~Chae.
\newblock The quasi-geostrophic equation in the {T}riebel-{L}izorkin spaces.
\newblock {\em Nonlinearity}, 16(2):479--495, 2003.

\bibitem{Chae:continuation-principles-euler-sqg}
D.~Chae.
\newblock On the continuation principles for the {E}uler equations and the
  quasi-geostrophic equation.
\newblock {\em J. Differential Equations}, 227(2):640--651, 2006.

\bibitem{Chae:geometric-approaches-singularities-inviscid-fluid-flows}
D.~Chae.
\newblock The geometric approaches to the possible singularities in the
  inviscid fluid flows.
\newblock {\em J. Phys. A}, 41(36):365501, 11, 2008.

\bibitem{Chae:behavior-solution-euler-related}
D.~Chae.
\newblock On the behaviors of solutions near possible blow-up time in the
  incompressible {E}uler and related equations.
\newblock {\em Comm. Partial Differential Equations}, 34(10-12):1269--1286,
  2009.

\bibitem{Chae-Constantin-Wu:deformation-symmetry-inviscid-sqg-3d-euler}
D.~Chae, P.~Constantin, and J.~Wu.
\newblock Deformation and symmetry in the inviscid {SQG} and the 3{D} {E}uler
  equations.
\newblock {\em J. Nonlinear Sci.}, 22(5):665--688, 2012.

\bibitem{Constantin-Lai-Sharma-Tseng-Wu:new-numerics-sqg}
P.~Constantin, M.-C. Lai, R.~Sharma, Y.-H. Tseng, and J.~Wu.
\newblock New numerical results for the surface quasi-geostrophic equation.
\newblock {\em J. Sci. Comput.}, 50(1):1--28, 2012.

\bibitem{Constantin-Majda-Tabak:formation-fronts-qg}
P.~Constantin, A.~J. Majda, and E.~Tabak.
\newblock Formation of strong fronts in the {$2$}-{D} quasigeostrophic thermal
  active scalar.
\newblock {\em Nonlinearity}, 7(6):1495--1533, 1994.

\bibitem{Constantin-Nie-Schorghofer:nonsingular-sqg-flow}
P.~Constantin, Q.~Nie, and N.~Sch{\"o}rghofer.
\newblock Nonsingular surface quasi-geostrophic flow.
\newblock {\em Phys. Lett. A}, 241(3):168--172, 1998.

\bibitem{Cordoba:nonexistence-hyperbolic-blowup-qg}
D.~Cordoba.
\newblock Nonexistence of simple hyperbolic blow-up for the quasi-geostrophic
  equation.
\newblock {\em Ann. of Math. (2)}, 148(3):1135--1152, 1998.

\bibitem{Cordoba-Fefferman:growth-solutions-qg-2d-euler}
D.~Cordoba and C.~Fefferman.
\newblock Growth of solutions for {QG} and 2{D} {E}uler equations.
\newblock {\em J. Amer. Math. Soc.}, 15(3):665--670, 2002.

\bibitem{Cordoba-Fefferman:scalars-convected-2d-incompressible-flow}
D.~Cordoba and C.~Fefferman.
\newblock Scalars convected by a two-dimensional incompressible flow.
\newblock {\em Comm. Pure Appl. Math.}, 55(2):255--260, 2002.

\bibitem{Cordoba-Fefferman-Rodrigo:almost-sharp-fronts-sqg}
D.~C{\'o}rdoba, C.~Fefferman, and J.~L. Rodrigo.
\newblock Almost sharp fronts for the surface quasi-geostrophic equation.
\newblock {\em Proc. Natl. Acad. Sci. USA}, 101(9):2687--2691, 2004.

\bibitem{Cordoba-Fontelos-Mancho-Rodrigo:evidence-singularities-contour-dynamics}
D.~C{\'o}rdoba, M.~A. Fontelos, A.~M. Mancho, and J.~L. Rodrigo.
\newblock Evidence of singularities for a family of contour dynamics equations.
\newblock {\em Proc. Natl. Acad. Sci. USA}, 102(17):5949--5952, 2005.

\bibitem{Cordoba-GomezSerrano-Zlatos:stability-shifting-muskat-II}
D.~C\'ordoba, J.~G\'omez-Serrano, and A.~Zlato{\v s}.
\newblock A note on stability shifting for the {M}uskat problem, {II}: {F}rom
  stable to unstable and back to stable.
\newblock {\em Anal. PDE}, 10(2):367--378, 2017.

\bibitem{Crandall-Rabinowitz:bifurcation-simple-eigenvalues}
M.~G. Crandall and P.~H. Rabinowitz.
\newblock Bifurcation from simple eigenvalues.
\newblock {\em J. Functional Analysis}, 8:321--340, 1971.

\bibitem{Deng-Hou-Li-Yu:non-blowup-2d-sqg}
J.~Deng, T.~Y. Hou, R.~Li, and X.~Yu.
\newblock Level set dynamics and the non-blowup of the 2{D} quasi-geostrophic
  equation.
\newblock {\em Methods Appl. Anal.}, 13(2):157--180, 2006.

\bibitem{Dritschel:exact-rotating-solution-sqg}
D.~G. Dritschel.
\newblock An exact steadily rotating surface quasi-geostrophic elliptical
  vortex.
\newblock {\em Geophys. Astrophys. Fluid Dyn.}, 105(4-5):368--376, 2011.

\bibitem{Fefferman-Luli-Rodrigo:spine-sqg-almost-sharp-front}
C.~Fefferman, G.~Luli, and J.~Rodrigo.
\newblock The spine of an {SQG} almost-sharp front.
\newblock {\em Nonlinearity}, 25(2):329--342, 2012.

\bibitem{Fefferman-Rodrigo:almost-sharp-fronts-sqg}
C.~Fefferman and J.~L. Rodrigo.
\newblock Almost sharp fronts for {SQG}: the limit equations.
\newblock {\em Comm. Math. Phys.}, 313(1):131--153, 2012.

\bibitem{Friedlander-Shvydkoy:unstable-spectrum-sqg}
S.~Friedlander and R.~Shvydkoy.
\newblock The unstable spectrum of the surface quasi-geostropic equation.
\newblock {\em J. Math. Fluid Mech.}, 7(suppl. 1):S81--S93, 2005.

\bibitem{Gancedo:existence-alpha-patch-sobolev}
F.~Gancedo.
\newblock Existence for the {$\alpha$}-patch model and the {QG} sharp front in
  {S}obolev spaces.
\newblock {\em Adv. Math.}, 217(6):2569--2598, 2008.

\bibitem{Gancedo-Strain:absence-splash-muskat-SQG}
F.~{Gancedo} and R.~M. {Strain}.
\newblock Absence of splash singularities for surface quasi-geostrophic sharp
  fronts and the {M}uskat problem.
\newblock {\em Proceedings of the National Academy of Sciences},
  111(2):635--639, 2014.

\bibitem{Gerschgorin:eigenvalues-theorem}
S.~A. Gershgorin.
\newblock {\"U}ber die {A}bgrenzung der {E}igenwerte einer {M}atrix.
\newblock {\em Bulletin de l'Acad\'emie des Sciences de l'URSS. Classe des
  sciences math\'ematiques et naturelles}, (6):749--754, 1931.

\bibitem{Hassainia-Hmidi:v-states-generalized-sqg}
Z.~Hassainia and T.~Hmidi.
\newblock On the {V}-states for the generalized quasi-geostrophic equations.
\newblock {\em Comm. Math. Phys.}, 337(1):321--377, 2015.

\bibitem{Held-Pierrehumbert-Garner-Swanson:sqg-dynamics}
I.~M. Held, R.~T. Pierrehumbert, S.~T. Garner, and K.~L. Swanson.
\newblock Surface quasi-geostrophic dynamics.
\newblock {\em J. Fluid Mech.}, 282:1--20, 1995.

\bibitem{Hmidi-Mateu-Verdera:rotating-vortex-patch}
T.~Hmidi, J.~Mateu, and J.~Verdera.
\newblock Boundary regularity of rotating vortex patches.
\newblock {\em Archive for Rational Mechanics and Analysis}, 209(1):171--208,
  2013.

\bibitem{CXSC}
W.~Hofschuster and W.~Kr{\"a}mer.
\newblock C-{X}{S}{C} 2.0--a {C}++ library for e{X}tended {S}cientific
  {C}omputing.
\newblock In {\em Numerical software with result verification}, pages 15--35.
  Springer, 2004.

\bibitem{Hou-Shi:dynamic-growth-vorticity-3d-euler-sqg}
T.~Y. Hou and Z.~Shi.
\newblock Dynamic growth estimates of maximum vorticity for 3{D} incompressible
  {E}uler equations and the {SQG} model.
\newblock {\em Discrete Contin. Dyn. Syst.}, 32(5):1449--1463, 2012.

\bibitem{Isett-Vicol:holder-continuous-active-scalar}
P.~Isett and V.~Vicol.
\newblock H{\"o}lder continuous solutions of active scalar equations.
\newblock {\em Annals of PDE}, 1(1):1--77, 2015.

\bibitem{Ju:geometric-constraints-global-regularity-sqg}
N.~Ju.
\newblock Geometric constrains for global regularity of 2{D} quasi-geostrophic
  flows.
\newblock {\em J. Differential Equations}, 226(1):54--79, 2006.

\bibitem{Kiselev-Nazarov:simple-energy-pump-sqg}
A.~Kiselev and F.~Nazarov.
\newblock A simple energy pump for the surface quasi-geostrophic equation.
\newblock In H.~Holden and K.~H. Karlsen, editors, {\em Nonlinear Partial
  Differential Equations}, volume~7 of {\em Abel Symposia}, pages 175--179.
  Springer Berlin Heidelberg, 2012.

\bibitem{Kramer-Wedner:adaptive-gauss-legendre-verified-computation}
W.~Kr\"amer and S.~Wedner.
\newblock Two adaptive {G}auss-{L}egendre type algorithms for the verified
  computation of definite integrals.
\newblock {\em Reliable Computing}, 2(3):241--253, 1996.

\bibitem{Lang:multidimensional-verified-gaussian-quadrature}
B.~Lang.
\newblock Derivative-based subdivision in multi-dimensional verified gaussian
  quadrature.
\newblock In G.~Alefeld, J.~Rohn, S.~Rump, and T.~Yamamoto, editors, {\em
  Symbolic Algebraic Methods and Verification Methods}, pages 145--152.
  Springer Vienna, 2001.

\bibitem{Lannes:water-waves-book}
D.~Lannes.
\newblock {\em The Water Waves Problem: Mathematical Analysis and Asymptotics}.
\newblock Mathematical Surveys and Monographs. Amer Mathematical Society, 2013.

\bibitem{Li:existence-theorems-2d-sqg-plane-waves}
D.~Li.
\newblock Existence theorems for the 2{D} quasi-geostrophic equation with plane
  wave initial conditions.
\newblock {\em Nonlinearity}, 22(7):1639--1651, 2009.

\bibitem{Majda-Bertozzi:vorticity-incompressible-flow}
A.~J. Majda and A.~L. Bertozzi.
\newblock {\em Vorticity and incompressible flow}, volume~27 of {\em Cambridge
  Texts in Applied Mathematics}.
\newblock Cambridge University Press, Cambridge, 2002.

\bibitem{Majda-Tabak:2d-model-sqg}
A.~J. Majda and E.~G. Tabak.
\newblock A two-dimensional model for quasigeostrophic flow: comparison with
  the two-dimensional {E}uler flow.
\newblock {\em Phys. D}, 98(2-4):515--522, 1996.
\newblock Nonlinear phenomena in ocean dynamics (Los Alamos, NM, 1995).

\bibitem{Marchand:existence-regularity-weak-solutions-sqg}
F.~Marchand.
\newblock Existence and regularity of weak solutions to the quasi-geostrophic
  equations in the spaces {$L^p$} or {$\dot H^{-1/2}$}.
\newblock {\em Comm. Math. Phys.}, 277(1):45--67, 2008.

\bibitem{Moore-Bierbaum:methods-applications-interval-analysis}
R.~Moore and F.~Bierbaum.
\newblock {\em Methods and applications of interval analysis}, volume~2.
\newblock Society for Industrial \& Applied Mathematics, 1979.

\bibitem{Ohkitani-Yamada:inviscid-limit-sqg}
K.~Ohkitani and M.~Yamada.
\newblock Inviscid and inviscid-limit behavior of a surface quasigeostrophic
  flow.
\newblock {\em Phys. Fluids}, 9(4):876--882, 1997.

\bibitem{Pedlosky:geophysical}
J.~Pedlosky.
\newblock Geophysical fluid dynamics.
\newblock {\em New York and Berlin, Springer-Verlag}, 1, 1982.

\bibitem{Resnick:phd-thesis-sqg-chicago}
S.~G. Resnick.
\newblock {\em Dynamical problems in non-linear advective partial differential
  equations}.
\newblock PhD thesis, University of Chicago, Department of Mathematics, 1995.

\bibitem{Rodrigo:evolution-sharp-fronts-qg}
J.~L. Rodrigo.
\newblock On the evolution of sharp fronts for the quasi-geostrophic equation.
\newblock {\em Comm. Pure Appl. Math.}, 58(6):821--866, 2005.

\bibitem{Rusin:logarithmic-spikes-uniqueness-weak-solution-active-scalars}
W.~Rusin.
\newblock Logarithmic spikes of gradients and uniqueness of weak solutions to a
  class of active scalar equations.
\newblock {\em Arxiv preprint arXiv:1106.2778}, 2011.

\bibitem{Scott:scenario-singularity-quasigeostrophic}
R.~K. Scott.
\newblock A scenario for finite-time singularity in the quasigeostrophic model.
\newblock {\em Journal of Fluid Mechanics}, 687:492--502, 11 2011.

\bibitem{Scott-Dritschel:self-similar-sqg}
R.~K. Scott and D.~G. Dritschel.
\newblock Numerical simulation of a self-similar cascade of filament
  instabilities in the surface quasigeostrophic system.
\newblock {\em Phys. Rev. Lett.}, 112:144505, 2014.

\bibitem{Tucker:validated-numerics-book}
W.~Tucker.
\newblock {\em Validated numerics}.
\newblock Princeton University Press, Princeton, NJ, 2011.
\newblock A short introduction to rigorous computations.

\bibitem{Wu:qg-equations-morrey-spaces}
J.~Wu.
\newblock Quasi-geostrophic-type equations with initial data in {M}orrey
  spaces.
\newblock {\em Nonlinearity}, 10(6):1409--1420, 1997.

\bibitem{Wu:solutions-2d-qg-holder}
J.~Wu.
\newblock Solutions of the 2{D} quasi-geostrophic equation in {H}\"older
  spaces.
\newblock {\em Nonlinear Anal.}, 62(4):579--594, 2005.

\end{thebibliography}

\begin{tabular}{l}
\textbf{Angel Castro} \\
  {\small Departamento de Matem\'aticas} \\
 {\small Universidad Aut\'onoma de Madrid} \\
 {\small Instituto de Ciencias Matem\'aticas-CSIC}\\
 {\small C/ Nicolas Cabrera, 13-15, 28049 Madrid, Spain} \\
  {\small Email: angel\_castro@icmat.es} \\
\\
\textbf{Diego C\'ordoba} \\
  {\small Instituto de Ciencias Matem\'aticas} \\
 {\small Consejo Superior de Investigaciones Cient\'ificas} \\
 {\small C/ Nicolas Cabrera, 13-15, 28049 Madrid, Spain} \\
  {\small Email: dcg@icmat.es} \\
\\
\textbf{Javier G\'omez-Serrano} \\
{\small Department of Mathematics} \\
{\small Princeton University}\\
{\small 610 Fine Hall, Washington Rd,}\\
{\small Princeton, NJ 08544, USA}\\
 {\small Email: jg27@math.princeton.edu} \\
  \\

\end{tabular}

\end{document}